\documentclass{amsart}

\usepackage[margin=1in]{geometry}

\usepackage{tikz}
\usetikzlibrary{shapes}

\newcommand{\stkout}[1]{\ifmmode\text{\sout{\ensuremath{#1}}}\else\sout{#1}\fi}

\def\C{\mathbb{C}}

\def\R{\mathbb{R}}
\def\X{\mbf{X}}
\def\x{\mbf{x}}

\def\Y{\mbf{Y}}
\def\y{\mbf{y}}

\def\z{\mbf{z}}

\def\b{\mbf{b}}

\def\m{\mbf{m}}

\usepackage{color}
\usepackage[capitalize,nameinlink,noabbrev,nosort]{cleveref}
\usepackage{graphicx}

\def\defeq{\mathrel{\mathop:}=}

\newcommand{\w}{\widetilde}

\newcommand{\ZZ}{ \mathbb{Z} }
\newcommand{\RR}{ \mathbb{R} }

\newtheorem{thm}{Theorem}[section]
\newtheorem{theorem}[thm]{Theorem}
\newtheorem{lemma}[thm]{Lemma}
\newtheorem{lem}[thm]{Lemma}
\newtheorem{prop}[thm]{Proposition}

\newtheorem{defn}[thm]{Definition}
\newtheorem{dfn}[thm]{Definition}

\newtheorem{conj}[thm]{Conjecture}

\newtheorem{remark}[thm]{Remark}

\newcommand\cell[3]{
\def\i{#1} \def\j{#2} \def\entry{#3}
\draw (\j-1,-\i)--(\j,-\i)--(\j,-\i+1);
\node at (\j-.5,-\i+.5) {\entry};
}

\newcommand\BLHLL[2]{
  \def\X{#1} \def\Y{#2}
  \foreach \i in {0,...,\X}
  {
\pgfmathsetmacro{\m}{\Y-1/(\i+1)};
    \draw[gray,thick] (\i,0) -- (\i,\m);
  }
  \foreach \j in {0,...,\Y}
  {
  }
\pgfmathsetmacro{\m}{\Y-1};
  \foreach \x in {1,...,\X}
{ \foreach \j in {0,...,\m}
  {
\draw[gray,thick] (\x-1,\j) -- (\x,\j);
}}
 \foreach \i in {2,...,\X}
 \foreach \j in {0,...,\m} 
{\foreach \y in {2,...,\i}
{
\pgfmathsetmacro{\w}{\j+(\y-1)/\i};
\pgfmathsetmacro{\z}{\j+(\y-1)/(\i+1)};
\draw[gray,thick] (\i-1,\w) -- (\i,\z);
}}
}

\newcommand\LHLL[2]{
  \def\X{#1} \def\Y{#2}
  \foreach \i in {0,...,\X}
  {
\pgfmathsetmacro{\m}{\Y-1/(\i+1)};
    \draw[gray,thick] (\i,0) -- (\i,\m);
    \node at (\i,-.3) {\i};
  }
  \foreach \j in {0,...,\Y}
  {
    \node at (-.3,\j) {\j};
  }
\pgfmathsetmacro{\m}{\Y-1};
  \foreach \x in {1,...,\X}
{ \foreach \j in {0,...,\m}
  {
\draw[gray,thick] (\x-1,\j) -- (\x,\j);
}}
 \foreach \i in {2,...,\X}
 \foreach \j in {0,...,\m} 
{\foreach \y in {2,...,\i}
{
\pgfmathsetmacro{\w}{\j+(\y-1)/\i};
\pgfmathsetmacro{\z}{\j+(\y-1)/(\i+1)};
\draw[gray,thick] (\i-1,\w) -- (\i,\z);
}}
\node at (\X+1,1) {$\cdots$};
\node at (\X+1,\Y-1) {$\cdots$};
}

\newcommand\DLL[2]{
  \def\X{#1} \def\Y{#2}
  \foreach \i in {0,...,\X}
  {
\pgfmathsetmacro{\m}{\Y-1/(\i+1)};
  }
  \foreach \j in {0,...,\Y}
  {
  }
\pgfmathsetmacro{\m}{\Y-1};
  \foreach \x in {1,...,\X}
{ \foreach \j in {0,...,\m}
  {
\node at (\x,\j){$\circ$};
\node at (\x-.1,\j+.1){$\bullet$};
\draw[gray,thick] (\x-1,\j) -- (\x-.1,\j+.1);
}}
  
 \foreach \j in {0,...,\m}
  {
\node at (0,\j){$\circ$};
\node at (0-.1,\j+.1){$\bullet$};
\draw[gray,thick] (0,\j) -- (-.1,\j+.1);
}

\foreach \i in {1,...,\X}
 \foreach \j in {0,...,\m} 
{\foreach \y in {1,...,\i}
{
\pgfmathsetmacro{\w}{\j+(\y)/(\i+1)};
\node at (\i,\w){$\circ$};
\node at (\i-.1,\w+.1){$\bullet$};
\draw[gray,thick]  (\i,\w)-- (\i-.1,\w+.1);
}}
\pgfmathsetmacro{\v}{\m-1};
\foreach \i in {1,...,\X}
 \foreach \j in {0,...,\v} 
{\foreach \y in {0,...,\i}
{
\pgfmathsetmacro{\w}{\j+(\y)/(\i+1)};
\pgfmathsetmacro{\q}{\j+(\y+1)/(\i+1)};
\draw[gray,thick]  (\i-.1,\w+.1)-- (\i,\q);
}}
\foreach \i in {1,...,\X}
 \foreach \j in {\m,...,\m} 
{\foreach \y in {1,...,\i}
{
\pgfmathsetmacro{\w}{\j+(\y-1)/(\i+1)};
\pgfmathsetmacro{\q}{\j+(\y)/(\i+1)};
\draw[gray,thick]  (\i-.1,\w+.1)-- (\i,\q);
}}
\foreach \i in {0,...,0}
 \foreach \j in {1,...,\m} 
{
{
\draw[gray,thick]  (\i-.1,\j-1+.1)-- (\i,\j);
}}

 \foreach \i in {2,...,\X}
 \foreach \j in {0,...,\m} 
{\foreach \y in {2,...,\i}
{

\pgfmathsetmacro{\w}{\j+(\y-1)/\i};
\pgfmathsetmacro{\z}{\j+(\y-1)/(\i+1)};
\draw[gray,thick] (\i-1,\w) -- (\i-.1,\z+.1);
}}
}

\newcommand\SHLL[2]{
  \def\X{#1} \def\Y{#2}
  \foreach \i in {0,...,\X}
  {
\pgfmathsetmacro{\m}{\Y-1/(\i+1)};
    \draw[gray,thick] (\i,0) -- (\i,\m);
    \node at (\i,-.3) {\i};
  }
  \foreach \j in {0,...,\Y}
  {
    \node at (-.3,\j) {\j};
  }
\pgfmathsetmacro{\m}{\Y-1};
  \foreach \x in {1,...,\X}
{ \foreach \j in {0,...,\m}
  {
}}
 \foreach \i in {1,...,\X}
 \foreach \j in {0,...,\m} 
{\foreach \y in {1,...,\i}
{
\pgfmathsetmacro{\w}{\j+(\y-1)/\i};
\pgfmathsetmacro{\z}{\j+(\y)/(\i+1)};
\draw[gray,thick] (\i-1,\w) -- (\i,\z);
}}
\node at (\X+1,1) {$\cdots$};
\node at (\X+1,\Y-1) {$\cdots$};
}

\newtheoremstyle{definition}
{3pt} 
{3pt} 
{} 
{} 
{\bfseries} 
{.} 
{.5em} 
{} 

\title{Arctic curves phenomena for bounded lecture hall Tableaux}

\date{\today}
\author{Sylvie Corteel}
\address{Department of Mathematics, UC Berkeley, USA}
\email{corteel@berkeley.edu}

\author{David Keating}
\address{Department of Mathematics, UC Berkeley, USA}
\email{dkeating@berkeley.edu}

\author{Matthew Nicoletti}
\address{Department of Mathematics, MIT, USA}
\email{mnicolet@mit.edu}

\begin{document}

\begin{abstract}
 Recently the first author and Jang Soo Kim introduced lecture hall tableaux in their study of multivariate
little $q$-Jacobi polynomials. They then enumerated bounded lecture hall tableaux and showed
that their enumeration is closely related to standard and semistandard Young tableaux. In this paper we 
study the asymptotic behavior of these bounded tableaux thanks to two other combinatorial models:
non-intersecting paths on a graph whose faces are squares and pentagons
and dimer models on a lattice whose faces are hexagons and octagons.
We use the tangent method to investigate the arctic curve in the model of non-intersecting lattice paths with 
fixed starting points and ending points distributed according to some arbitrary piecewise differentiable 
function. We then study the dimer model and use an ansatz to guess the asymptotics 
of the inverse of the Kasteleyn, which
confirm the arctic curve computed with the tangent method for two examples.
\end{abstract}

\maketitle

\section{Introduction}
\label{intro}

Recently the first author and Jang Soo Kim introduced lecture hall tableaux in their study of multivariate
little $q$-Jacobi polynomials \cite{CK}. They then enumerated bounded lecture hall tableaux and showed
that their enumeration is closely related to standard and semistandard Young tableaux \cite{CK2}.

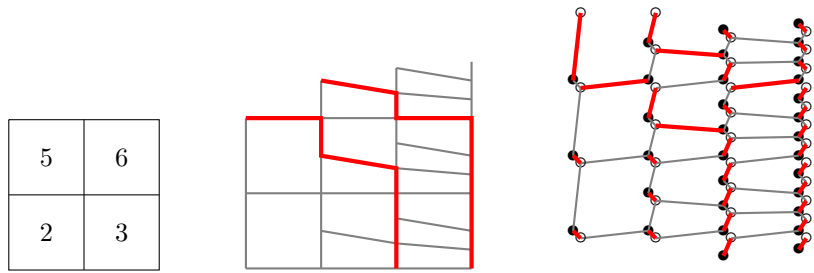
\begin{figure}
  \centering
\begin{tikzpicture}
\draw(0,0)--(2,0)--(2,-2)--(0,-2)--(0,0);
\draw(0,-1)--(2,-1);
\draw(1,-2)--(1,0);
\node at (0.5,-0.5){5};
\node at (1.5,-0.5){6};
\node at (1.5,-1.5){3};
\node at (0.5,-1.5){2};
\end{tikzpicture}\hspace{1cm}
\begin{tikzpicture}
\BLHLL{3}3
\draw [red,ultra thick](0,2)--(1,2)--(1,3/2)--(2,4/3)--(2,0);
\draw [red,ultra thick](1,5/2)--(2,7/3)--(2,2)--(3,2)--(3,0);
\end{tikzpicture}\hspace{1cm}
\begin{tikzpicture}
\DLL{3}3
\draw [red,ultra thick](0,2)--(1-.1,2.1);
\draw [red,ultra thick](1,2)--(1-.1,1.6);;
\draw [red,ultra thick](1,1.5)--(2-.1,4/3+.1);;
\draw [red,ultra thick](2,4/3)--(2-.1,1+.1);;
\draw [red,ultra thick](2,1)--(2-.1,2/3+.1);;
\draw [red,ultra thick](2,2/3)--(2-.1,.1+1/3);;
\draw [red,ultra thick](2,1/3)--(2-.1,.1);;
\draw [red,ultra thick](1,5/2)--(2-.1,7/3+.1);
\draw [red,ultra thick](2,7/3)--(2-.1,6/3+.1);
\draw [red,ultra thick](2,6/3)--(3-.1,6/3+.1);
\draw [red,ultra thick](3,8/4)--(3-.1,7/4+.1);
\draw [red,ultra thick](3,7/4)--(3-.1,6/4+.1);
\draw [red,ultra thick](3,6/4)--(3-.1,5/4+.1);
\draw [red,ultra thick](3,5/4)--(3-.1,4/4+.1);;
\draw [red,ultra thick](3,4/4)--(3-.1,3/4+.1);;
\draw [red,ultra thick](3,3/4)--(3-.1,2/4+.1);;
\draw [red,ultra thick](3,2/4)--(3-.1,1/4+.1);;
\draw [red,ultra thick](3,1/4)--(3-.1,0/4+.1);;
\draw [red,ultra thick](0,1)--(0-.1,1.1);
\draw [red,ultra thick](0,0)--(0-.1,0.1);
\draw [red,ultra thick](1,0)--(1-.1,0.1);
\draw [red,ultra thick](1,0.5)--(1-.1,0.6);
\draw [red,ultra thick](1,1)--(1-.1,1.1);
\draw [red,ultra thick](2,5/3)--(2-.1,5/3+.1);
\draw [red,ultra thick](2,8/3)--(2-.1,8/3+.1);
\draw [red,ultra thick](3,2+1/4)--(3-.1,2+.1+1/4);
\draw [red,ultra thick](3,2+2/4)--(3-.1,2+.1+2/4);
\draw [red,ultra thick](3,2+3/4)--(3-.1,2+.1+3/4);
\node at (0,3){$\circ$};
\node at (1,3){$\circ$};
\node at (2-.1,-1/3+.1){$\bullet$};
\node at (3-.1,-1/4+.1){$\bullet$};
\draw [red,ultra thick](-.1,2.1)--(0,3);
\draw [red,ultra thick](1-.1,2.6)--(1,3);
\draw [red,ultra thick](2-.1,-1/3+.1)--(2,0);
\draw [red,ultra thick](3-.1,-1/4+.1)--(3,0);
\end{tikzpicture}
\caption{Tableau, non-intersecting paths, and dimers}
\label{tabpathdim}
\end{figure}


Given a positive integer $t$ and a partition $\lambda=(\lambda_1,\ldots ,\lambda_n)$
with $\lambda_1\ge \ldots \ge \lambda_n\ge 0$, the bounded lecture hall tableaux
are  fillings of the diagram of $\lambda$ with integers $T_{i,j}$ such that
\begin{enumerate}
\item $T_{i,j}<t(n-i+j)$
\item $T_{i,j}/(n-i+j)\ge T_{i,j+1}/(n-i+j+1)$
\item $T_{i,j}/(n-i+j)> T_{i+1,j}/(n-i-1+j)$
\end{enumerate}
We call them bounded lecture hall tableaux (BLHT)
of shape $\lambda$ and bounded by $t$.
On the left of Figure \ref{tabpathdim}, we give an example of such a tableau for $t=3$
and $\lambda=(2,2)$.
 In this paper we 
study the asymptotic behavior of these bounded tableaux thanks to two other combinatorial models:
the non-intersecting paths on a graph whose faces are squares and pentagons
and the dimer models on a lattice whose faces are hexagons and octagons.
An example of the path model and the dimer model is given on the middle and the right of Figure \ref{tabpathdim}.
Detailed definitions will be given in Section \ref{combi}.


One special quality of this model is that the number of configurations is relatively 
easy to compute \cite{CK2}.
Given $t,n$ and $\lambda=(\lambda_1,\ldots ,\lambda_n)$, the number $Z_\lambda(t)$ 
of bounded lecture hall tableaux
of shape $\lambda$ bounded by $t$ is
\begin{equation*}
Z_\lambda(t)=t^{|\lambda|}\prod_{1\le i<j\le n}\frac{\lambda_i-i-\lambda_j+j}{j-i},
\end{equation*}
where $|\lambda|=\lambda_1+\ldots +\lambda_n$.

Our main interest here is to compute their asymptotic behavior.
Given $n\gg 0$ and a function $\alpha:[0,1]\rightarrow [0,k]$ with $k\in\mathbb {R}$ satisfying some conditions
that will be given in a later section
and $\tau\in \mathbb {R}$, our  main question is to understand the asymptotic
behavior of bounded lecture hall tableaux of shape $\lambda=(\lambda_1,\ldots ,\lambda_n)$ with $\lambda_i+n-i=\lfloor n\alpha\left(\frac{i}{n}\right)\rfloor$
bounded by $t=\tau n$.
The function $\alpha$ describes the limiting profile of $\lambda$.

In this paper, we will detail two examples 
\begin{itemize}
\item The staircase: $\lambda=(n,n-1,\ldots ,2,1)$. In this case $\alpha(u)=2-2u$.
\item The square:  $\lambda=(n,n,\ldots ,n)$. In this case $\alpha(u)=2-u$.
\end{itemize}
We will also present a general result that computes a parameterization of the arctic curve in the general case.
\begin{thm} \label{thm:main}
Assuming the tangent method holds,
as $n\rightarrow \infty$, lecture hall tableaux of shape $\lambda=(\lambda_1,\ldots ,\lambda_n)$ 
bounded by $\tau n$ exhibit the arctic curve phenomenon. The arctic curve can be parameterized by
\begin{align}
\begin{split}
X(x) = \frac{x^2I'(x)}{I(x) + x I'(x)} \\
Y(x) = \tau \frac{1}{I(x) + x I'(x)}
\end{split}
\label{thm:main}
\end{align}
for an appropriate range of $x$. Here $I(x) = e^{-\int_0^1 du\; \frac{1}{x-\alpha(u)}}$
and  $\lambda_i+n-i=\lfloor n\alpha(i/n)\rfloor$.
\end{thm}

Even though this model seems more complicated than the typical systems 
coming from the square grid graphs \cite{DG}, we will discover that they have some 
surprising properties and lots of similarities with non-intersecting paths on the grid (or equivalently
dimer models on the hexagonal lattice, semi standard young tableaux or lozenge tilings). 
Numerous asymptotic results exist for non-intersecting paths, or equivalently tilings models 
or  dimer models on graphs that are regular and ${\mathbb Z}^2$ invariant \cite{K}.
The arctic curve phenomenon was named about twenty years ago when Cohn, Elkies and Propp
studied the tilings of a large Aztec diamond with dominoes \cite{CEP96, JPS98}.
Indeed the ``typical" tiling of the Aztec diamond with dominoes is known to display an arctic circle
separating frozen phases in the corners which are regularly tiled from a liquid phase in the
center which is disordered.
Many tiling problems of finite plane domains of large size are known to exhibit the same phenomenon.
Typically, one studies the asymptotics of tilings of scaled domains
whose limits can be nicely characterized. Dimer models on regular graphs, which are the
dual version of tiling problems, exhibit the same phenomenon \cite{KO06, KO07}. 
The general method to obtain the arctic curve location is the asymptotic study of bulk
expectation values, which requires the computation of the inverse of the Kasteleyn matrix
or at least its asymptotics. Other rigorous methods use for example the machinery of
cluster integrable systems of dimers \cite{DFSG14,KP,PS}. Recently several papers use the
method of Colomo and Sportiello \cite{CS,CPS} called the {\em tangent method} to
compute (non rigorously) the arctic curves  \cite{DG,DG1,DG2,DL,DR}.
A very recent preprint of Aggarwal builds a method to make this heuristic
rigorous in the case of the 6-vertex model \cite{A}.

As our model is not $\mathbb Z^2$ invariant we can not apply directly all the methods elaborated
for the  $\mathbb Z^2$ invariant models.
In Section \ref{combi}
we will define the path model, the dimer model and explain the connections between these models and bounded lecture hall tableaux.
In Section \ref{simu} we explain how we randomly generate the tableaux and
present some simulations. In Section \ref{tangent}
we use the tangent method to compute the arctic curve for any function $\alpha$. 
 Using the dimer model we compute (non rigorously) the arctic curve
for our two running examples using an ansatz to guess the asymptotic behavior 
of the inverse of the Kasteleyn matrix. This will be presented in Section \ref{dimer}.
We end this paper in Section \ref{conclu}
with open questions and concluding remarks.\\

{\bf Acknowledgments.} The authors want to thank Nicolai Reshetikhin, Ananth Sridhar and Andrea Sportiello
for their precious comments and advice during the elaboration of this work.
They also acknowledge the constructive comments of the referees that helped to
improve the quality of the manuscript.
SC was in residence at MSRI in Berkeley (NSF grant DMS-1440140) 
during the fall of 2018. SC is partially funded by the ANR grant ANR-18-CE40-0033
and the UC Berkeley start up funds. DK is partially supported by the NSF grant DMS-1902226 and the FRG grant DMS-1664521.

\section{Combinatorics and counting}
\label{combi}

In this section , we give definitions and basic properties of our three combinatorial models: the tableaux,
the path model and the dimer model.

\subsection{Lecture hall tableaux}

Lecture hall partitions were studied by Bousquet-M\'elou and Eriksson \cite{BME1,BME2,BME3} in the context of the combinatorics of affine Coxeter groups.
They are sequences $(T_1,\ldots ,T_n)$ such that
$$
\frac{T_1}{n}\ge \frac{T_2}{n-1}\ge \ldots \ge \frac{T_n}{1}\ge 0.
$$
They have been studied extensively in the last two decades. See the recent survey written by Savage \cite{LHPSavage}.  
In \cite{CK} the first author and Jang Soo Kim showed 
that these objects are closely related to the little $q$-Jacobi polynomials.
Thanks to this approach they defined lecture hall tableaux related to the multivariate  little $q$-Jacobi polynomials.

Given a partition $\lambda=(\lambda_1,\ldots ,\lambda_n)$ such that $\lambda_1\ge \ldots \ge \lambda_n\ge 0$,
the Young diagram of $\lambda$ is a left justified
union of cells such that the $i^{th}$ row contains $\lambda_i$ cells.
The cell in row $i$ and column $j$ is denoted by $(i,j)$.

\begin{dfn}\cite{CK}
For an integer $n$ and a partition $\lambda=(\lambda_1,\ldots ,\lambda_n)$ with $n$ non-negative parts, a \emph{lecture hall tableau} of shape $\lambda$ 
 is a filling $T$ of the cells in the Young diagram of $\lambda$ with non-negative integers satisfying the following conditions:
\[
\frac{T(i,j)}{n-i+j}  \ge \frac{T(i,j+1)}{n-i+j+1}, \qquad
\frac{T(i,j)}{n-i+j} > \frac{T(i+1,j)}{n-i-1+j}.
\]
where $T(i,j)$ is the filling of the cell in row $i$ and column $j$.
\end{dfn}
 See Figure~\ref{fig:LHT} for an example of a lecture hall tableau on the left of the Figure. 
On the right of the Figure, we show that this tableau is ``lecture hall" by exhibiting $T_{i,j}/(n-i+j)$ for all $i,j$.

\begin{figure}
  \centering
\begin{tikzpicture}[scale=.6]
\cell11{16} \cell12{16} \cell13{9} \cell144 
\cell21{12}\cell22{13} \cell236 
\cell312 
\draw (0,-3)--(0,0)--(4,0);
\end{tikzpicture} \qquad \qquad
\begin{tikzpicture}[scale=.6]
\cell11{$\frac{16}{5}$} \cell12{$\frac{16}{6}$} \cell13{$\frac97$} \cell14{$\frac48$} 
\cell21{$\frac{12}{4}$}\cell22{$\frac{13}{5}$} \cell23{$\frac{6}{6}$} 
\cell31{$\frac{2}{3}$} 
\draw (0,-3)--(0,0)--(4,0);
\end{tikzpicture}
  \caption{On the left is a lecture hall tableau $T$ for $n=5$ and $\lambda=(4,3,1,0,0)$. 
The diagram on the right shows the number $T(i,j)/(n-i+j)$.}
  \label{fig:LHT}
\end{figure}

In this paper we study lecture hall tableaux with an extra condition. We impose that each entry $T(i,j)$ is striclty
less then $t(n+j-i)$. We say that the tableau are bounded by $t$.
These tableaux are called {\em bounded} lecture hall tableaux and were enumerated  in \cite{CK2}.
\begin{prop}\cite{CK2}
Given $t,n$ and $\lambda$, the number of bounded lecture hall tableaux is
\begin{equation}
Z_\lambda(t)=t^{|\lambda|}\prod_{1\le i<j\le n}\frac{\lambda_i-i-\lambda_j+j}{j-i},
\end{equation}
where $|\lambda|=\lambda_1+\ldots +\lambda_n$.
\end{prop}

\subsection{Paths on the lecture hall graph and the dual graph}

In this Section we give a bijection between bounded lecture hall tableaux and non-intersecting paths on a graph.
Let us give a detailed definition of our graph.
\begin{dfn}
Given a positive integer $t$, the {\em lecture hall graph}
is a graph ${\mathcal G}_t=(V_t,E_t)$. This graph is conveniently described through an embedding
in the plane, in which the vertices are:
\begin{itemize}
\item $(i,j/(i+1))$ for $i\ge 0$ and $0\le j<t(i+1)$.
\end{itemize}
and the directed edges are
\begin{itemize}
\item from $(i,k+r/(i+1))$ to  $(i+1,k+r/(i+2))$  for $i\ge 0$ and $0\le r\le i$ and $0\le k<t$.
\item from $(i,k+(r+1)/(i+1))$ to $(i,k+r/(i+1))$ for $i\ge 0$ and $0\le r\le i$ and $0\le k<t-1$
or for $i\ge 0$ and $0\le r< i$ and $k=t-1$.
\end{itemize}
\end{dfn}
This graph was defined in \cite{CK2}.
An example of the graph ${\mathcal G}_3$ is given on Figure \ref{t3}.
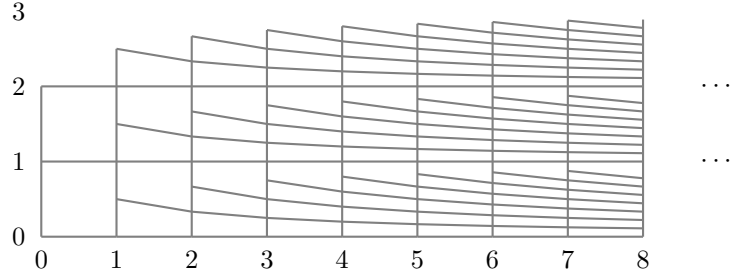
\begin{figure}
  \centering
\begin{tikzpicture}
\LHLL{8}3
\end{tikzpicture}
\caption{The graph ${\mathcal G}_t$ for $t=3$}
\label{t3}
\end{figure}

Given a positive integer $t$ and a partition $\lambda=(\lambda_1,\ldots ,\lambda_n)$
with $\lambda_1\ge \ldots \ge \lambda_n\ge 0$,
the non-intersecting path system is a system of $n$ paths on the graph ${\mathcal G}_t$. The $i^{th}$ path
starts at $(n-i,t-1/(n-i+1))$ and ends at $(\lambda_i+n-i,0)$. The paths are said to be not intersecting if they 
do not share a vertex.
On Figure \ref{fig:t4} we give an example of non-intersecting paths on ${\mathcal G}_4$
for $n=5$ and $\lambda=(4,3,1,0,0)$.
Note that the paths are on a finite portion of  ${\mathcal G}_t$ and we can delete all the vertices
$(x,y)$ with $x>\lambda_1+n-1$.

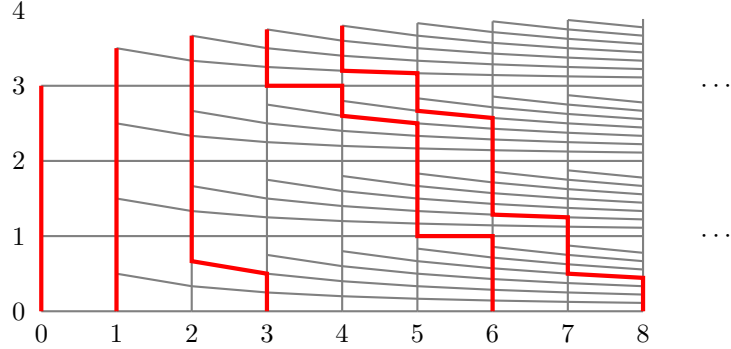
\begin{figure}
  \centering
\begin{tikzpicture}
\LHLL{8}4
\draw [red,ultra thick] (0,3)--(0,0);
\draw [red,ultra thick](1,7/2)--(1,0);
\draw [red,ultra thick](2,11/3)--(2,2/3)--(3,2/4)--(3,0);
\draw [red,ultra thick](3,15/4)--(3,12/4)--(4,15/5)--(4,13/5)--(5,15/6)--(5,6/6)--(6,1)--(6,0);
\draw [red,ultra thick] (4,19/5) -- (4,16/5) -- (5,19/6) -- (5,16/6) -- (6,18/7) -- (6,9/7) -- (7,10/8)-- (7,4/8) -- (8,4/9) -- (8,0);
\end{tikzpicture}
\caption{non-intersecting lattice paths on ${\mathcal G}_4$
for $n=5$ and $\lambda=(4,3,1,0,0)$.}
\label{fig:t4}
\end{figure}

We give a sketch of the proof of their enumeration:
\begin{theorem}\cite{CK2}
Given $t,n$ and $\lambda$, the number $Z_\lambda(t)$ of non-intersecting path configurations that
start at $(n-i,t-1/(n-i+1))$ and end at $(\lambda_i+n-i,0)$ for $i=1,\ldots ,n$
is
\begin{equation*}
t^{|\lambda|}\prod_{1\le i<j\le n}\frac{\lambda_i-i-\lambda_j+j}{j-i},
\end{equation*}
where $|\lambda|=\lambda_1+\ldots +\lambda_n$.
\end{theorem}
\begin{proof}
Using the Lindstr\"om-Gessel-Viennot Lemma \cite{GV}, we know that
the number of configurations is equal to
$$
\det(P(u_i,v_j))_{1\le i,j\le n}
$$
where $P(u_i,v_j)$ is the number of paths from $(n-i,t-1/(n-i+1))$
$(\lambda_j+n-j,0)$. It is easy to prove that if $t=1$ the number of paths is
${\lambda_j+n-j\choose n-i}$. Using induction on $t$ one can easily check that
$$
P(u_i,v_j)={\lambda_j+n-j\choose n-i}t^{\lambda_j-j+i}.
$$
The result follows. Indeed computing the determinant 
can be done using induction on $n$ and the fact that 
$$
t^{\lambda_j-j+i}{\lambda_j+n-j\choose n-i}=\frac{t^{\lambda_j-j}(\lambda_j+n-j)}{t^{-i}(n-i)}{\lambda_j+n-j-1\choose n-i-1}.
$$
Details and generalizations can be found in \cite{CK2}.
\end{proof}

In the case $t=1$, we get back a very classical result \cite[Chapter 7]{stanley}. The Schur polynomial
$s_\lambda(x_1,\ldots ,x_n)$ specialized at $x_i=1$ for all $i$ is equal to
$$
s_\lambda(1,\ldots ,1)=\prod_{1\le i<j\le n}\frac{\lambda_i-i-\lambda_j+j}{j-i}.
$$
Now let us present the link between the paths and the tableaux:
\begin{theorem}\cite{CK2}
There exists a bijection between the bounded lecture tableaux of shape $\lambda$ and bounded by $t$
and non-intersecting paths on ${\mathcal G}_t$ starting at $(n-i,t-1/(n-i+1)$ and ending at
$(n-i+\lambda_i,0)$ for $i=1,\ldots, n$.
\end{theorem}
\begin{proof}
The $i^{th}$ path starts at $(n-i,t-1/(n-i+1)$ and ends at $(\lambda_i+n-i,0)$. It is in bijection with the $i^{th}$ row
of the tableau. The number of cells under the $j^{th}$ horizontal step of the $i^{th}$
path is exactly $T_{i,j}$.
\end{proof}
The tableau on the left of Figure \ref{fig:LHT} is in bijection with the paths on Figure \ref{fig:t4}.

The graph ${\mathcal G}_t$ has a dual graph that we denote by ${\mathcal D}_t$.
The duality here is {\emph not} the duality of planar graphs but the duality of
paths on these graphs. The concept of dual paths is an idea due
to Gessel and Viennot \cite[Section4]{GesselViennot} which we generalize to our context.
\begin{defn}
Given a positive integer $t$, the {\em dual lecture hall graph}
is a graph ${\mathcal D}_t=(V_t,E_t)$.
The vertices of the  ${\mathcal D}_t$ are:
\begin{itemize}
\item $(i,j/(i+1))$ for $i\ge 0$ and $0\le j<t(i+1)$.
\end{itemize}
and the directed edges are
\begin{itemize}
\item from $(i,k+r/(i+1))$ to  $(i+1,k+(r+1)/(i+2))$  for $i\ge 0$ and $0\le r\le i$ and $0\le k<t$.
\item from $(i,k+r/(i+1))$ to $(i,k+(r+1)/(i+1))$ for $i\ge 0$ and $0\le r\le i$ and $0\le k<t-1$ or for $i\ge 0$ and $0\le r\le i$ and $k=t-1$.
\end{itemize}
\end{defn}
Note that the vertices of ${\mathcal G}_t$ and ${\mathcal D}_t$ are embedded in the plane in the same way.
The only difference is that the edges $(i,k+r/(i+1))$ to  $(i+1,k+r/(i+2))$ in ${\mathcal G}_t$
are replaced by $(i,k+r/(i+1))$ to  $(i+1,k+(r+1)/(i+2))$  in ${\mathcal D}_t$
and the direction of the vertical edges is reversed.

An example of the graph ${\mathcal D}_3$ is given on Figure \ref{d3}.
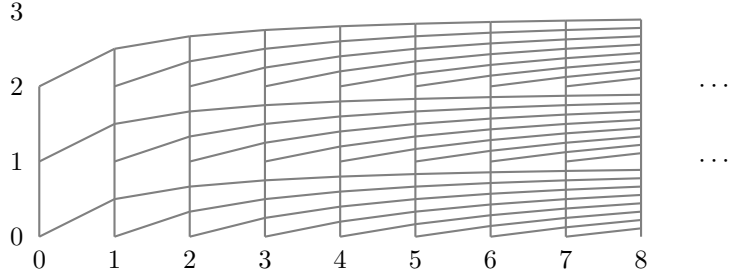
\begin{figure}
  \centering
\begin{tikzpicture}
\SHLL{8}3
\end{tikzpicture}
\caption{The graph ${\mathcal D}_t$ for $t=3$}
\label{d3}
\end{figure}

Let us now explain why
the paths on the graphs ${\mathcal D}_t$ and 
${\mathcal G}_t$ are dual.
Given $\lambda=(\lambda_1,\ldots ,\lambda_n)$ with $\lambda_1\le m$,
let $\lambda'=(\lambda'_1,\ldots \lambda'_m)$ be such that
$$
\lambda'_i=\#\{j\ | \ \lambda_j\ge i\}.
$$
We call $\lambda'$ the conjugate of $\lambda$.
\begin{prop}
There exists a bijection between
\begin{itemize}
\item Systems of $n$ non-intersecting paths on ${\mathcal G}_t$ that start at 
$(n-i,t-1/(n-i+1))$ and end at $(\lambda_i+n-i,0)$ for $i=1,\ldots ,n$.
and 
\item  Systems of $m$ non-intersecting paths on ${\mathcal D}_t$ that start at 
$(n+i-1-\lambda'_i,0)$ and end at $(n+i-1,t-1/(n+i))$
for $i=1,\ldots ,m$.
\end{itemize}
\end{prop}
\begin{proof}
Starting from a system of $n$ non-intersecting paths on ${\mathcal G}_t$ (see Figure \ref{fig:t4}),
we delete all the vertical steps on ${\mathcal G}_t$. See Figure \ref{bijection}.
We then replace each horizontal step of the form $(i,k+r/(i+1))$ to  $(i+1,k+r/(i+2))$ on ${\mathcal G}_t$
by a step  $(i,k+r/(i+1))$ to  $(i+1,k+(r+1)/(i+2))$  on ${\mathcal D}_t$. See Figure \ref{bijection2}.
We  add the vertical steps  on ${\mathcal D}_t$
so that the paths start at $(\lambda_1+i+1-\lambda'_i,0)$ and ending at $(\lambda_1+i+1,t-1/(\lambda_1+i+2))$. See Figure \ref{fig:d4}.
This is easily reversible.
\end{proof}
An example is given on Figures \ref{bijection} and \ref{bijection2}. We start from the paths from Figure \ref{fig:t4}
for $\lambda=(4,3,1,0,0)$ and $n=5 $ and end with the paths on Figure \ref{fig:d4}
for $\lambda'=(3,2,2,1)$ and $m=4$. We draw the horizontal edges
in blue and the vertical edges in red to illustrate the construction.

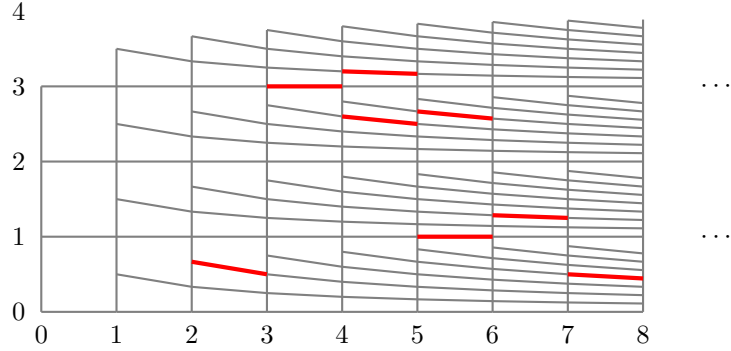
\begin{figure}
  \centering
\begin{tikzpicture}
\LHLL{8}4
\draw [red,ultra thick](2,2/3)--(3,2/4);
\draw [red,ultra thick](3,12/4)--(4,15/5);
\draw [red,ultra thick](4,13/5)--(5,15/6);
\draw[red,ultra thick](5,6/6)--(6,1);
\draw [red,ultra thick](4,16/5) -- (5,19/6);
\draw [red,ultra thick](5,16/6) -- (6,18/7);
\draw [red,ultra thick] (6,9/7) -- (7,10/8);
\draw [red,ultra thick](7,4/8) -- (8,4/9);
\end{tikzpicture}
\caption{Horizontal steps of the non-intersecting lattice paths on Figure \ref{fig:t4}}
\label{bijection}
\end{figure}

\begin{figure}
  \centering
\begin{tikzpicture}
\SHLL{8}4
\draw [blue,ultra thick](2,2/3)--(3,3/4);
\draw [blue,ultra thick](3,12/4)--(4,16/5);
\draw [blue,ultra thick](4,13/5)--(5,16/6);
\draw[blue,ultra thick](5,6/6)--(6,1+1/7);
\draw [blue,ultra thick](4,16/5) -- (5,20/6);
\draw [blue,ultra thick](5,16/6) -- (6,19/7);
\draw [blue,ultra thick] (6,9/7) -- (7,11/8);
\draw [blue,ultra thick](7,4/8) -- (8,5/9);
\end{tikzpicture}
\caption{Horizontal steps of the dual graph}
\label{bijection2}
\end{figure}
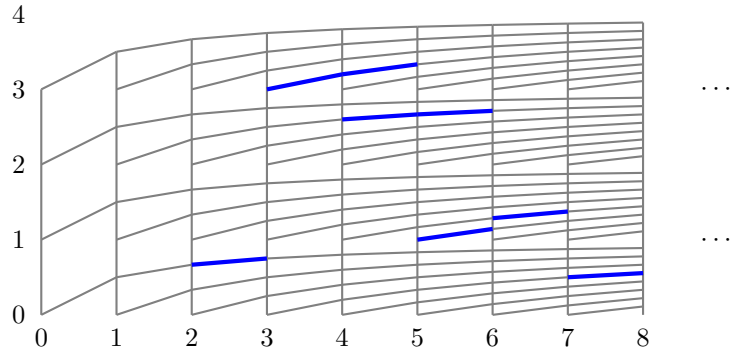

\begin{figure}
  \centering
\begin{tikzpicture}
\SHLL{8}4
\draw [red,ultra thick](2,0)--(2,2/3);
\draw [blue,ultra thick](2,2/3)--(3,3/4);
\draw [red,ultra thick](3,3/4)--(3,3);
\draw [blue,ultra thick](3,12/4)--(4,16/5);
\draw [red,ultra thick](4,0)--(4,13/5);
\draw [blue,ultra thick](4,13/5)--(5,16/6);
\draw[red,ultra thick](5,0)--(5,6/6);
\draw[blue,ultra thick](5,6/6)--(6,1+1/7);
\draw[red,ultra thick](6,1+1/7)--(6,9/7);
\draw [blue,ultra thick](4,16/5) -- (5,20/6);
\draw[red,ultra thick](5,20/6)--(5,23/6);
\draw [blue,ultra thick](5,16/6) -- (6,19/7);
\draw[red,ultra thick](6,19/7)--(6,27/7);
\draw [blue,ultra thick] (6,9/7) -- (7,11/8);
\draw[red,ultra thick](7,11/8)--(7,31/8);
\draw[red,ultra thick](7,0)--(7,4/8);
\draw [blue,ultra thick](7,4/8)--(8,5/9);
\draw[red,ultra thick](8,5/9)--(8,35/9);
\end{tikzpicture}
\caption{Non-intersecting lattice paths on ${\mathcal D}_4$
for $m=4$ and $\lambda'=(3,2,2,1)$.} 
\label{fig:d4}
\end{figure}
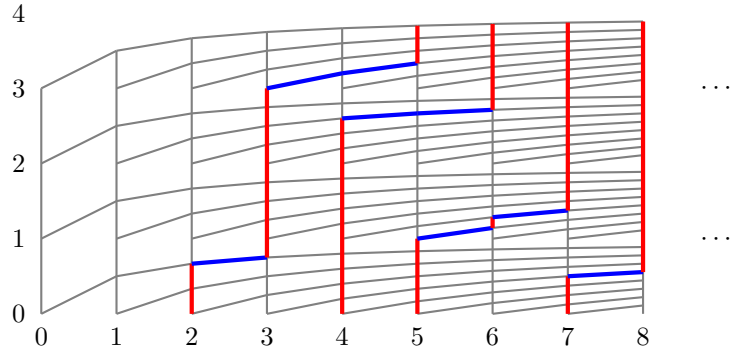

\subsection{Dimer models}

In this section we give a bijection between lecture hall tableaux of shape $\lambda$
bounded by $t$ and the dimer model on an embedded bipartite graph whose faces are hexagons
and octagons. To do so we first replace each vertex $(i,k+r/(i+1))$  of the lecture Hall graph ${\mathcal
G}_t$ by a white vertex $w(i,k+r/(i+1))$ and a black vertex $b(i,k+r/(i+1))$ joined by an edge.
Then we replace
\begin{itemize}
\item the edge from $(i,k+r/(i+1))$ to  $(i+1,k+r/(i+2))$  for $i\ge 0$ and $0\le r\le i$ and $0\le k<t$;
by the edge from the black vertex  $b(i+1,k+r/(i+2))$  to the  white vertex  $w(i,k+r/(i+1))$.
\item the edge from $(i,k+(r+1)/(i+1))$ to $(i,k+r/(i+1))$ for $i\ge 1$ and $0\le r\le i$ and $0\le k<t$;
by the edge from the weight vertex  $w(i,k+(r+1)/(i+1))$ to the black vertex  $b(i,k+r/(i+1))$.
\end{itemize}
We call this new graph the {\em lecture hall lattice} and denote it by ${\mathcal H}_t$. An example 
for $t=3$ is given on Figure \ref{t33}. 
To simplify the notation, we now write the black vertex $(i,k+r(i+1))$ for the black vertex $b(i,k+r(i+1))$.
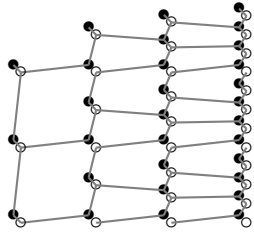
\begin{figure}
  \centering
\begin{tikzpicture}
\DLL{3}3
\end{tikzpicture}
\caption{The lecture hall lattice  ${\mathcal H}_3$}
\label{t33}
\end{figure}

Now to build a bijection from lecture Hall tableaux of shape $\lambda$
bounded by $t$ and dimer models on  ${\mathcal H}_t$, we look at dimer configurations where
\begin{itemize}
\item we add a white vertex $(n-i,t)$ and an edge from this vertex to the black vertex
$(n-i,t-1/(n-i+j)$;
\item we add a black vertex $(n-i+\lambda_i,-1/(n-i+\lambda_i+1))$ and an edge from the white vertex
$(n-i+\lambda_i,0)$ to  this vertex;
\end{itemize}
for $1\le i\le n$. We call this graph the decorated lecture hall lattice  ${\mathcal H}_t(\lambda)$.
On Figure \ref{t32} we give an example of ${\mathcal H}_3(2,2)$.

\begin{figure}
  \centering
\begin{tikzpicture}
\DLL{3}3
\node at (0,3){$\circ$};
\node at (1,3){$\circ$};
\node at (2-.1,-1/3+.1){$\bullet$};
\node at (3-.1,-1/4+.1){$\bullet$};
\draw [gray,thick](-.1,2.1)--(0,3);
\draw [gray,thick](1-.1,2.6)--(1,3);
\draw [gray,thick](2-.1,-1/3+.1)--(2,0);
\draw [gray,thick](3-.1,-1/4+.1)--(3,0);
\end{tikzpicture}
\caption{The decorated lecture hall lattice ${\mathcal H}_3(2,2)$}
\label{t32}
\end{figure}
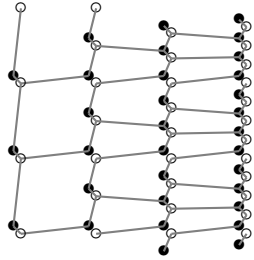

We now give the bijection from the non-intersecting paths to the dimer model. The bijection is 
a specialization to our graphs of a classical technique in non-intersecting lattice 
paths and dimer models. 
Whenever a path uses an edge from vertex $(x,y)$ to vertex $(w,z)$ on the lecture hall graph, we put a dimer
on the edge from the white vertex $(x,y)$ to the black vertex $(w,z)$ on the lecture hall lattice.
Whenever a vertex $(x,y)$ is not used by any path,  we put a dimer
on the edge from the black vertex $(x,y)$ to the white vertex $(x,y)$.
Then we add dimers on the edge from the white vertex $(n-i,t)$ to the  black vertex
$(n-i,t-1/(n-i+j)$ and on the edge from the white vertex
$(n-i+\lambda_i,0)$ to  the black  vertex $(n-i+\lambda_i,-1/(n-i+\lambda_i+1))$ for $1\le i\le n$.
An example is given on Figure \ref{t322}. Here $n=2$ and $\lambda=(2,2)$.

\begin{figure}
  \centering
\begin{tikzpicture}
\LHLL{3}3
\draw [red,ultra thick](0,2)--(1,2)--(1,3/2)--(2,4/3)--(2,0);
\draw [red,ultra thick](1,5/2)--(2,7/3)--(2,2)--(3,2)--(3,0);
\end{tikzpicture}\hspace{1cm}
\begin{tikzpicture}
\DLL{3}3
\draw [red,ultra thick](0,2)--(1-.1,2.1);
\draw [red,ultra thick](1,2)--(1-.1,1.6);;
\draw [red,ultra thick](1,1.5)--(2-.1,4/3+.1);;
\draw [red,ultra thick](2,4/3)--(2-.1,1+.1);;
\draw [red,ultra thick](2,1)--(2-.1,2/3+.1);;
\draw [red,ultra thick](2,2/3)--(2-.1,.1+1/3);;
\draw [red,ultra thick](2,1/3)--(2-.1,.1);;
\draw [red,ultra thick](1,5/2)--(2-.1,7/3+.1);
\draw [red,ultra thick](2,7/3)--(2-.1,6/3+.1);
\draw [red,ultra thick](2,6/3)--(3-.1,6/3+.1);
\draw [red,ultra thick](3,8/4)--(3-.1,7/4+.1);
\draw [red,ultra thick](3,7/4)--(3-.1,6/4+.1);
\draw [red,ultra thick](3,6/4)--(3-.1,5/4+.1);
\draw [red,ultra thick](3,5/4)--(3-.1,4/4+.1);;
\draw [red,ultra thick](3,4/4)--(3-.1,3/4+.1);;
\draw [red,ultra thick](3,3/4)--(3-.1,2/4+.1);;
\draw [red,ultra thick](3,2/4)--(3-.1,1/4+.1);;
\draw [red,ultra thick](3,1/4)--(3-.1,0/4+.1);;
\draw [red,ultra thick](0,1)--(0-.1,1.1);
\draw [red,ultra thick](0,0)--(0-.1,0.1);
\draw [red,ultra thick](1,0)--(1-.1,0.1);
\draw [red,ultra thick](1,0.5)--(1-.1,0.6);
\draw [red,ultra thick](1,1)--(1-.1,1.1);
\draw [red,ultra thick](2,5/3)--(2-.1,5/3+.1);
\draw [red,ultra thick](2,8/3)--(2-.1,8/3+.1);
\draw [red,ultra thick](3,2+1/4)--(3-.1,2+.1+1/4);
\draw [red,ultra thick](3,2+2/4)--(3-.1,2+.1+2/4);
\draw [red,ultra thick](3,2+3/4)--(3-.1,2+.1+3/4);
\node at (0,3){$\circ$};
\node at (1,3){$\circ$};
\node at (2-.1,-1/3+.1){$\bullet$};
\node at (3-.1,-1/4+.1){$\bullet$};
\draw [red,ultra thick](-.1,2.1)--(0,3);
\draw [red,ultra thick](1-.1,2.6)--(1,3);
\draw [red,ultra thick](2-.1,-1/3+.1)--(2,0);
\draw [red,ultra thick](3-.1,-1/4+.1)--(3,0);
\end{tikzpicture}
\caption{From paths on the lecture hall graph ${\mathcal G}_3$ to dimers on the lecture hall lattice  ${\mathcal H}_3(2,2)$}
\label{t322}
\end{figure}
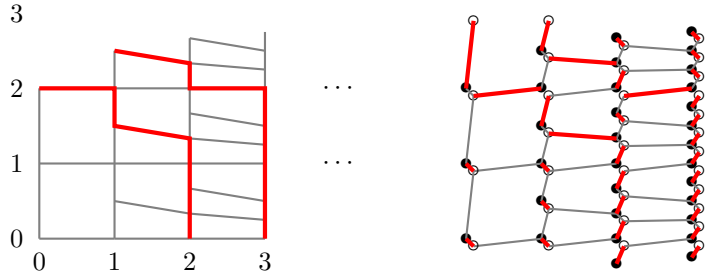

\section{Simulations}
\label{simu}

In order to validate numerically our findings on the Arctic Curve of the model,
we have performed the uniform sampling of large bounded lecture hall tableaux.
In this section we explain the algorithm by which  we randomly generated the bounded lecture hall tableaux. We use an algorithm called ``coupling from the past" \cite{PW}. We adapted a
parallel implementation of coupling from the past due to the second author and Sridhar 
that generates random tilings using GPU \cite{KS}.

Given a partition $\lambda=(\lambda_1,\ldots ,\lambda_n)$ and an integer $t$ we define a partial order on the bounded lecture hall tableaux of shape $\lambda$.
\begin{dfn}
Given two lecture hall tableau $T$ and $U$ of shape $\lambda$ bounded by $t$ then $T\le U$ if and only if 
$$
T(i,j)\le U(i,j)\ \ {\rm for}\ {\rm all}\ i,j.
$$
\end{dfn}
This partial order has a unique minimum $T_{min}$ and maximum $T_{max}$ where
\begin{eqnarray*}
T_{min}(i,j)&=&n-i;\\
T_{max}(i,j)&=&t(n-i+j)-j; \ \ {\rm for}\ {\rm all}\ i,j.
\end{eqnarray*}

We build the Markov chain $MLHT$ on the set of lecture hall tableaux of shape $\lambda$ bounded by $t$.
The vertices of the chain 
are all the tableaux counted by $Z_\lambda(t)$ 
and there is a transition from a tableau $T$ to a tableau $U$
with transition probability $\pi(T,U)=1/|\lambda|$ with $|\lambda|=\lambda_1+\ldots +\lambda_n$; if
$$
T\le U; \ \ {\rm and}\  1+\sum_{i,j}T(i,j)=\sum_{i,j}U(i,j);
$$
or
$$
U\le T; \ \ {\rm and}\  1+\sum_{i,j}U(i,j)=\sum_{i,j}T(i,j).
$$
Then we add the transitions $\pi(T,T)=1-\sum_{U\neq T}\pi(T,U)$ for all tableaux $T$.
This Markov chain is reversible and symmetric. Thus its stationary distribution is uniform.

To sample a random bounded lecture hall tableau, we would like to perform a random walk on $MLHT$.
Nevertheless we do not know the mixing time of this Markov chain. So we use a celebrated technique
due to Propp and Wilson \cite{PW} called coupling from the past.
The coupling-from-the-past algorithm effectively simulates running the Markov chain for an infinite
time. It works as follows: we will run two backward  walks  one starting at $T_{min}$ and
the other one at $T_{max}$. Let $T_{min}^m$ and $T_{max}^m$ be the tableaux after $m$ steps.
These walks will be such that for each $m$, $T_{min}^m\le T_{max}^m$. The algorithm stops
when $T_{min}^m=T_{max}^m$. 

Let us now explain how one step of the algorithm is performed. We change slightly the Markov chain
to speed up the generation.
We number the cells of $\lambda$ from 1 to $|\lambda|$. 
Given two tableaux $T_{min}^m$ and $T_{max}^m$. 
We pick two numbers uniformly in the interval $[0,1]$. Let us call then $k$
and $\ell$. Then we pick the  cell $\lfloor k|\lambda|\rfloor$ of the diagram of $\lambda$.
Let us call this cell $(x,y)$. If the cell $(x,y)$ of $T_{min}^m$ could contain the integers $\{a,a+1,\ldots ,b\}$
without violating the condition for $T_{min}^m$ to be lecture hall then
we change the cell $(x,y)$ to the value $\lfloor a+(b-a+1)\ell\rfloor$.
If the cell $(x,y)$ of $T_{max}^m$ could contain the integers $\{c,c+1,\ldots ,d\}$
without violating the condition for $T_{max}^m$ to be lecture hall then
we change the cell $(x,y)$ to the value $\lfloor c+(d-c+1)\ell\rfloor$. 
We do not change the value of the other cells.
The result is a pair of tableaux $T_{min}^{m+1},T_{max}^{m+1}$ and we denote this by
$(T_{min}^{m+1},T_{max}^{m+1})=R_{k,\ell}(T_{min}^m,T_{max}^m)$.
As $a\le c$ and $b\le d$, then $a+(b-a+1)\ell\le c+(d-c+1)\ell$ and
$T_{min}^{m+1}\le T_{max}^{m+1}$.

To get to $T_{min}^m$ and  $T_{max}^m$, from $T_{min}$ and $T_{max}$
we need to pick $k_1,\ldots ,k_m$ and $\ell_1,\ldots ,\ell_m$
and as we run the walks backwards, then
$$
(T_{min}^m,T_{max}^m)=R_{k_1,\ell_1}(R_{k_2,\ell_2}(\ldots (R_{k_m,\ell_m}(T_{min},T_{max})\ldots ).
$$
In practice we stop when $T_{min}^m$ and $T_{max}^m$
are ``close".
This is very similar to what was done for lozenge and domino tilings. See \cite{KS} for example.

Here we present the simulations for our two running examples. We will always present the non-intersecting paths on the lecture hall graph
starting at $(n-i,t-1/(n-i+1))$ and ending at $(\lambda_i+n-i,0)$ or on the dual lecture hall graph.
We give the example for $\lambda=(n,n-1,\ldots ,2,1)$ on Figure \ref{staircase} and for 
$\lambda=(n,\ldots ,n)$ on Figure \ref{square}. Here we set $t=n$.
As we can see on these figures, we have several types of behavior: some regions are {\em frozen}, i.e. are empty or are filled with vertical paths
and some regions are {\em liquid}, i.e. seem random. Some other things that we can see is that the separation from the liquid to the the frozen
region is always sharp and the shape of this separation is always given by the trajectory of one of the paths on the lecture hall graph or on the dual lecture hall graph.
We will  use this observation to compute the curves separating the regions in Section \ref{tangent}.
On Figures \ref{staircase} and \ref{square} we draw the conjectural curve that separates the frozen and liquid regions.
As one can guess, these seem to be a semicircle and an ellipse. In Sections \ref{tangent} and \ref{dimer}
we will explain how to compute  these curves.

\begin{figure}
\includegraphics[width=6cm]{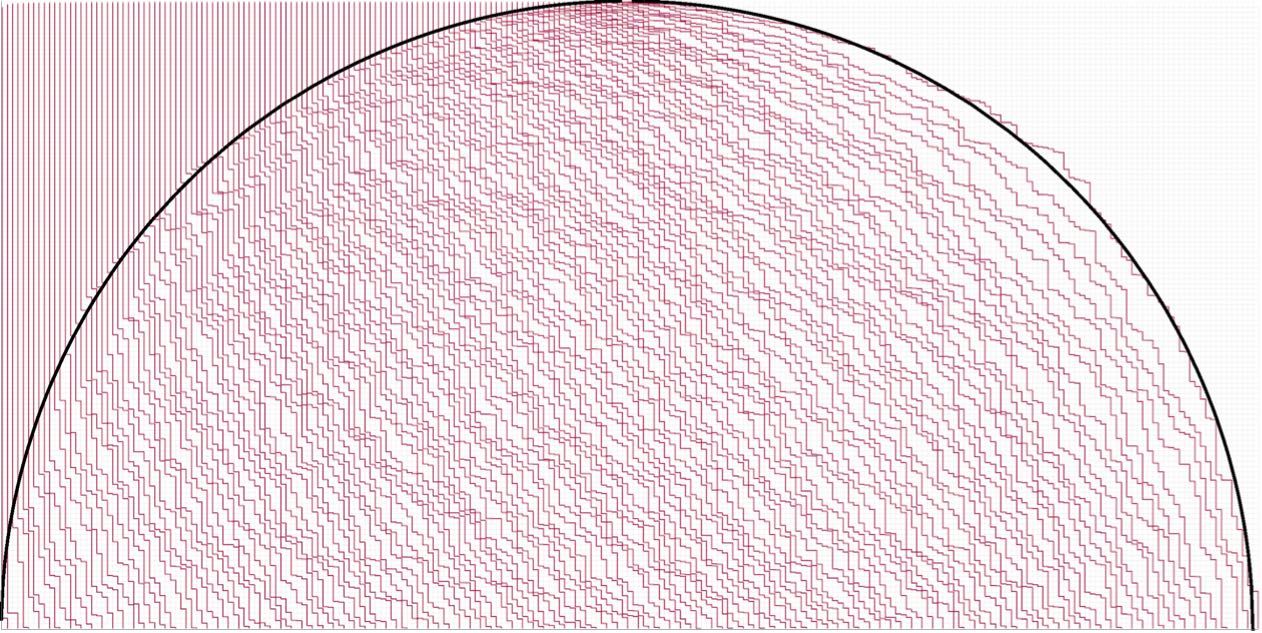}
\caption{Simulation for $\lambda=(n,n-1,\ldots ,2,1)$, $t=n=120$}
\label{staircase}
\end{figure}

\begin{figure}
\includegraphics[width=6cm]{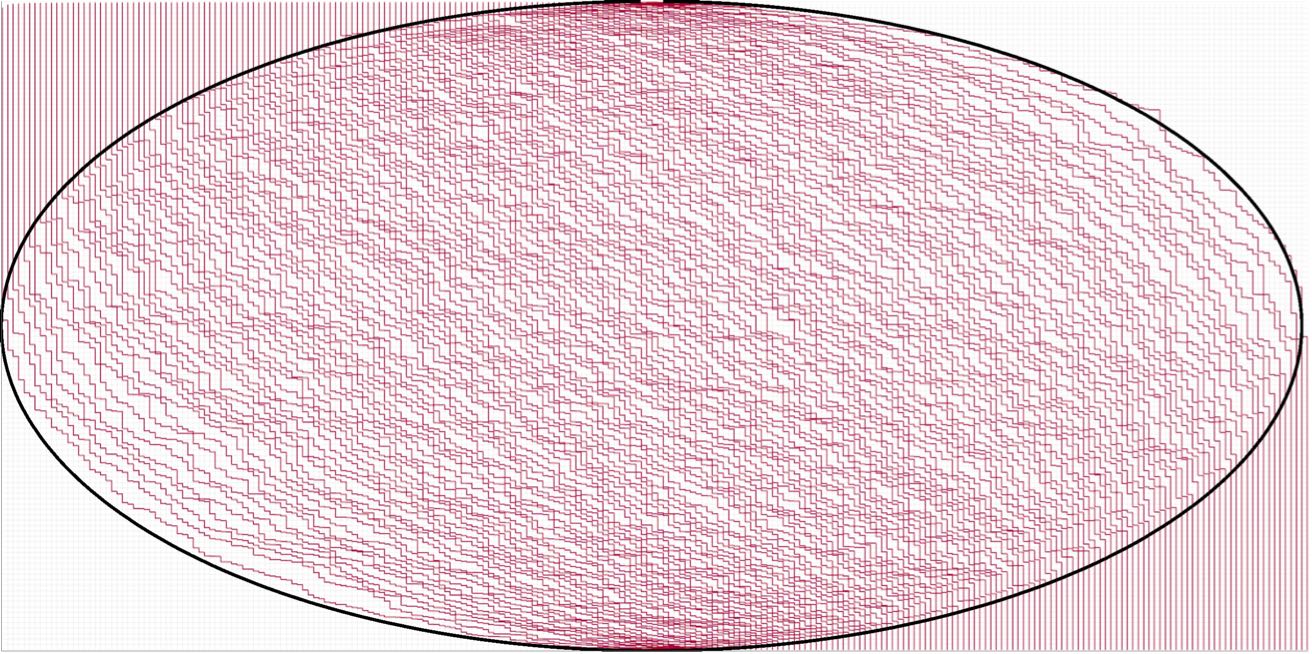}
\caption{Simulation for $\lambda=(n,n,\ldots ,n,n)$, $t=n=120$}
\label{square}
\end{figure}

\section{Tangent Method}
\label{tangent}

Developed in \cite{CS} the tangent method provides a simple way to compute arctic curves for models that can be described as a configuration of non-intersecting paths. Rather than computing bulk correlation functions to determine the boundary between the ordered and disordered phases, the tangent method requires only the computation of a boundary one-point function. 

We will explain the procedure in the framework of lecture hall tableaux. Consider a lecture hall tableaux configuration drawn as a set of non-intersecting paths. Suppose that the boundary between the frozen and liquid regions is given by the trajectory of one of the paths. That is, in the thermodynamic limit a portion of the arctic curve follows the expected value of the position of this boundary path. Call the endpoint of this path along the bottom boundary $p_0$. The assumption behind the tangent method states that as we vary $p_0$, the path will follow its original trajectory, along the arctic curve, until it can travel in a straight line to its endpoint. Moreover, this line will be tangent to the arctic curve. Together these are known as the {\em tangency assumption} of the tangent method.

In practice, we extend our domain to allow our path to end at a new point shifted horizontally by $q$ and lowered by $s$ from its original endpoint.  We then calculate the  expected location at which the path exited the original domain. Between these two points (the end point and the exit location) the path travels in a region empty of all other paths and its most probable free trajectory (or ``geodesic") is determined by the underlying graph. If the underlying graph is translationally invariant, the geodesics are straight lines. As a potentially surprising further analogy between lecture-hall tableaux and ordinary SSYT, and a justification \emph{a posteriori} of our choice of embedding of the graph, it turns out that the geodesics are straight lines also in our setting. The tangency assumption then states that the path will continue along this geodesic until meeting the arctic curve, and that the geodesic will in fact be tangent to the arctic curve.

If the tangent method holds (that is, the tangency assumption is true), then as we vary $q$ we obtain a family of lines forming an envelope of the arctic curve. From this envelope one can obtain a parameterization  of the arctic curve itself. In what follows, we will employ this method to derive parameterizations for the arctic curve for BLHT.

In this section, we first compute the asymptotic shape of a single path as a warm-up (and justification of the claim above). We then compute a parameterization the arctic curve determined by the outer most paths. Finally, we consider the arctic curve at a so called ``freezing boundary". The main result of this section is equation (\ref{thm:main}).

\subsection{Single Path}\label{subsec:sp} 
Suppose we have a single path starting at $(n-1, t-\frac{1}{n})$ and ending at $(n+k-1,0)$, corresponding to a single lecture hall partition. Recall that $Z_{n,k}(t) =  \binom{n+k-1}{k} t^k$ gives the number of paths with the given starting point and ending point. We want to know where the path crosses a horizontal slice at height $s$. To this end we divide the full path into two paths, one starting at $(n-1, t-\frac{1}{n})$ and ending at $(n+l-1,s)$ and the other one starting at $(n+l-1, s-\frac{1}{n+l})$ and ending at $(n+k-1,0)$. We rewrite the partition function for the full path as the product of the partition functions for the two partial paths, summing over all possible intermediate points $l$. We have
\begin{equation} \label{eq:singlepath}
\begin{aligned}
Z_{n,k}(t)  =&  \sum_{l=0}^k Z_{n,l}(t-s)\; Z_{n+l,k-l}(s) \\
=& \sum_{l=0}^k \binom{n+l-1}{l}(t-s)^l \binom{n+k-1}{k-l} s^{k-l}.
\end{aligned}
\end{equation} 
As an exercise in the type of calculations that will follow we prove the following:
\begin{prop}
Consider the limit $N\to \infty$, for parameters scaling as $n=N \nu$, $k = N \kappa$, $t =N \tau$, $s=N\sigma$, $l= N \lambda$, in which all Greek letters are $O(1)$.  Asymptotically the single path travels in a straight line between its starting and ending points.
\end{prop}
\begin{proof}
Recall Stirling's approximation for the factorial gives
\[
x! = \sqrt{2\pi x} x^x e^{-x} \left(1 + O\left(\frac{1}{x}\right)\right)
\] as $x\to \infty$.
We can use this to approximate the binomial coefficients in equation (\ref{eq:singlepath}) in the large $N$ limit. Doing so, we have asymptotically
\begin{equation}
Z_{N\nu,N\kappa}(N\tau)  \sim \frac{1}{2\pi}e^{N\kappa\,\ln(N)} \int_0^\kappa d\lambda \sqrt{\frac{\nu+\kappa}{\lambda \nu(\kappa-\lambda)}} e^{N \, S(\lambda)}
\end{equation}
where
\begin{equation*}
S(\lambda) = \lambda\,\ln(\tau-\sigma) + (\kappa-\lambda)\ln(\sigma) + (\nu +\kappa)\ln(\nu+\kappa) - \lambda\,\ln(\lambda) - \nu\,\ln(\nu) - (\kappa-\lambda)\ln(\kappa-\lambda).
\end{equation*}

This integral can be approximated via Laplace's method. Note that $S$ has its only critical point at
 \[
\frac{(\tau-\sigma)(\kappa-\lambda^*)}{\sigma \lambda^*} = 1
\]
or, rearranging,
\begin{equation} \label{eq:sp}
\lambda^* = \kappa \left(1 - \frac{\sigma}{\tau}\right).
\end{equation}
It is easy to check that this is a maximum.

Let $X,Y$ be coordinates on the rescaled domain $[0, \nu+\kappa]\times [0,\tau]$. Since the integrand is exponentially suppressed away from $\lambda^*$, the most likely position at which our path crosses the slice $Y=\sigma$ is at  $X=\nu+\lambda^*$. Equation (\ref{eq:sp}) then says that most likely path connecting the points $(\nu,\tau)$ to $(\nu+\kappa,0)$, follows the straight line
\[
Y = -\frac{\tau}{\kappa}X+\frac{\tau}{\kappa}(\nu+\kappa).
\]
\end{proof}

\subsection{Outer Boundary} 
Here let us consider the outer boundary of the arctic curve; that is, a section where the boundary between the frozen and disordered region is given by the trajectory of the first or the last path (in either path description). The following analysis is very similar to that of \cite{DG}.

Let $\lambda = (\lambda_1,\ldots ,\lambda_n)$, and consider the bounded lecture hall tableaux of shape $\lambda$ and height $t$. Recall that there is a bijection between the number of BLHT and configurations of non-intersecting, down-right paths with starting points $(n-i,t-\frac{1}{n+1-i})$ and ending points $(n+\lambda_j -j,0)$. 

In what follows, by the $k^{th}$ path we mean the path starting at $(n-k,t-\frac{1}{n+1-k})$ and ending at $(n+\lambda_k -k,0)$. Equivalently, this is the path corresponding to $k^{th}$ row of $\lambda$.

Recall
\begin{align} \label{eq:LHTpf}
Z_\lambda(t)
& = t^{|\lambda|} \prod_{1\le i < j \le n} \frac{\lambda_i - \lambda_j +j-i}{j-i}.
\end{align}
To shorten the notation we write $Z=Z_\lambda(t)$. 

\subsubsection{First Path}
In order to use the tangent method we consider the possible configurations of paths in which the first path ends at the point $(n+\lambda_1+q,-s)$. Let the partition function for this model be called $Z_{qs}$. Define $Z_r$ to be the partition function for BLHT of shape $\mu = (\lambda_1+r,\lambda_2,\ldots,\lambda_n)$. In terms of non-intersecting paths, $Z_r$ is the partition function for configurations in which we have shifted the end point of the first path to the right by $r$.  See Figure \ref{fig:semicircleDiag} for a diagram. With this $Z_{qs}$ can be written as the sum over $r$ of $Z_r$ times the partition function of a single path starting at $(n+\lambda_1 + r - 1,-\frac{1}{n+\lambda_1 + r})$ and ending at $(n+\lambda_1+q-1,-s)$.  We normalize by $Z$ (see eqn. (\ref{eq:LHTpf})) to get
\begin{equation}  \label{eq:pfratio2}
\frac{Z_{qs}}{Z}  = \sum_{r=0}^q \frac{Z_r}{Z} s^{q-r} \binom{n+\lambda_1-1+q}{q-r}.
\end{equation}

\begin{figure}[!htb]
\includegraphics[width=\linewidth]{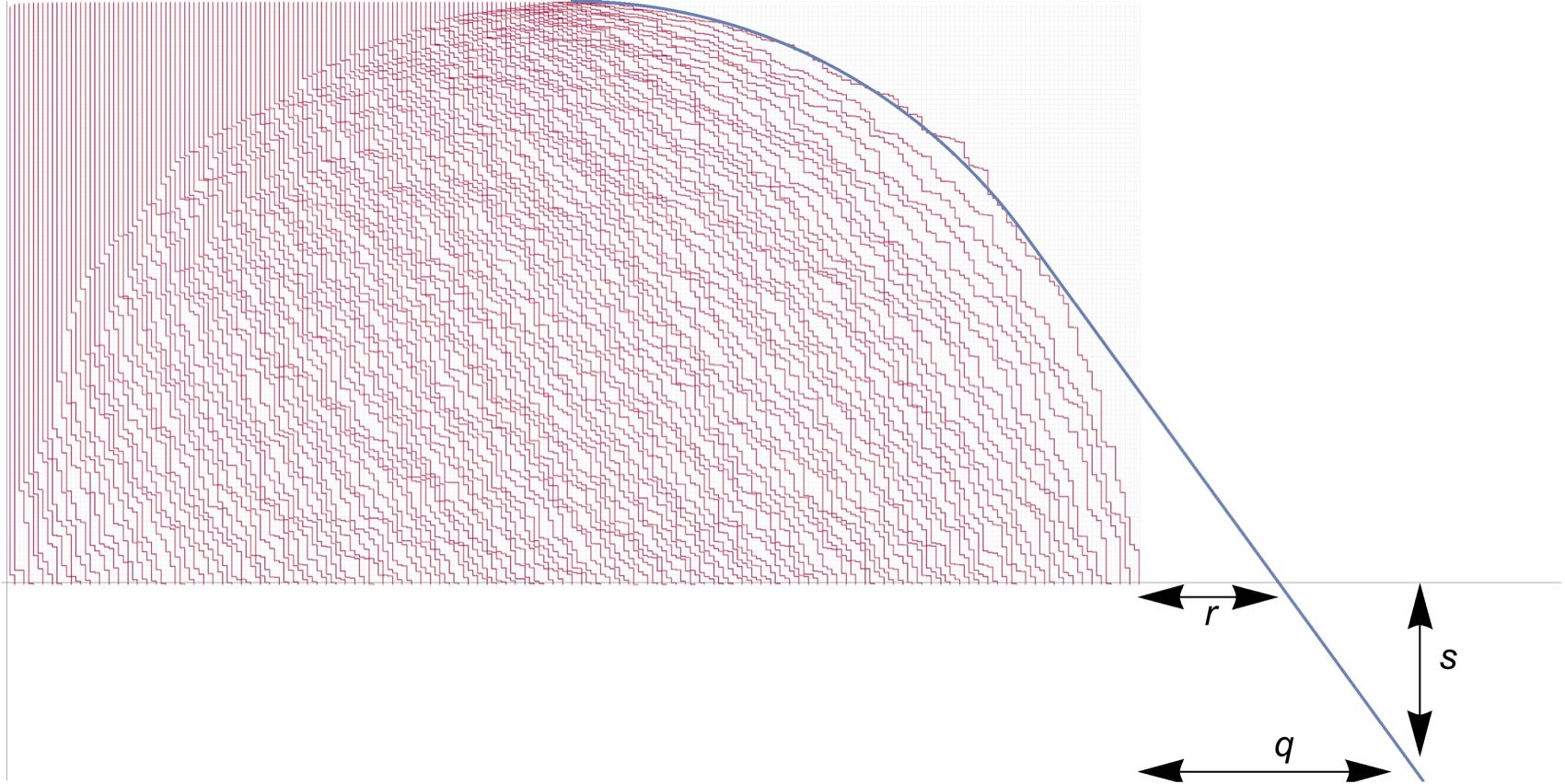}
\caption{A configuration of paths corresponding to a BLHT of shape $\lambda =(n,n-1,\ldots,1)$. The blue curve schematically shows the extended first path ending at $(n+\lambda_1+q-1,-s)$ and passing through $(n+\lambda_1+r-1,0)$.} \label{fig:semicircleDiag}
\end{figure}

From the product formula (\ref{eq:LHTpf}), the ratio of $Z_r$ and $Z$ takes the simple form
\begin{align} \label{eq:pfratio1}
\frac{Z_r}{Z}  & = t^{|\mu| - |\lambda|} \prod_{1\le i < j \le n} \frac{\mu_i - \mu_j +j-i}{\lambda_i - \lambda_j +j-i} \nonumber \\
& = t^r\prod_{i=1,2\le j \le n} \frac{\lambda_1+r - \lambda_j +j-1}{\lambda_1 - \lambda_j +j-1} \prod_{2\le i < j \le n} \frac{\lambda_i - \lambda_j +j-i}{\lambda_i - \lambda_j +j-i} \nonumber \\
& = t^r \prod_{j=2}^n \frac{\lambda_1+r - \lambda_j+j-1}{\lambda_1 - \lambda_j+j-1}.
\end{align}

Using equation (\ref{eq:pfratio1}) we can rewrite equation (\ref{eq:pfratio2}) as
\begin{align} \label{eq:pfratio3}
\frac{Z_{qs}}{Z} & = \sum_{r=0}^q \frac{Z_r}{Z} s^{q-r} \binom{n+\lambda_1-1+q}{q-r} \nonumber \\
& = \sum_{r=0}^q t^r s^{q-r} \prod_{j=2}^n \frac{\lambda_1+r-\lambda_j+j-1}{\lambda_1-\lambda_j +j-1} \binom{n+\lambda_1-1+q}{q-r} \nonumber \\
& =  \sum_{r=0}^q t^r s^{q-r} \prod_{j=2}^n \frac{a_1-a_j +r}{a_1-a_j} \binom{a_1+q}{q-r}
\end{align} 
where $a_i = n+\lambda_i-i$. 

Note that what we have done is write $Z_{qs}$ as a sum over the possible ways the first path can cross the horizontal slice $y=0$. The following lemma describes the most likely location the path will cross this slice as a function of the limiting ratios of $q$ and $s$ to $n$.

\begin{lem} \label{lem:obz}
Consider the limit $n\to \infty$, for parameters scaling as $t= n \tau$, $s= n \sigma$, $r= n \rho$, $q = n z$, and $a_i = \lfloor n \alpha\left(\frac{i}{n}\right)\rfloor$. In this limit, the first path passes through the point $(\alpha(0)+\rho,0)$, where $\rho$ is related to $z$ by
\[
z = \frac{\sigma}{\tau} (\alpha(0) +\rho) e^{-\int_0^1 du \; \frac{1}{\alpha(0)-\alpha(u)+\rho}} + \rho.
\]
\end{lem}
\begin{proof}
Taking the above limit in equation (\ref{eq:pfratio3}), the binomial coefficient can be approximated using Stirling's approximation giving
\[
 \binom{a_1+q}{q-r} \sim \frac{1}{\sqrt{2\pi n}}\sqrt{\frac{\alpha(0)+z}{(\alpha(0)+\rho)(z-\rho)}} e^{n \left((\alpha(0)+z) \ln(\alpha(0)+z) - (\alpha(0)+\rho) \ln(\alpha(0)+\rho) -(z-\rho)\ln(z-\rho)\right) },
\]
while the product term can be written 
\[
\prod_{j=2}^n \frac{a_1-a_j +r}{a_1-a_j} = e^{n \left(\frac{1}{n} \sum_{j=2}^n \ln\left( \frac{a_1-a_j +r}{a_1-a_j} \right)\right) } \sim e^{n\int_0^1 du \; \ln\left( \frac{\alpha(0)-\alpha(u) +\rho}{\alpha(0)-\alpha(u)} \right)}.
\] 

All together equation (\ref{eq:pfratio3}) becomes
\begin{align} \label{eq:outerpflimit}
\frac{Z_{qs}}{Z} \sim \sqrt{\frac{n}{2\pi}} e^{nz\; \ln(n)} \int_0^z d\rho \; \sqrt{\frac{\alpha(0) + z}{(\alpha(0)+\rho)(z-\rho)}} e^{n S(\rho)}
\end{align}
where
\[
\begin{split}
S(\rho)  = \rho \; \ln(\tau) + (z-\rho) \ln(\sigma) - (\alpha(0)+\rho)\ln(\alpha(0)+\rho) -(z-\rho) \ln(z-\rho) \\
+ \int_0^1du \; \ln\left( \frac{\alpha(0)-\alpha(u) + \rho}{\alpha(0)-\alpha(u)}\right).
\end{split}
\]

As in Section (\ref{subsec:sp}), the integral is dominated by the contributions from the maximum of $S(\rho)$. This critical point occurs when
\[
\frac{\tau(z-\rho)}{\sigma (\alpha(0) + \rho)} e^{\int_0^1 du \; \frac{1}{\alpha(0)-\alpha(u)+\rho}} =1
\]
or rearranging  
\[
z = \frac{\sigma}{\tau} (\alpha(0) +\rho) e^{-\int_0^1 du \; \frac{1}{\alpha(0)-\alpha(u)+\rho}} + \rho.
\]
\end{proof}

Asymptotically the first path passes through the points $(\alpha(0)+\rho,0)$ and $(\alpha(0)+z,-\sigma)$, where we know $\rho$ in terms of $z$ (and vice versa) from Lemma \ref{lem:obz}. By varying $z$ (or more conveniently $\rho$), we obtain a family of lines which, according to the tangent method, form an envelope of the arctic curve. From these lines we will construct a parameterization of the curve. Let $X$ and $Y$ be continuum coordinates on the domain $[0,\alpha(0)]\times[0,\tau]$ into which we embed our collection of paths. We show
\begin{thm} \label{thm:ob}
Assuming the tangent method holds, the portion of the arctic curve following the first path can be parameterized by
\begin{align}
\begin{split}
X(x) = \frac{x^2I'(x)}{I(x) + x I'(x)} \\
Y(x) =\tau \frac{1}{I(x) + x I'(x)}
\end{split}
\end{align}
with $x\in [\alpha(0),\infty)$.
\end{thm}
\begin{proof}
From the two points $(\alpha(0)+\rho,0)$ and $(\alpha(0)+z,-\sigma)$  we have the line
\[
Y = -\frac{\sigma}{z-\rho} (X-\alpha(0)-\rho).
\]
Using Lemma \ref{lem:obz} to eliminate $z$, this becomes
\[
Y = -\frac{\tau}{\alpha(0)+\rho} e^{\int_0^1 du \; \frac{1}{\alpha(0)+\rho - \alpha(u)}} (X-\alpha(0)-\rho)
\]
Rearranging, and letting $x = \alpha(0) + \rho$, we have
\begin{equation} \label{eq:obline}
\frac{x}{\tau} I(x) Y + X-x = 0
\end{equation}
where $I(x) = e^{-\int_0^1 du \; \frac{1}{x-\alpha(u)}}$. Note that $\rho \in [0,\infty)$ implies $x = \alpha(0)+\rho\in[\alpha(0),\infty)$.

Taking the derivative of equation (\ref{eq:obline}) with respect to $x$ gives the system of equations for $(X,Y)$
\begin{align*}
& \frac{x}{\tau}I(x)Y+X-x=0 \\
&\frac{1}{\tau} (I(x)+xI'(x))Y -1 =0
\end{align*}
which can be solved to yield the desired parameterization.
\end{proof}

In what follows, we reuse much of the notation from the preceding section.

\subsubsection{Last Path}
We can consider the same calculation as above on a portion of the arctic curve which follows the last path in the thermodynamic limit. For such a section of the arctic curve to exist, we must have that $\lambda_n$ is of size proportional to $n$. This means that the limiting profile must satisfy $\alpha(1)>0$.  

Suppose $\lambda$ is such a partition. We first consider the case when the endpoint of the last path is shifted to the left by $r$, that is, the path ends at $(\lambda_n-r,0)$, with $r\in[0,\lambda_n]$. In the same manner as the previous section we have
\begin{equation} \label{eq:obfpratio}
\frac{Z_r}{Z} = t^{-r} \prod_{i=1}^{n-1} \frac{\lambda_i-\lambda_n+r+n-i}{\lambda_i-\lambda_n+n-i}.
\end{equation}

Now suppose the last path ends at a point $(\lambda_n +q , -s)$, for $q\ge - \lambda_n$.  Write the total number of configurations as $Z_{qs}$. As in the previous section, using equation (\ref{eq:obfpratio}), we write this partition function as
\begin{align}
\frac{Z_{qs}}{Z} & = \sum_{r=0}^{\lambda_n} \frac{Z_r}{Z} s^{q+r} \binom{\lambda_n+q}{q+r} \nonumber \\
& = \sum_{r=0}^{\lambda_n} t^{-r} s^{q+r} \prod_{i=1}^{n-1} \frac{a_i-a_n-r}{a_i-a_n} \binom{a_n +q}{q+r}.
\end{align} 
See Figure \ref{fig:circleDiag}.

\begin{figure}[!htb]
\includegraphics[width=\linewidth]{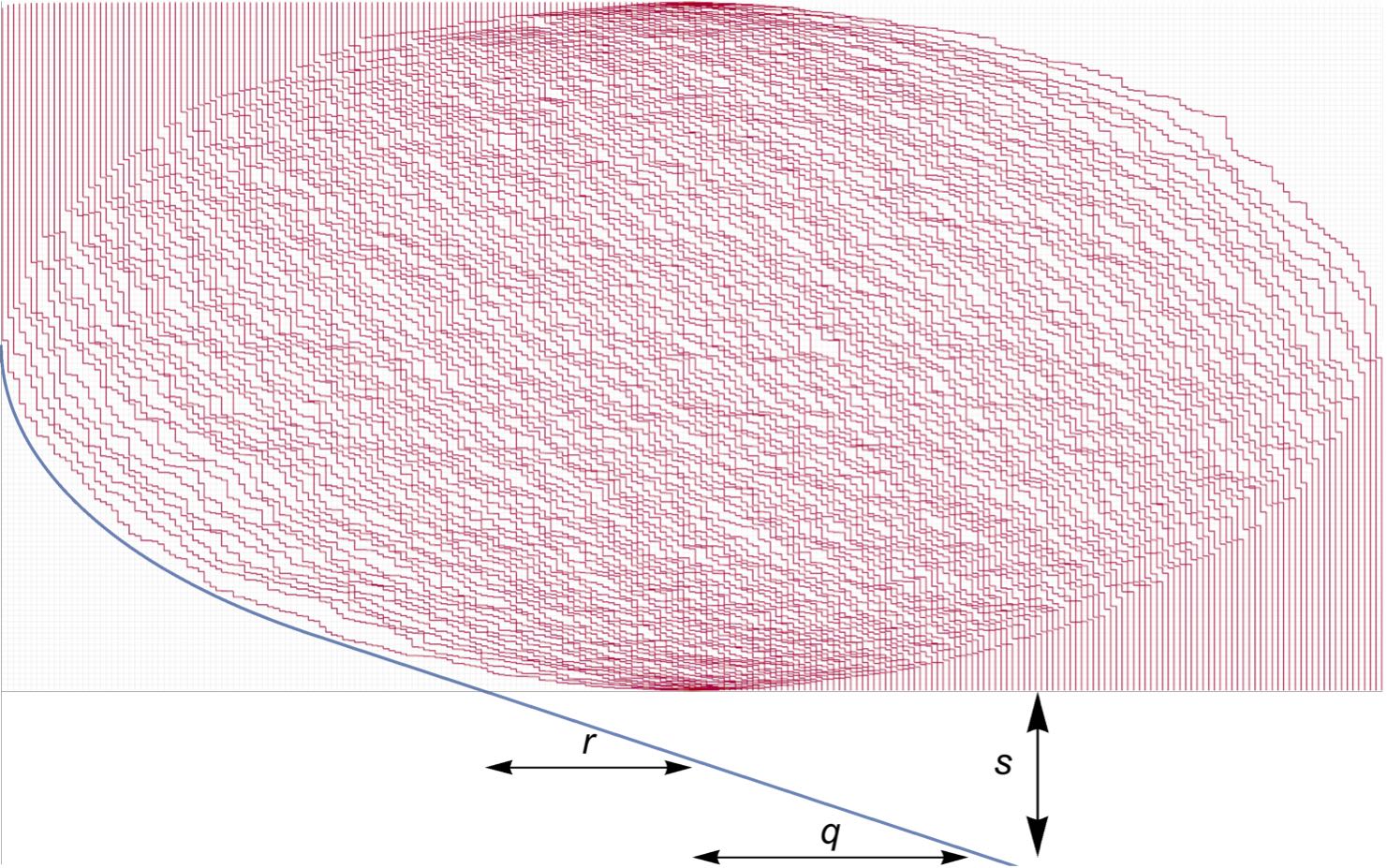}
\caption{A configuration of paths corresponding to a BLHT of shape $\lambda = (n,\ldots,n)$. The blue curve schematically shows the trajectory of the extended last path which ends at $(\lambda_n+q,-s)$ and passes through $(\lambda_n-r,0)$.} \label{fig:circleDiag}
\end{figure}

\begin{lem}
Consider the limit $n\to \infty$, for parameters scaling as $t=N\tau$, $s = n\sigma$, $r=n\rho$, $q=n z$, and $a_i=\lfloor n\alpha\left(\frac{i}{n}\right)\rfloor$. In this limit, the last path passes through the point $(\alpha(1)-\rho,0)$, where $\rho$ is related to $z$ by
\begin{equation} \label{eq:obz2}
z =\frac{\sigma}{\tau} (\alpha(1) - \rho)e^{-\int_0^1 du \frac{1}{\alpha(1)-\rho-\alpha(u)}}-\rho.
\end{equation}
\end{lem}
\begin{proof}
In the above limit, we have
\[
\frac{Z_{qs}}{Z} \sim \sqrt{\frac{n}{2\pi}}e^{nz\; \ln(n)} \int_0^{\alpha(1)} d\rho \sqrt{\frac{\alpha(1)+z}{(\alpha(1)-\rho)(z+\rho)}}e^{n\;S(\rho)}
\]
where
\[
\begin{split}
S(\rho) = -\rho\; \ln(\tau) +(z+\rho)\ln(\sigma) - (\alpha(1) - \rho) \ln(\alpha(1)-\rho) - (z+\rho)\ln(z+\rho) \\
+ \int_0^1 du \; \ln\left( \frac{\alpha(u)-\alpha(1)+\rho}{\alpha(u)-\alpha(1)} \right)
\end{split}
\]
The critical point occurs when 
\[
\frac{\tau}{\sigma} \frac{z+\rho}{\alpha(1)-\rho} e^{\int_0^1 du \frac{1}{\alpha(1)-\rho-\alpha(u)}} = 1.
\]
Rearranging  gives the desired result.
\end{proof}

Using this we get a parameterization of this section of the arctic curve.
\begin{thm}
Assuming the tangent method holds, the portion of the arctic curve following the last path is parameterized by
\begin{align}
\begin{split}
X(x) = \frac{x^2I'(x)}{I(x) + x I'(x)} \\
Y(x) = \tau \frac{1}{I(x) + x I'(x)}
\end{split}
\end{align}
with $x\in [0,\alpha(1)]$.
\end{thm}
\begin{remark}
Note this is the same parameterization as Theorem \ref{thm:ob} for a different range of parameter, as is normally the case for dimer models (and \emph{not} the case for models which are not free-fermionic). For $\ZZ^2$ invariant graphs, this is a theorem of Kenyon and Okounkov \cite{KO06}. We will see that we the same parameterization works for all portions of the arctic curve of the BLHT, despite the fact that the lecture hall graph is not $\ZZ^2$ invariant. 
\end{remark}

\begin{proof}
From the points $(\alpha(1)+z,-\sigma)$, $(\alpha(1)-\rho,0)$, and equation (\ref{eq:obz2}) we have a family of lines
\[
x\;I(x)Y +X-x=0
\] with $x=\alpha(1)-\rho$ and $I(x) = \frac{1}{\tau} e^{-\int_0^1 du \frac{1}{x-\alpha(u)}}$. Taking the derivative with respect to $x$, we get the system of equations
\begin{align*}
& xI(x)Y+X-x=0 \\
& (I(x)+xI'(x))Y -1 =0
\end{align*}
which can be solved to yield the desired parameterization. Note that since $\rho \in [0,\alpha(1)]$, the range of $x$ is $[0,\alpha(1)]$ as well.
\end{proof}

Before moving on, we prove the following proposition. Recall the quantities $T_{i j}/(n + j-i)$ in the definition of a lecture hall tableaux depend on $n$, so apriori the lecture hall tableaux of shape $(\lambda_1, \dots, \lambda_n)$ and $(\lambda_1, \dots, \lambda_n, 0, \dots, 0)$ are not the same.
\begin{prop} \label{prop:translation}
Assuming the tangent method holds, the arctic curve is unchanged when we extend the partition $\lambda = (\lambda_1,\ldots,\lambda_n)$ to $(\lambda_1,\ldots,\lambda_n,0,\ldots,0)$, where we add $m$ parts of size zero and $m$ scales as $m=nM$ in the thermodynamic limit.
\end{prop}
\begin{proof}
We'll show this for the portion of the arctic curve following the first path. After extending $\lambda$, equation (\ref{eq:pfratio1}) becomes
\begin{align}
\frac{Z_r}{Z} & = t^r \prod_{j=2}^{n+m} \frac{\lambda_1 - \lambda_j +j - 1+r}{\lambda_1 - \lambda_j+j-1} \nonumber \\
& = t^r \prod_{j=2}^{n+m} \frac{a_1 - a_j + r}{a_1-a_j}
\end{align}
where now the products range from 2 to $n+m$. With this equation (\ref{eq:pfratio2}) becomes
\begin{equation}
\frac{Z_{qs}}{Z} = \sum_{r=0}^q t^rs^{q-r} \prod_{j=2}^{n+m} \frac{a_1 - a_j + r}{a_1-a_j} \binom{m+a_1+q}{q-r}.
\end{equation}
In the thermodynamic limit (with $m=n M$, and the rest as before), this is dominated by the maximum of
\[
\begin{split}
S(\rho) = \rho\; \ln(\tau) + (z-\rho)\; \ln(\sigma) -(z-\rho)\; \ln(z-\rho) - (M+\alpha(0)+\rho)\ln(M+\alpha(0)+\rho) \\
+\int_0^1 du\; \ln\left( \frac{\alpha(0)+\rho-\alpha(u)}{\alpha(0)-\alpha(u)}\right) + \int_1^{1+M} du \; \ln\left(\frac{\alpha(0)+\rho +u-1}{\alpha(0) +u-1} \right)
\end{split}
\]
This is given when
\[
\frac{\tau (z-\rho)}{\sigma (M+\alpha(0)+\rho)} e^{\int_0^1du\; \frac{1}{\alpha(0)+\rho -\alpha(u)} }\frac{M+\alpha(0)+\rho}{\alpha(0)+\rho} = 1.
\]
Letting $x=\alpha(0)+\rho$, this results in the same parameterization.

The other portion can be done similarly. See Figure \ref{fig:semicircleNormalDual}.
\end{proof}

\begin{figure} [!htb]
\begin{minipage}[c]{0.45\linewidth}
\includegraphics[width=\linewidth]{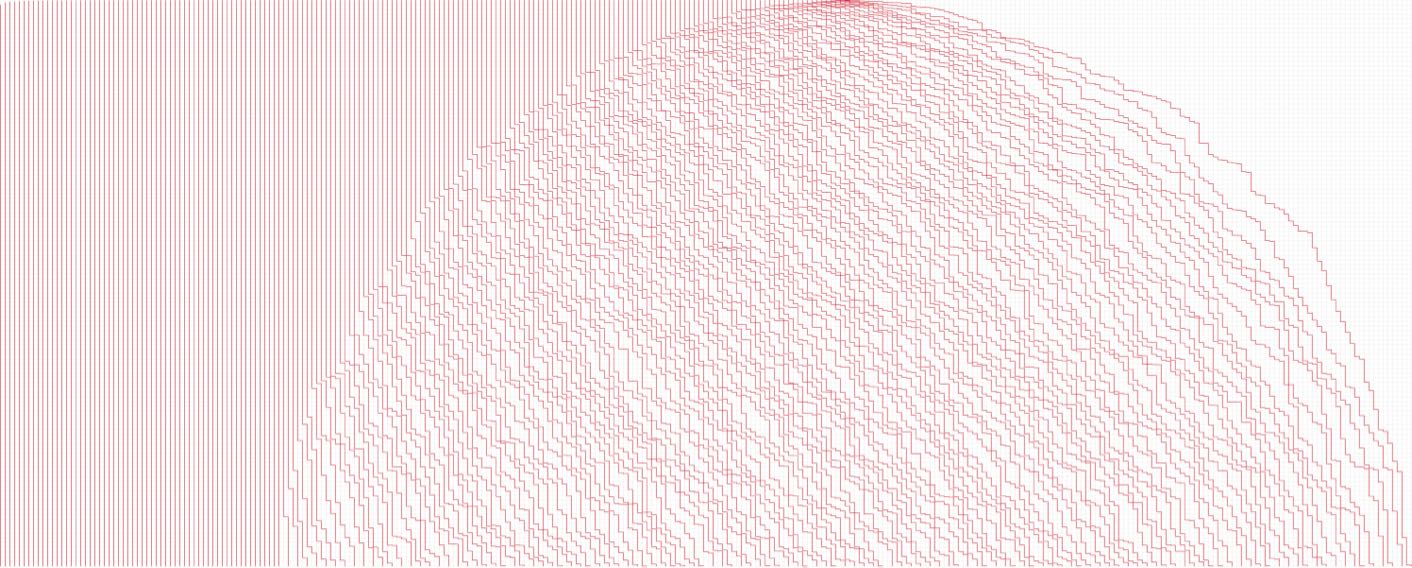}
\end{minipage}
\hfill
\begin{minipage}[c]{0.45\linewidth}
\includegraphics[width=\linewidth]{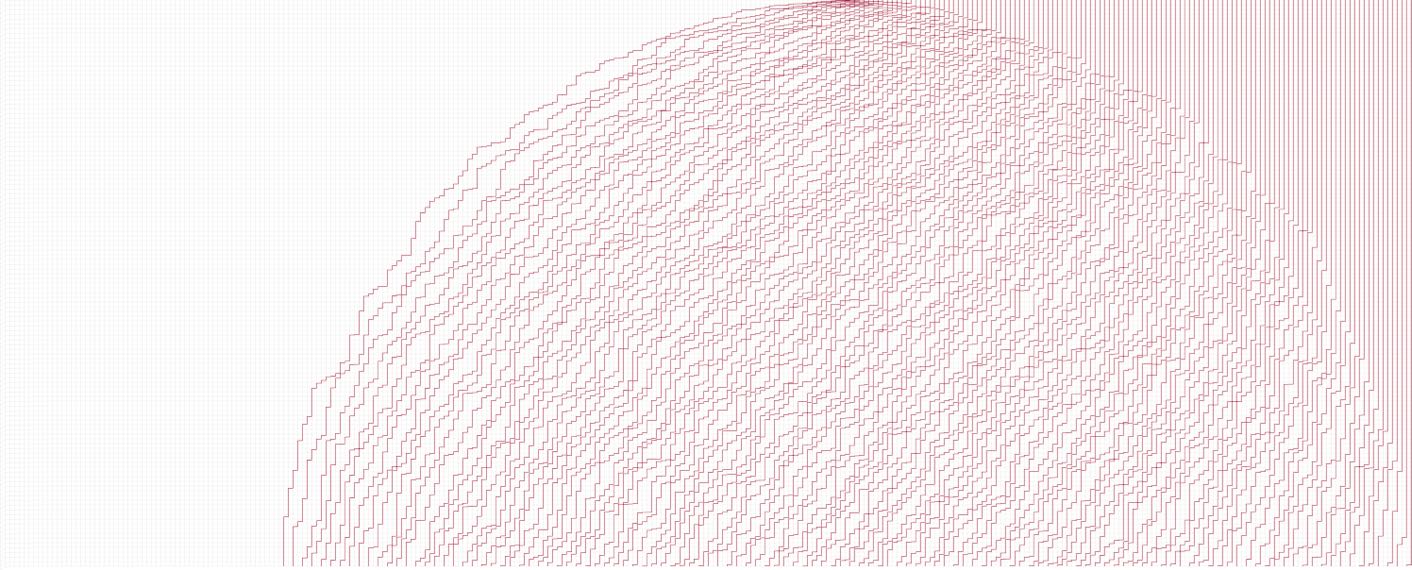}
\end{minipage}%
\caption{On the left, a configuration of paths for $\lambda=(n,n-1,\ldots,1,0,\ldots,0)$, where there are $n/2$ parts of size $0$. On the right, the corresponding dual paths.}
\label{fig:semicircleNormalDual}
\end{figure}

\subsection{Dual Path Formulation} 
Let's consider the analogous calculation in the dual path formulation. In the case, the paths begin at $u_i = (n-1+i-\lambda_i',-\frac{1}{(n+i-\lambda_i')^2})$ and end at $v_j=(n+j-1,t-\frac{1}{n+j-1})$. See Figure \ref{fig:semicircleNormalDual}. Recall that vertical edges that were previously empty now have a path, while vertical edges that previously had a path are now empty. This means that switching to the dual path formulations swaps frozen regions of no paths and frozen regions of vertical paths. In particular, the arctic curve remains the same. From this perspective, portions of the arctic curve on the boundary between an area of frozen vertical paths and a disordered region in the original formulation now follow the trajectory of one of the dual paths. For what follows below we assume $ \lambda $ has $n$ parts all greater than zero. In particular, $\lambda_1'=n$.

By the $k^{th}$ dual path we mean the path beginning at $(n-1+k-\lambda_k', -\frac{1}{(n+k-\lambda_k')^2})$ and ending at $(n+k-1,t-\frac{1}{n+k-1})$. Equivalently, this is the path corresponding to $k^{th}$ column of $\lambda$.

\subsubsection{First Dual Path} To calculate a parameterization for the portion of the arctic curve following the first dual path we must first extend $\lambda$ to $(\lambda_1,\ldots,\lambda_n,0,\ldots,0)$ where we've added $m$ parts of length zero. 

Let the number of BLHT with this shape be $Z$.  In the dual paths picture, moving the start point of the first dual path to the left by $r$ corresponds to changing the first of the $r$ zeros in the extension of our partition to ones; that is, the partition becomes
\[
\mu:=(\lambda_1,\ldots,\lambda_n,\underbrace{1,\ldots,1}_r,\underbrace{0,\ldots,0}_{m-r}).
\] 
We call this partition $\mu$ and the corresponding partition function $Z_r$. Using the product formula for the partition function of the BLHT, we have
\[
\frac{Z_r}{Z}  = t^r \prod_{1\le i<j\le n+m} \frac{\mu_i - \mu_j +j-1}{\lambda_i-\lambda_j +j-i}.
\]
Note that the only terms in the product that are not one come from $1\le i\le n, n<j\le n+r$ and $n<i\le n+r, n+r<j\le n+m$, and the product can be simplified to
\[
\frac{Z_r}{Z}  = t^r \prod_{1\le i\le n < j \le n+r} \frac{\lambda_i - 1 +j-i}{\lambda_i +j-i} \prod_{n< i\le n+r<j\le n+m} \frac{1+j-i}{j-i}.
\]
Fixing $i$, both products are telescoping, giving
\begin{equation*}
\begin{split}
\prod_{n<j\le n+r} \frac{\lambda_i - 1 +j-i}{\lambda_i +j-i}   = \frac{\lambda_i +n-i}{\lambda_i +n+r-i}, \\
 \prod_{ n+r<j\le n+m} \frac{1+j-i}{j-i} = \frac{n+m-1+i}{n+r+1-i}
\end{split}
\end{equation*}
from which we have
\begin{equation}
\frac{Z_r}{Z} = t^r \prod_{i=1}^n \frac{\lambda_i+n-i}{\lambda_i+n+r-i} \prod_{i=n+1}^{n+r} \frac{n+m+1-i}{n+r+1-i}.
\end{equation}

Now consider the possible configuration with the first dual path starting at $(-q,-s-\frac{1}{(m-q+1)^2})$, for some $q$ and $s$ such that $m>q\ge 0$ and $s>0$. See Figure \ref{fig:semicircledualDiag}. Call the partition function $Z_{qs}$. Summing over the possible ways the path can cross the $y=0$ slice, we have
\begin{align}\label{eq:qsfd}
\frac{Z_{qs}}{Z} & = \sum_{r=0}^q \frac{Z_r}{Z} \binom{m-r}{m-q} s^{q-r} \nonumber \\
& = \sum_{r=0}^q t^r s^{q-r}  \prod_{i=1}^n \frac{\lambda_i+n-i}{\lambda_i+n+r-i} \prod_{i=n+1}^{n+r} \frac{n+m+1-i}{n+r+1-i} \binom{m-r}{m-q}  \nonumber \\
& = \sum_{r=0}^q t^r s^{q-r}  \prod_{i=1}^n \frac{a_i}{a_i+r} \prod_{i=n+1}^{n+r} \frac{a_i+m}{a_i+r} \binom{m-r}{m-q} 
\end{align}
where $ \binom{m-r}{q-r} s^{q-r}$ counts the number of configurations of a single path from $(-q,-s-\frac{1}{(m-q+1)^2})$ to $(-r,-\frac{1}{m-r})$.

\begin{figure}[!htb]
\includegraphics[width=\linewidth]{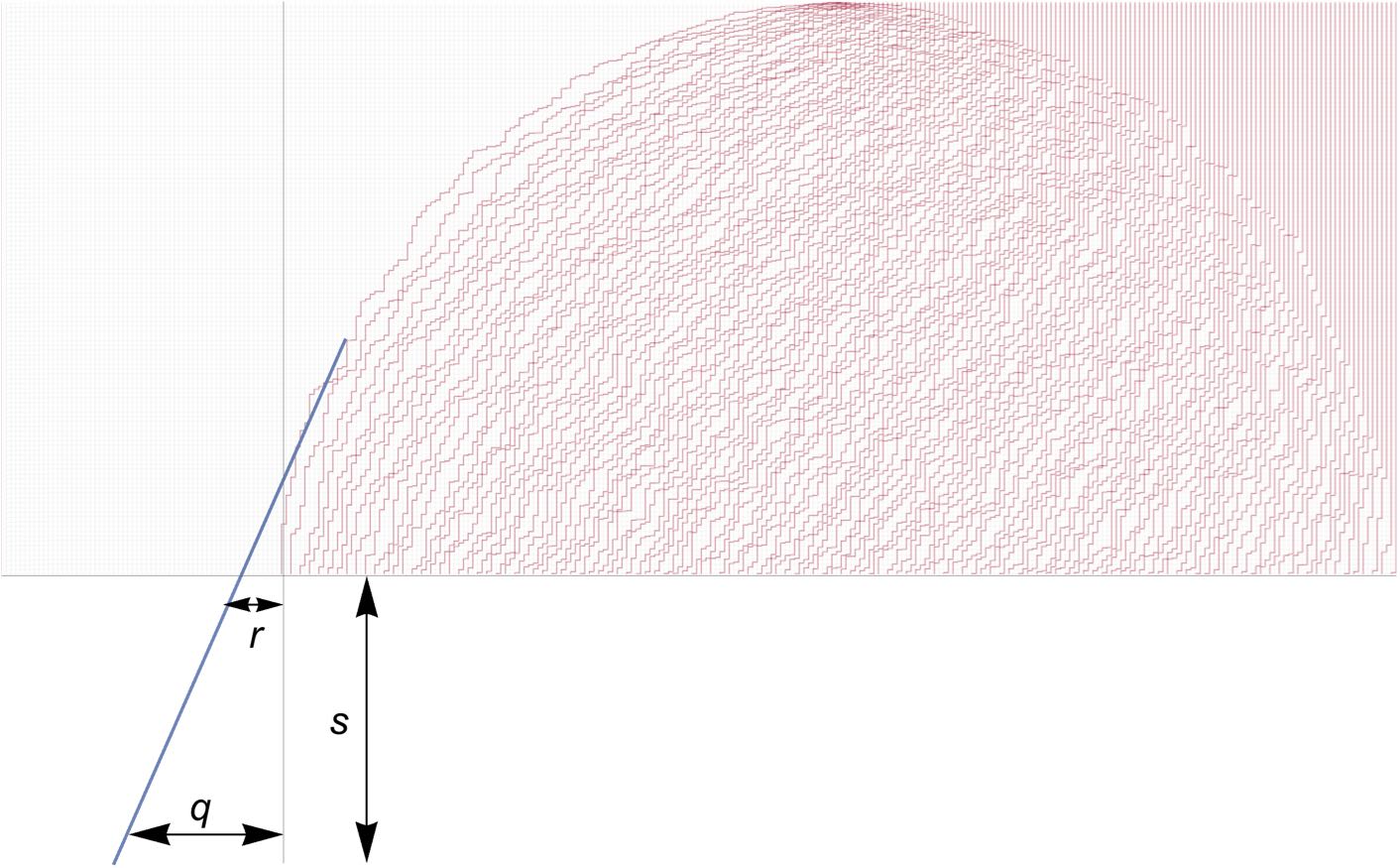}
\caption{A configuration of dual paths corresponding to a BLHT of shape $\lambda = (n,n-1,\ldots,1,0,\ldots,0)$ with $n/2$ parts of size 0. The blue curve represents the trajectory of the extended first dual path, which starts at $(-q,-s-\frac{1}{(n-q+1)^2})$ and passes through $(-r,0)$.}  \label{fig:semicircledualDiag}
\end{figure}

\begin{lem}
Consider the limit $n\to \infty$, for parameters scaling as $t=n\tau$, $s=n\sigma$, $r=n\rho$, $q=n z$, $m=n M$, and $a_i = \lfloor n\alpha\left(\frac{i}{n}\right)\rfloor$. In this limit, the first dual path passes through the point $(-\rho, 0)$, where $\rho$ is related to $z$ by
\begin{equation} \label{eq:obdz1}
z = \rho\frac{\sigma}{\tau} e^{\int_0^1 du \; \frac{1}{\alpha(u)+\rho}} + \rho
\end{equation}
\end{lem}
\begin{proof}
Taking the limit $n\to \infty$, the sum in equation (\ref{eq:qsfd}) above becomes an integral which is dominated by the maximum of
\[
\begin{split}
S(\rho) = \rho\; \ln(\tau) + (z-\rho)\ln(\sigma) + (M -\rho)\ln(M - \rho) - (z-\rho)\ln(z-\rho) \\
+ \int_0^1du\; \ln\left( \frac{\alpha(u)}{\alpha(u)+\rho}\right) + \int_1^{1+\rho}du\; \ln\left(\frac{\alpha(u)+M}{\alpha(u)+\rho}\right) \\ 
\end{split}
\]
Note that for $u\in(1,1+\rho]$, $\alpha(u) = 1-u$. Using this, we compute
\[
\begin{split}
 \int_1^{1+\rho}du\; \ln\left(\frac{\alpha(u)+M}{\alpha(u)+\rho}\right) = \int_1^{1+\rho}du\; \ln\left(\frac{1-u+M}{1-u+\rho}\right) \\
 = M\; \ln(M) - (M - \rho)\ln(M-\rho) - \rho\; \ln(\rho)
 \end{split}
\]
The maximum of $S$ occurs when
\[
\frac{\tau(z-\rho)}{\sigma \; \rho} e^{-\int_0^1du\; \frac{1}{\alpha(u)+\rho}} = 1
\]
Rearranging gives the desired result.
\end{proof}
Finally we have
\begin{thm}
Assuming the tangent method holds, the portion of the arctic curve following the first dual path is parameterized by
\begin{align}
\begin{split}
X(x) = \frac{x^2I'(x)}{I(x) + x I'(x)} \\
Y(x) = \tau \frac{1}{I(x) + x I'(x)}
\end{split}
\end{align}
with $x\in (-\infty,0]$.
\end{thm}
\begin{proof}
In the large $n$ limit, the shifted first path passes through the point $(-z,-\sigma)$ and $(-\rho,0)$.  These define the family of lines
\[
Y = \frac{\sigma}{z-\rho} (X+\rho)
\]
Using equation (\ref{eq:obdz1}), and rearranging the above, we have
\[
\frac{x}{\tau}I(x) Y +X-x =0 
\] 
where $x=-\rho\in(-\infty,0]$. Taking the derivative with respect to $x$, we get the system of equations for $(X,Y)$
\begin{align*}
& \frac{x}{\tau}I(x)Y+X-x=0 \\
& \frac{1}{\tau}(I(x)+xI'(x))Y -1 = 0
\end{align*}
which can be solved to yield the desired parameterization.
\end{proof}

\subsubsection{Last Dual Path}
In the case that $\lambda_n'$ is of size proportional to $n$, then this last dual path will be the boundary of a frozen region.  In terms of the original partition, $\lambda_n'$ being macroscopically large means that the first $\lambda_n'$ parts of $\lambda$ are equal to $\lambda_1$. 
Here, for simplicity, we assume $\lambda_1=n$, so that $\alpha(0) = \lim_{n\to\infty} 1+\frac{\lambda_1}{n} -\frac{1}{n} =2$.

We can repeat the above process varying the starting point of the last dual path instead of the first.  We extend $\lambda$ to $(\lambda_1,\ldots,\lambda_n,0,\ldots,0)$. Call the partition function $Z$. Next we decrease $\lambda_n'$ by $r<\lambda_n'$.  This varies the starting point of the $n^{th}$ dual path to the right by $r$. Call the partition function $Z_r$.  In terms of the original partition, this means taking $\lambda_i$ to $\lambda_i-1$ for each $i=\lambda_n',\ldots,\lambda_n'-r+1$. Call this new partition $\mu$. From the product formula (\ref{eq:LHTpf}) we have
\begin{align} \label{eq:LDPprod}
\frac{Z_r}{Z}  & = t^{|\mu|-|\lambda|}\prod_{1\le i<j\le n+m} \frac{\mu_i-\mu_j+j-i}{\lambda_i-\lambda_j+j-i}
\end{align}

The product can be simplified giving
\[
\begin{split}
\frac{Z_r}{Z} = t^{-r} \binom{\lambda_n'}{r} \prod_{\lambda_n'+1\le j \le n+m} \frac{(\lambda_{\lambda_n'} -1) - \lambda_j +j - \lambda_n'}{\lambda_{\lambda_n'-r+1} -\lambda_j + j -(\lambda_n'-r+1)}.
\end{split}
\]

Now extend the first path to start at $(2n-1-\lambda_n'-q,-s-\frac{1}{(2n-\lambda_n'-q)^2})$, with $q\in (-\lambda_n',\infty)$. See Figure \ref{fig:circledualDiag} for a diagram. Call the resulting partition function $Z_{qs}$. This can be written
\begin{align} \label{eq:qsld}
\frac{Z_{qs}}{Z} & = \sum_{r=\max(0,q)}^{\lambda_n'} \frac{Z_r}{Z}  s^{q+r} \binom{m+2n-1-\lambda_n'+r}{m+2n-1-\lambda_n'-q} \nonumber \\
& =  \sum_{r=\max(q,0)}^{\lambda_n'} t^{-r}s^{q+r}  \binom{\lambda_n'}{r}    \prod_{\lambda_n'+1\le j \le n+m} \frac{a_{\lambda_n'}-a_j-1}{a_{\lambda_n'-r+1} -a_j }\binom{m+2n-1-\lambda_n'+r}{m+2n-1-\lambda_n'-q}
\end{align}

\begin{figure}[!htb]
\includegraphics[width=\linewidth]{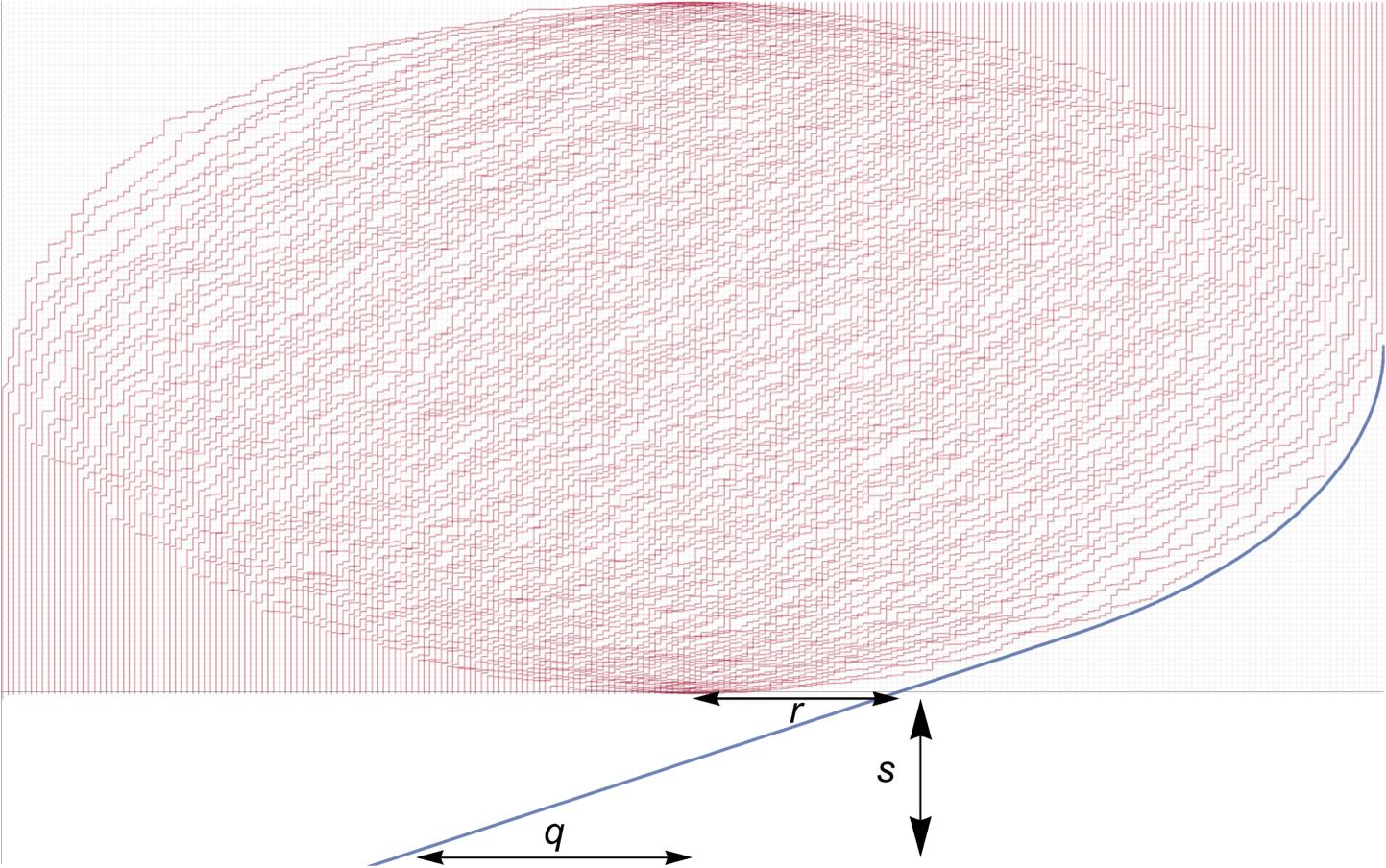}
\caption{A configuration of dual paths corresponding to a BLHT of shape $\lambda =(n,\ldots,n)$. The blue curve represents the trajectory of the extended last dual path, which starts at $(2n-\lambda_n'+q-1, -s-\frac{1}{(2n-\lambda_n'-q)^2})$ and passes through $(2n-\lambda_n'+r-1,0)$. } \label{fig:circledualDiag}
\end{figure}

\begin{lem}
Consider the limit $n\to \infty$, for parameters scaling as $t=n\tau$, $s=n\sigma$, $r=n\rho$, $q=nz$, $m=nM$, $\lambda_n' = n \Lambda$, and $a_i = \lfloor n\alpha(\frac{i}{n})\rfloor$. In this limit, the last dual path passes through the point $(\alpha(\Lambda)+\rho,0)$ where $\rho$ is related to $z$ by
\begin{equation} \label{eq:zdl}
z = \frac{\sigma}{\tau} (\alpha(\Lambda) +\rho) e^{-\int_0^1 du\; \frac{1}{\alpha(\Lambda)+\rho - \alpha(u)}} + \rho
\end{equation}
\end{lem}
\begin{proof}
Taking the limit $n\to \infty$, the sum (\ref{eq:qsld}) is dominated at the maximum of
\begin{equation} \label{eq:sdl}
\begin{split}
S(\rho) &= -\rho\; \ln(\tau) + (z+\rho)\ln(\sigma) + (M+2-\Lambda +\rho)\ln(M+2-\Lambda +\rho) -(z+\rho)\ln(z+\rho) \\
&  - (\Lambda-\rho) \ln(\Lambda-\rho) -\rho\; \ln(\rho) + \int_{\Lambda}^{1+M} du\; \ln\left( \frac{\alpha(\Lambda)-\alpha(u)}{\alpha(\Lambda-\rho) - \alpha(u)}\right)
\end{split} 
\end{equation}
where we have neglected the terms not depending on $\rho$ as they do not effect the location of the critical point.

We note that  $\alpha(u)=2-u$ when $u\in[0,\Lambda]$ and $\alpha(u)=1-u$ for $u\in[1,M]$. It follows that $\alpha(\Lambda)=2-\Lambda$ and $\alpha(\Lambda-\rho)=\alpha(\Lambda) + \rho$ for $\rho\ge 0$. Using these observations we can simplify (\ref{eq:sdl}). The portion of the integral in equation (\ref{eq:sdl}) over $u\in[1,1+M]$ can be expressed as
{\tiny
\begin{align*}
&\int_1^{1+M} du\; \ln\left( \frac{\alpha(\Lambda)-\alpha(u)}{\alpha(\Lambda-\rho) - \alpha(u)} \right) = \int_1^{1+M} du\; \ln\left( \frac{1+u-\Lambda}{1+u-\Lambda +\rho} \right) \\
&= (2-\Lambda+M) \ln(2-\Lambda+M) + (2-\Lambda+\rho)\ln(2-\Lambda+\rho)-(M+2-\Lambda+\rho)\ln(M+2-\Lambda+\rho) - (2-\Lambda)\ln(2-\Lambda).
\end{align*} 
} With this equation  (\ref{eq:sdl}) becomes
\begin{equation} \label{eq:sdl2}
\begin{split}
S(\rho) &= -\rho\; \ln(\tau) + (z+\rho)\ln(\sigma) + (2-\Lambda +\rho)\ln(M+2-\Lambda +\rho) -(z+\rho)\ln(z+\rho) \\
&  - (\Lambda-\rho) \ln(\Lambda-\rho) -\rho\; \ln(\rho) + \int_{\Lambda}^1 du\; \ln\left( \frac{\alpha(\Lambda)-\alpha(u)}{\alpha(\Lambda-\rho) - \alpha(u)}\right)
\end{split} 
\end{equation}
where we have again dropped terms not depending on $\rho$. Note we are left with only the integral over $u\in[\Lambda,1]$. The critical point of (\ref{eq:sdl2}) is given by
\[
\frac{\sigma}{\tau} \frac{\alpha(\Lambda) +\rho}{z+\rho} \frac{ \Lambda-\rho}{\rho} e^{-\int_\Lambda^{1} du\; \frac{1}{\alpha(\Lambda)+\rho - \alpha(u)}} =1.
\] 
Rearranging this gives
\[
z = \frac{\sigma}{\tau} (\alpha(\Lambda) +\rho) \frac{ \Lambda-\rho}{\rho} e^{-\int_\Lambda^{1} du\; \frac{1}{\alpha(\Lambda)+\rho - \alpha(u)}} + \rho.
\]
Finally we note that
\[
e^{-\int_0^\Lambda du\; \frac{1}{\alpha(\Lambda) +\rho - \alpha(u)} }= e^{-\int_0^\Lambda du\; \frac{1}{u-\Lambda +\rho } }= e^{\ln\left(\frac{\rho-\Lambda}{\rho}\right)}.
\] Recall that $\rho<\Lambda$. Choosing the branch of the logarithm along the positive real axis we are left with
\[
e^{-\int_0^\Lambda du\; \frac{1}{\alpha(\Lambda) +\rho - \alpha(u)}} =-\frac{ \Lambda-\rho}{\rho}.
\]
All together we have
\[
z = -\frac{\sigma}{\tau} (\alpha(\Lambda) +\rho) e^{-\int_0^1 du\; \frac{1}{\alpha(\Lambda)+\rho - \alpha(u)}} + \rho
\] 
as desired.
\end{proof}

\begin{thm}
Assuming the tangent method holds, the portion of the arctic curve following the last dual path is parameterized by
\begin{align}
\begin{split}
X(x) = \frac{x^2I'(x)}{I(x) + x I'(x)} \\
Y(x) = \tau \frac{1}{I(x) + x I'(x)}
\end{split}
\end{align}
with $x\in [\alpha(\Lambda),2]$.
\end{thm}
\begin{proof}
In the large $n$ limit, the last dual path passes through $(\alpha(\Lambda)-z,-\sigma)$ and $(\alpha(\Lambda)+\rho,0)$.
This defines the line
\[
Y= \frac{\sigma}{z+\rho}(X-\alpha(\Lambda)-\rho).
\]
Using equation (\ref{eq:zdl}) and letting $x=\alpha(\Lambda)+\rho$, the above simplifies to
\[
\frac{x}{\tau}I(x)Y+X-x = 0.
\]
As $\rho \in [0,\Lambda]$ we see $x\in[\alpha(\Lambda),2]$. Taking the derivative with respect to $x$, we get the system of equations
\begin{align*}
& \frac{x}{\tau}I(x)Y+X-x=0 \\
& \frac{1}{\tau}(I(x)+xI'(x))Y -1 =0
\end{align*}
which can be solved to yield the desired parameterization.
\end{proof}

\subsection{Examples}
\subsubsection{$\lambda = (n,\ldots,n)$} 
As an example of computing the outer boundary arctic curve consider the case of $\lambda =(n,\ldots,n)$, where $\lambda$ has $n$ parts. In this case, $\alpha(u) = 2-u$, $u\in[0,1]$. We have
\[
I(x) = e^{-\int_0^1 du \frac{1}{x-2+u}} =\frac{x-2}{x-1}
\]
and
\[
I'(x) = \frac{1}{(x-1)^2}.
\]
Plugging this into our parameterization (equation (\ref{thm:main})) we have
\begin{equation}
\begin{split}
X(x) = \frac{x^2}{x^2-2x+2} \\
Y(x) = \frac{\tau (x-1)^2}{x^2-2x+2}
\end{split}
\end{equation}
for $x\in \RR$. Eliminating the parameter we get a formula for the arctic curve
\[
(X - 1)^2 + \left(\frac{2 Y - \tau}{\tau}\right)^2 = 1.
\]
See Figures \ref{fig:circle} and \ref{fig:circlet4} for $\tau = 1$ and $\tau = 4$ respectively.

\begin{figure}[!htb] 
\includegraphics[width=\linewidth]{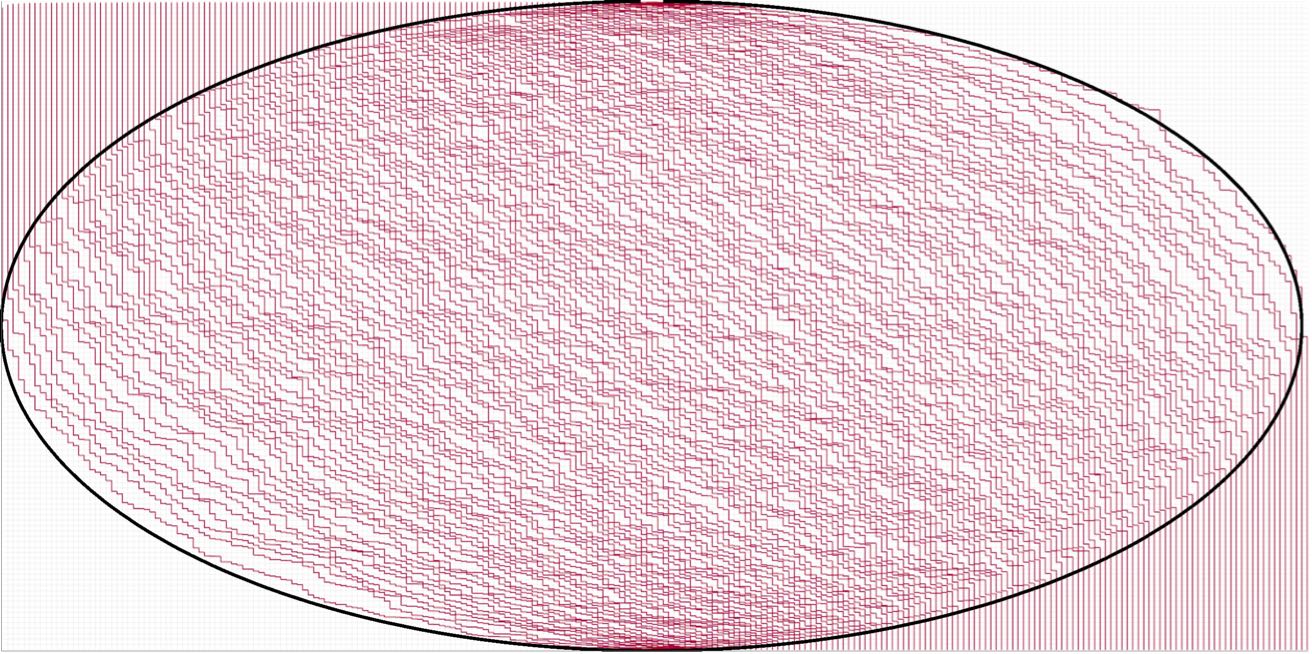}
\caption{A uniformly sampled configuration of paths corresponding to a BLHT of shape $\lambda=(n,\ldots,n)$, with $n=t=120$. In black is the computed arctic curve.}
 \label{fig:circle}
\end{figure}

\begin{figure}[!htb] 
\includegraphics[height=.6\linewidth]{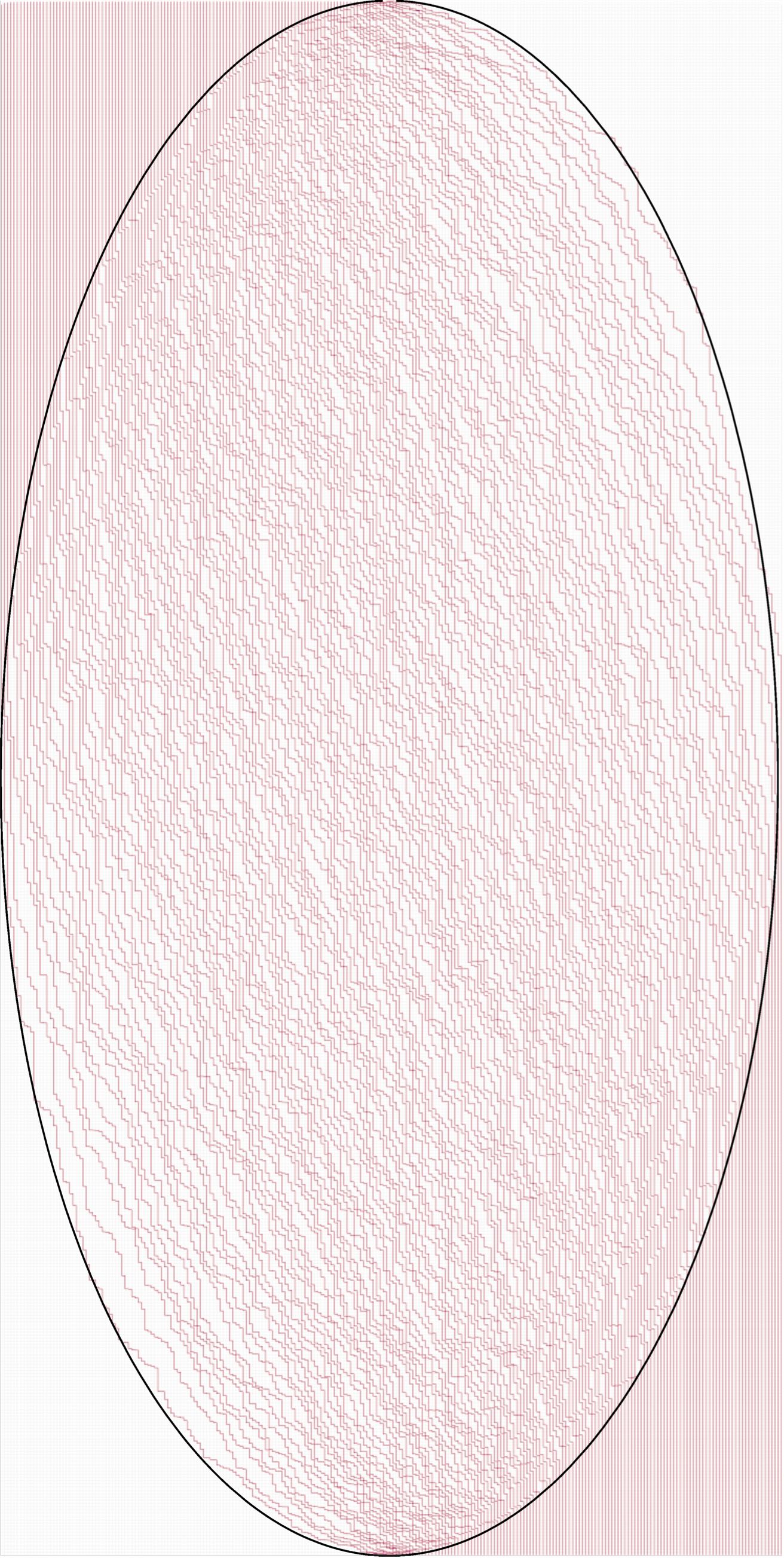}
\caption{A uniformly sampled configuration of paths corresponding to a BLHT of shape $\lambda=(n,\ldots,n)$, with $n=120$ and $t=480$. In black is the computed arctic curve.}
 \label{fig:circlet4}
\end{figure}

\subsubsection{$\lambda = (n,n-1,\ldots,1)$} 
As a second example consider $\lambda = (n,n-1,\ldots,1)$. Here $\alpha(u) = 2-2u$, $u\in[0,1]$. We have
\[
I(x) = e^{-\int_0^1 du \frac{1}{x-2+2u}} =  \sqrt{\frac{x-2}{x}}
\]
and
\[
I'(x) = \frac{1}{x^2} \sqrt{\frac{x}{x-2}}.
\]
This results in the parameterization
\begin{equation}
\begin{split}
X(x) =  \frac{x}{x-1}\\
Y(x) = \frac{\tau \sqrt{x(x-2)}}{|x-1|}
\end{split}
\end{equation}
for $x\in(-\infty,0] \cup [2,\infty)$. This results in the arctic curve
\[
(X-1)^2 +\left( \frac{Y}{\tau}\right)^2 = 1, \hspace{1cm} Y\ge 0.
\]
See Figure \ref{fig:semicircle}.

\begin{figure}[!htb] 
\includegraphics[width=\linewidth]{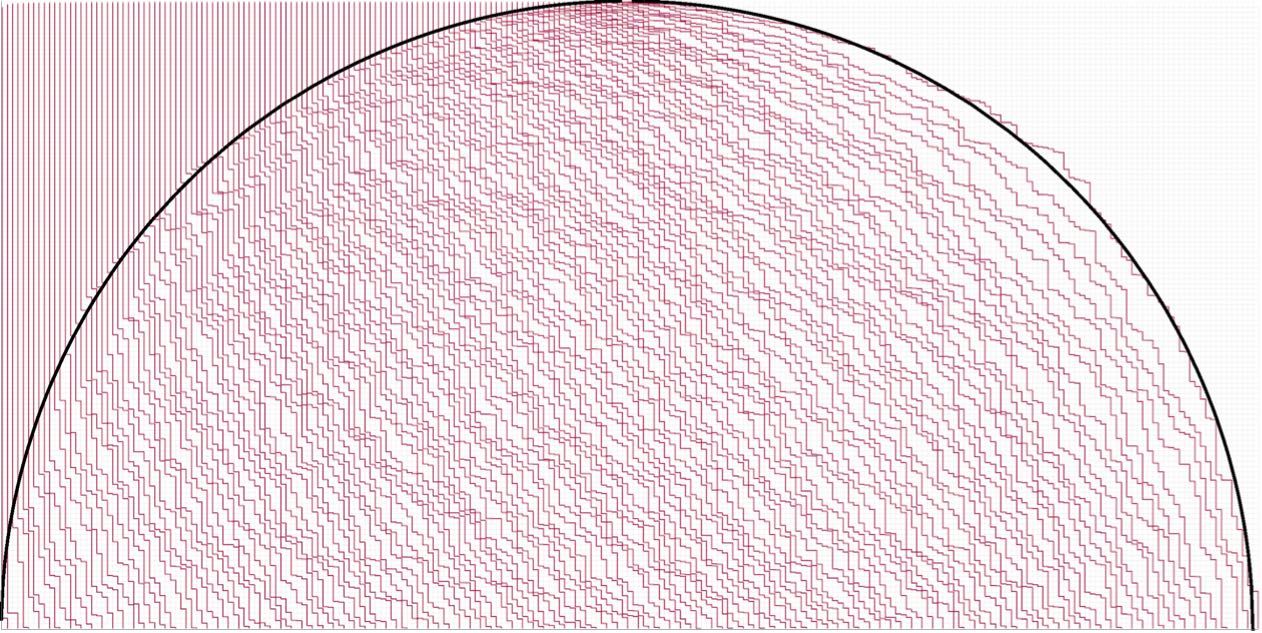}
\caption{A uniformly sampled configuration of paths corresponding to a BLHT of shape $\lambda=(n,n-1,\ldots,1)$, with $n=t=120$. In black is the computed arctic curve.}
 \label{fig:semicircle}
\end{figure}

\subsubsection{$\lambda = ((p-1)n,\ldots,(p-1)n)$}
As a generalization of the first example above, we consider $\lambda$ of the form $\lambda = ((p-1)n,  \ldots , (p-1)n)$ where $p\in \mathbb{N}$, $p>1$, is fixed. The limiting profile is $\alpha(u) = p - u$. We see that
\[
I(x) = e^{-\int_0^1 du \frac{1}{x-p+u}} =  \frac{x-p}{x-p+1}
\]
and
\[
I'(x) = \frac{1}{(x-p+1)^2}
\]
with $x\in\RR$. This gives the parameterization
\begin{equation}
\begin{split}
X(x) = \frac{x^2}{x^2-2(p-1)x+p(p-1)}\\
Y(x) = \tau\frac{(x+p-1)^2}{x^2-2(p-1)x+p(p-1)}
\end{split}
\end{equation}
which gives the arctic curve
\[
(X-p+1)^2+\left(\frac{p}{\tau} Y -p+1\right)^2 + \frac{2p-4}{\tau} XY  = p^2-2p+1.
\]
See Figure \ref{fig:ellipse}.

\begin{figure}[!htb] 
\includegraphics[width=\linewidth]{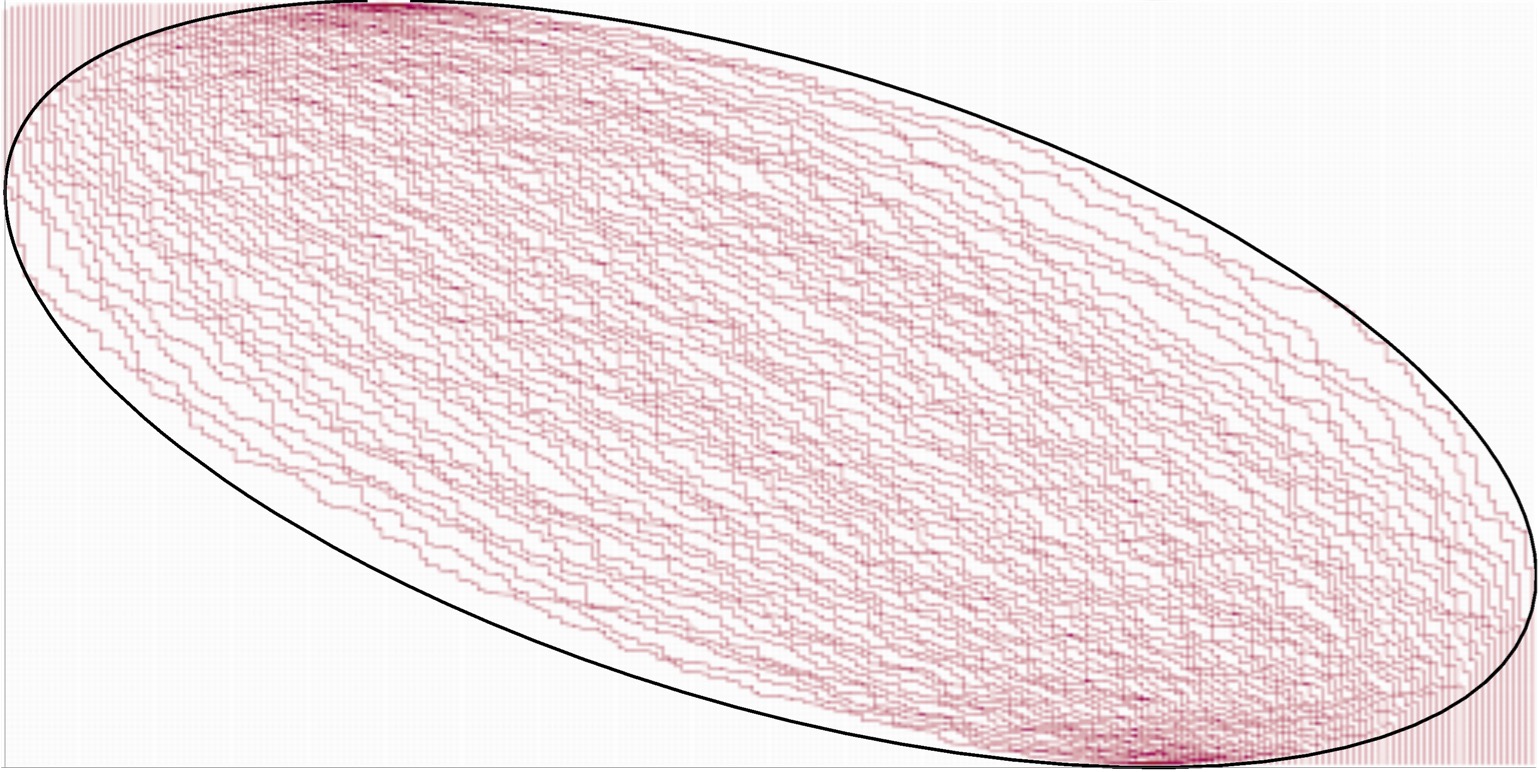}
\caption{A uniformly sampled configuration of paths corresponding to a BLHT of shape $\lambda= ((p-1)n,  \ldots , (p-1)n)$, with $p=4$, $n=60$, $t=120$. In black is the computed arctic curve.}
 \label{fig:ellipse}
\end{figure}

\subsubsection{$\lambda = ((p-1)n, (p-1)(n-1), \ldots , p-1) $}
As a generalization of the second example above, we consider $\lambda$ of the form $\lambda = ((p-1)n, (p-1)(n-1), \ldots , p-1)$ where $p\in \mathbb{N}$, $p>1$, is fixed. In this case, the limiting profile is $\alpha(u) = p(1-u)$. Computing the parameterization we have
\[
I(x) = e^{-\int_0^1 du \frac{1}{x-p+pu}} =  \left(\frac{x-p}{x} \right)^{\frac{1}{p}}
\]
with $x\in(-\infty,0]\cup[p,\infty)$. This gives
\begin{equation}
\begin{split}
X(x) = \frac{x}{x-p+1}\\
Y(x) = \tau\frac{x-p}{x-p+1} \left(\frac{x}{x-p} \right)^{\frac{1}{p}}.
\end{split}
\end{equation}
Eliminating the parameter leaves us with the arctic curve
\[
Y = \tau\frac{X-p}{1-p} \left(\frac{(p-1)X}{X-p} \right)^{\frac{1}{p}}
\]
or
\[
\left((1-p)\frac{Y}{\tau}\right)^{p-1} =  X (X-p)^{p-1} .
\]
Note that this example can provide algebraic curves of degree higher than 2.  See Figure \ref{fig:semithing}. 

\begin{figure}[!htb] 
\includegraphics[width=\linewidth]{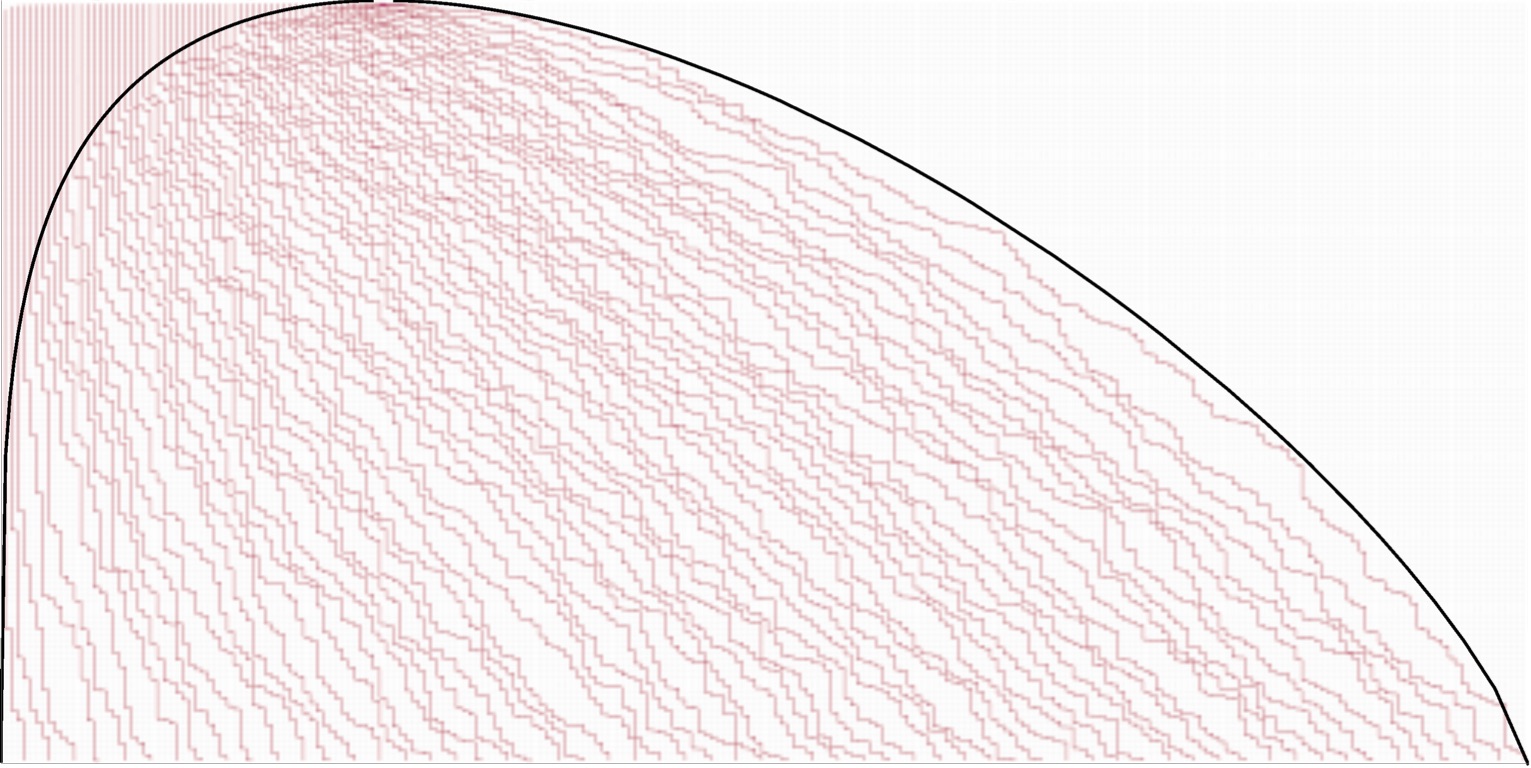}
\caption{A uniformly sampled configuration of paths corresponding to a BLHT of shape $\lambda= ((p-1)n, (p-1)(n-1), \ldots , p-1)$, with $p=4$, $n=60$, $t=120$. In black is the computed arctic curve.}
 \label{fig:semithing}
\end{figure}

\subsection{Freezing Boundaries}
Besides the outer boundary, other portions of the arctic curve can exist at the so called ``freezing boundaries". These occur when the choice of partition freezes a section of paths near the bottom boundary of the domain. When this occurs there will be a new portion of arctic curve separating this frozen region from the disordered region in the bulk. For example, taking $\lambda = (2n,\ldots,2n,n,n-1,\ldots,1)$ where $\lambda$ has $2n$ parts with $n$ having value $2n$, we see a frozen region of no paths resulting in an arctic curve taking the form of a cusp. See Figure \ref{fig:cuspRightDiag}. In general, these frozen regions will come from a macroscopic jump in either the description of the partition $\lambda$, or of the conjugate partition $\lambda'$. That is, either $\lambda_k-\lambda_{k+1}$ or $\lambda'_k-\lambda_{k+1}'$ are linear in $n$, for some row or column of $\lambda$. The following analysis is very similar to that of \cite{DR}.

\begin{remark}
In the analogous problem on the square grid, there are three possible types of frozen region: empty (no paths), horizontal paths, and vertical paths. The parameterization of these freezing boundaries was worked out in \cite{DR}. In our case, the lecture hall tableaux do not appear to develop frozen regions of horizontal paths.
\end{remark}

\subsubsection{Boundary of an empty region}
Consider a BLHT of shape $\lambda$  such that $\lambda$ induces a frozen region of no paths along the bottom boundary of the domain. This means, for some $k$, we must have that $\lambda_{k-1}-\lambda_{k}$ is linear in $n$. In terms of the paths this means that the endpoints of the $(k-1)^{st}$ and $k^{th}$ path satisfy $a_{k-1} -a_{k}$ being linear in $n$.  As the region is empty in the thermodynamic limit, the left-most portion of the arctic curve will follow the path ending at $(n+\lambda_k -k,0)$. By varying the endpoint of this path we will be able to parameterize the arctic curve bounding this frozen region. See Figure \ref{fig:cuspRightDiag} for a diagram.

\begin{figure}
\includegraphics[width=\linewidth]{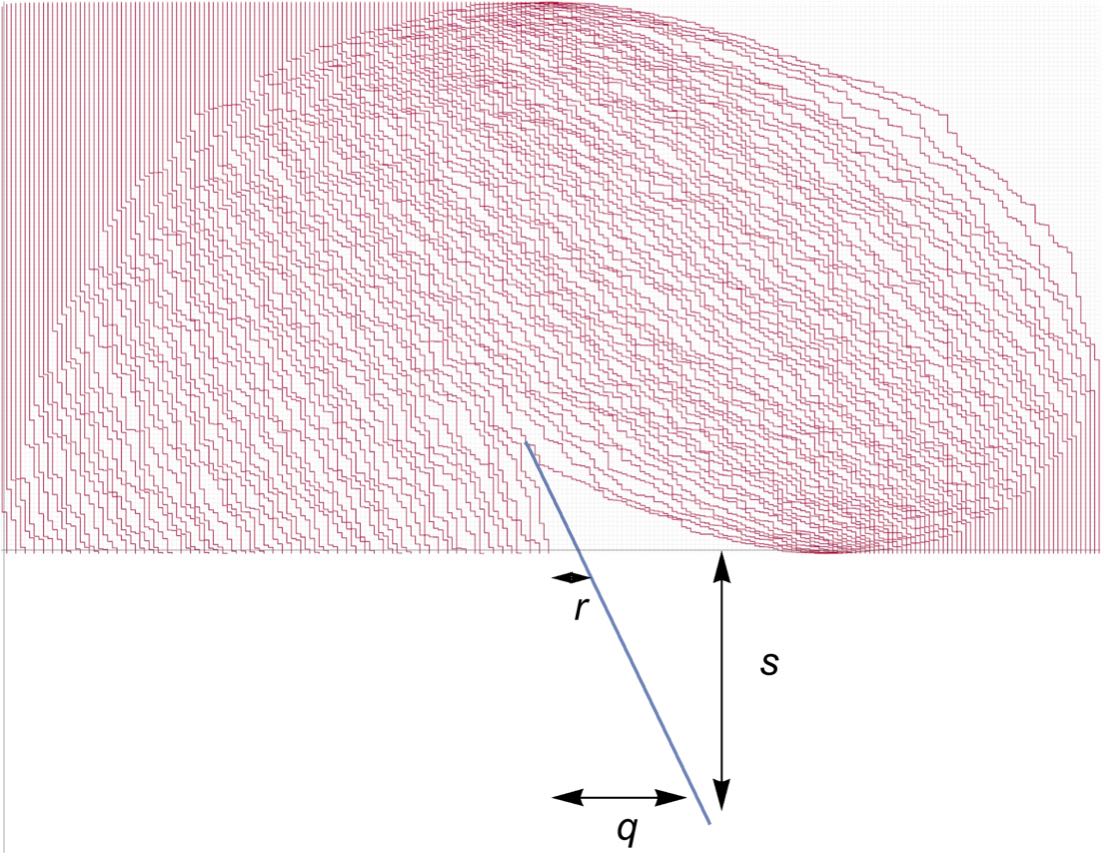}
\caption{A configuration of paths corresponding to a BLHT of shape $\lambda = (2n,\ldots,2n,n,\ldots,1)$. The blue curve represents the trajectory of the extended $k^{th}$ path, which ends at $(n+\lambda_k-k+q,-s)$ and passes through $(n+\lambda_k-k+r,0)$.}
 \label{fig:cuspRightDiag}
\end{figure}

To implement the tangent method, we follow the same procedure as for the outer boundary. Let the total number of configurations be $Z$. Define $\Delta$ through $a_{k-1}-a_{k} \sim n\Delta$. It is the asymptotic size of the jump between the endpoints of the $(k-1)^{st}$ and $k^{th}$ path. Suppose we move the ending point of the $k^{th}$ path to $(n+\lambda_k-k+r,0)$ where $r\in[0,n\Delta)$. Call the new partition function $Z_r$. The ratio of partition functions is
\begin{align}
\frac{Z_r}{Z} & = t^r \prod_{1\le i \le k-1} \frac{\lambda_i - \lambda_k -r+k-i}{\lambda_i - \lambda_k + k - i} \prod_{k+1\le j\le n} \frac{\lambda_k + r -\lambda_j +j-k}{ \lambda_k - \lambda_j +j - k} \nonumber \\
& =  t^r \prod_{1\le j \le n, j\ne k} \frac{\lambda_k + r -\lambda_j +j-k}{ \lambda_k - \lambda_j +j - k}.
\end{align}
Now suppose the $k^{th}$ path is extended to end at $(n+\lambda_k-k+q,-s)$ with $q\in[0,\infty)$. Call the partition function $Z_{qs}$. As before, we can write $Z_{qs}$ as
\begin{align} \label{eq:fbE}
\frac{Z_{qs}}{Z} & = \sum_{r=0}^{\min(q,n\Delta)} \frac{Z_r}{Z} s^{q-r} \binom{n+\lambda_k-k+q}{q-r} \nonumber \\
& =  \sum_{r=0}^{\min(q,n\Delta)} t^r s^{q-r} \prod_{1\le j \le n, j\ne k} \frac{a_k - a_j+r}{ a_k-a_j} \binom{a_k+q}{q-r}.
\end{align}

\begin{lem} \label{lem:fbE}
Consider the limit $n\to \infty$, for parameters scaling as $t=n\tau$, $s=n\sigma$, $r=n\rho$, $q=nz$, $a_i = \lfloor n\alpha(\frac{i}{n})\rfloor$, $a_{k-1}-a_{k} = n\Delta$, and $\frac{k}{n} = v$ . In this limit, the $k^{th}$ path passes though the point $(\alpha(v)+\rho,0)$, with $\rho$ related to $z$ by
\begin{equation} \label{eq:fbEz}
z = (\alpha(v) + \rho) \frac{\sigma}{\tau} e^{-\int_0^1du \frac{1}{\alpha(v)+\rho-\alpha(u)}} +\rho.
\end{equation}
\end{lem}
\begin{proof}

In the $n\to \infty$ limit, equation (\ref{eq:fbE}) takes the form
\begin{equation}
\frac{Z_{qs}}{Z} \sim \sqrt{\frac{n}{2\pi}} e^{n\,z\; \ln(n)} \int_0^{\Delta} d\rho \; \sqrt{\frac{\alpha(v)+z}{(\alpha(v)+\rho)(z-\rho)}}e^{n\,S(\rho)} 
\end{equation} where
\begin{equation*}
\begin{split}
S(\rho) = \rho \; \ln(\tau) - (z-\rho)\ln(\sigma) - (\alpha(v) + \rho) \ln(\alpha(v)+\rho) - (z-\rho) \ln(z-\rho) \\
+ \int_0^1 du\; \ln\left(\frac{\alpha(v)+\rho - \alpha(u)}{\alpha(v)-\alpha(u)} \right).
\end{split}
\end{equation*}
The critical point occurs when
\[
\frac{\tau}{\sigma} \frac{z-\rho}{\alpha(v) + \rho} e^{\int_0^1 du \frac{1}{\alpha(v)+\rho -\alpha(u)}} = 1.
\]
Rearranging we arrive at equation (\ref{eq:fbEz}).
\end{proof}
Using Lemma \ref{lem:fbE}, we get a parameterization of the desired portion of the arctic curve.
\begin{thm}
Assuming the tangent method holds, the portion of the arctic curve bounding such a frozen region is parameterized by
\begin{align}
\begin{split}
X(x) = \frac{x^2I'(x)}{I(x) + x I'(x)} \\
Y(x) = \tau\frac{1}{I(x) + x I'(x)}
\end{split}
\end{align}
with $x\in [\alpha(v),\alpha(v)+\Delta]$.
\end{thm}
\begin{proof}
From the points $(\alpha(v)+z,-\sigma)$, $(\alpha(v)+\rho,0)$, and Lemma \ref{lem:fbE} we have a family lines
\[
\frac{x}{\tau}I(x)Y +X-x=0
\] with $x=\alpha(v)+\rho$. Taking the derivative with respect to $x$, we get the system of equations
\begin{align*}
& \frac{x}{\tau}I(x)Y+X-x=0 \\
& \frac{1}{\tau} (I(x)+xI'(x))Y -1 = 0
\end{align*}
which can be solved to yield the desired parameterization. Note that since $\rho \in [0,\Delta)$, the range of $x$ is $[\alpha(v),\alpha(v)+\Delta]$.
\end{proof}

\subsubsection{Boundary of a vertically frozen region} Now suppose we have a frozen region of vertical paths. In terms of the partition $\lambda$, there exists integers $1\le a< b \le n$ such that $\lambda_a = \lambda_{a+1} = \ldots = \lambda_b$, and $b-a$ is proportional to $n$. In terms of the dual partition, this means there exists $k$ such that  $\lambda_k' -\lambda_{k+1}'$ is proportional to $n$. 

In the path description we see a frozen region of vertical paths, which in the dual paths description becomes an empty frozen region whose left boundary follows the $k^{th}$ dual path in the large $n$ limit. See Figure \ref{fig:cuspLeftwDual} for an example such a frozen region and Figure \ref{fig:cuspLeftDiag} for a diagram.

\begin{figure} [!htb]
\begin{minipage}[c]{0.45\linewidth}
\includegraphics[width=\linewidth]{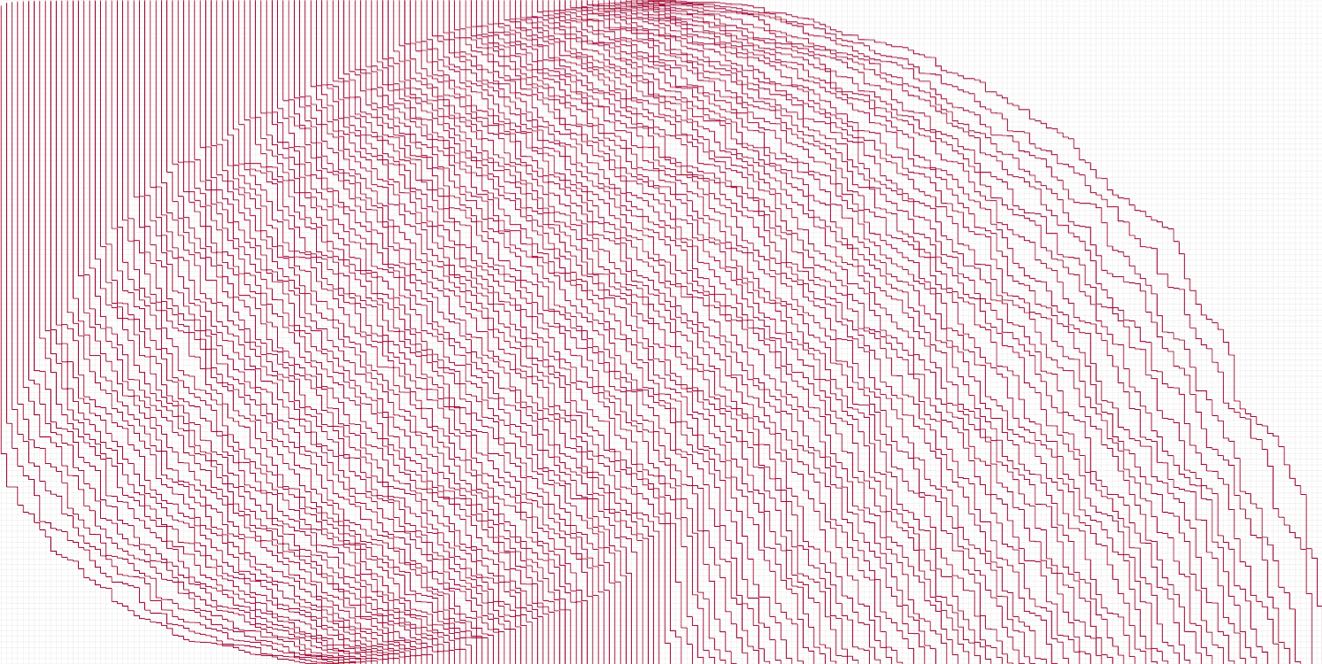}
\end{minipage}
\hfill
\begin{minipage}[c]{0.45\linewidth}
\includegraphics[width=\linewidth]{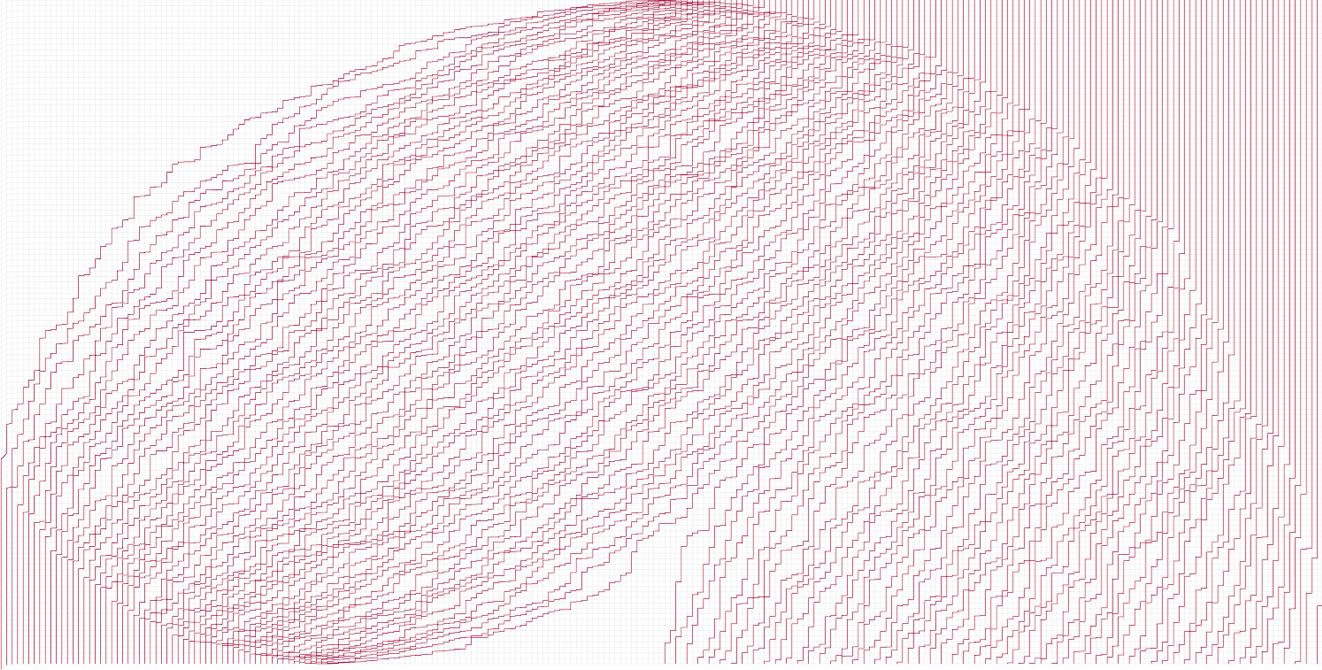}
\end{minipage}%
\caption{A configuration of paths with the corresponding dual paths for a BLHT of shape $\lambda =(2n,2n-1,\ldots,n+1,n,\ldots,n)$.}
\label{fig:cuspLeftwDual}
\end{figure}

\begin{figure}[!htb]
\includegraphics[width=\linewidth]{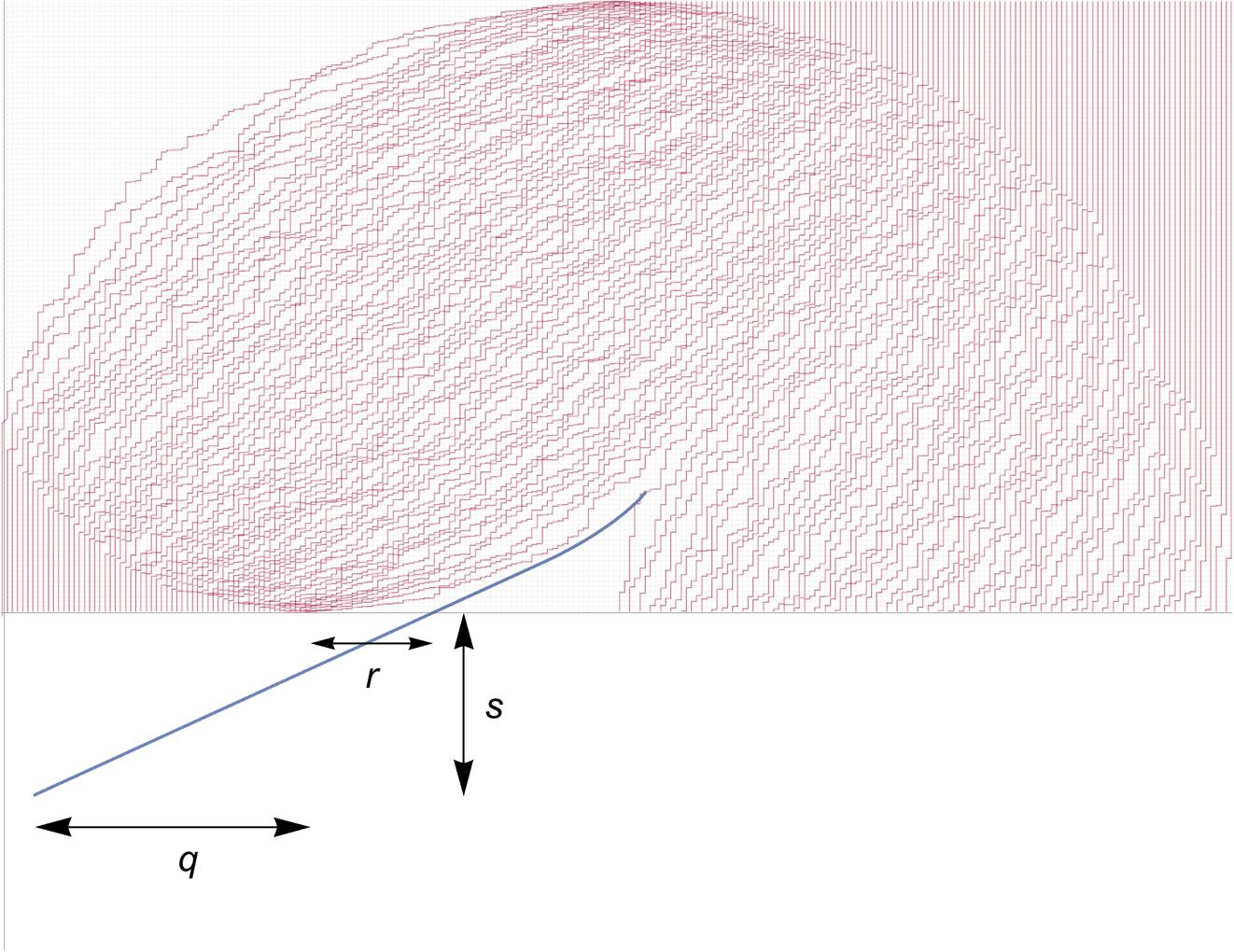}
\caption{A configuration of dual paths corresponding to a BLHT of shape $\lambda =(2n,2n-1,\ldots,n+1,n,\ldots,n)$. The blue curve represents the trajectory of the extended $k^{th}$ dual path, which starts at $(2n-\lambda_k'+q-1,-s-\frac{1}{(2n-\lambda_k'+q)^2})$ and passes through $(2n-\lambda_k'+r-1,0)$.} \label{fig:cuspLeftDiag}
\end{figure}

Let $\lambda$ be a partition such that $\lambda_k'- \lambda_{k+1}'\sim n\Delta$. First extend $\lambda$ to $(\lambda_1,\ldots,\lambda_n,0,\ldots,0)$ by adding $m$ parts of zero to the end of $\lambda$. 
Call $Z_r$ the partition function with the starting point of the $k^{th}$ dual path moved to the right by $r$. 
In terms of the original partition, moving the starting point of the $k^{th}$ dual path to the right by $r$ corresponds to changing $\lambda_i$ to $\lambda_i - 1$ for $i=b,b-1,\ldots,b-r+1$. Call the resulting partition $\mu$. 

From the product formula we have
\[
\begin{split}
\frac{Z_r}{Z} & = t^{|\mu|-|\lambda|} \prod_{1\le i < j \le n+m} \frac{\mu_i-\mu_j +j-i}{\lambda_i-\lambda_j+j-i}\\
& =t^{-r} \prod_{\substack{ 1\le i\le b-r \\ b-r+1 \le j \le b }} \frac{\lambda_i -(\lambda_j-1)+j-i}{\lambda_i -\lambda_j+j-i} \prod_{\substack{b-r+1 \le i\le b \\b+1 \le j\le n+m} } \frac{(\lambda_i-1) -\lambda_j+j-i}{\lambda_i -\lambda_j+j-i} \\
\end{split}
\] where the other other terms in the product are one. Both remaining terms in the product are telescoping and we have
\[
\frac{Z_r}{Z} = t^{-r} \prod_{1\le i \le b-r} \frac{\lambda_i - \lambda_b+b+1 -i }{\lambda_i - \lambda_b+b-r+1-i} \prod_{b+1\le j\le n+m} \frac{\lambda_b-\lambda_j+j-(b+1)}{\lambda_b-\lambda_j+j-(b-r+1)}.
\] We can rewrite the above as follows
\[
\begin{split}
\frac{Z_r}{Z} & = t^{-r} \prod_{\substack{1\le i \le n+m\\ i\ne b-r+1}} \frac{\lambda_i - \lambda_b+b+1 -i }{\lambda_i - \lambda_b+b-r+1-i} \prod_{b-r+2\le i \le b} \frac{\lambda_i - \lambda_b+b-r+1-i}{\lambda_i - \lambda_b+b+1 -i} \\
& = t^{-r}(-1)^{r-1} \prod_{\substack{1\le i \le n+m\\ i\ne b-r+1}} \frac{\lambda_i - \lambda_b+b+1 -i }{\lambda_i - \lambda_b+b-r+1-i}.
\end{split}
\]

Now extend the $k^{th}$ dual path to begin at $(2n-1-\lambda_k'+q,-s-\frac{1}{(2n-\lambda_k'+q)^2})$, with $q\in(-\infty,\Delta]$. Call the corresponding partition function $Z_{qs}$. Note that $\lambda_b=k$ and $\lambda_k'=b$. $Z_{qs}$ can be written
\begin{align}
\frac{Z_{qs}}{Z} & = \sum_{r=\max(q,0)}^{\Delta} \frac{Z_r}{Z} s^{r-q} \binom{m+n-1+k-\lambda_k'+r}{m+n-1+k-\lambda_k'+q} \nonumber \\
& =  \sum_{r=\max(q,0)}^{\Delta}  t^{-r} (-1)^{r-1} \prod_{\substack{1\le i \le n+m\\ i\ne b-r+1}} \frac{\lambda_i - \lambda_b+b+1 -i }{\lambda_i - \lambda_b+b-r+1-i}  s^{r-q} \binom{m+n-1+\lambda_b -b + r}{m+n-1+\lambda_b -b + q} \nonumber \\
& =  \sum_{r=\max(q,0)}^{\Delta}  t^{-r} s^{r-q} (-1)^{r-1} \prod_{\substack{1\le i \le n+m\\ i\ne b-r+1}} \frac{a_i -a_b+1 }{a_i -a_b-r+1}  \binom{m-1+a_b + r}{m-1+a_b + q}.
\end{align}

\begin{lem} \label{lem:fbV}
Consider the limit $n\to \infty$, for parameters scaling as $t=n\tau$, $s=n\sigma$, $r=n\rho$, $q=nz$, $m=nM$, $\lambda_k'-\lambda_{k+1}' = n \Delta$, $b=n\beta$, and $a_i = \lfloor n\alpha(\frac{i}{n}) \rfloor$. In this limit, the $k^{th}$ dual path passes through the point $(\alpha(\beta)+\rho,0)$, with $\rho$ related to $z$ by
\begin{equation}
z = (\alpha(\beta)+\rho) \frac{\sigma}{\tau} e^{-p.v.\int_0^{1}du \frac{1}{\alpha(\beta)+\rho-\alpha(u)}}  +\rho.
\end{equation}
\end{lem}
\begin{proof}
Taking the limit $n\to \infty$
\begin{align*}
\frac{Z_{qs}}{Z}\sim\sqrt{\frac{n}{2\pi}} e^{-nz\,\ln(n)}\int_{\max(z,0)}^\Delta d\rho \sqrt{\frac{M+\alpha(\beta)+\rho}{(M+\alpha(\beta)+z)(\rho-z)}}  e^{n\, S(\rho)}
\end{align*}
where
{\small
\[
\begin{split}
S(\rho) = (\rho-z)\ln(\sigma) - \rho\,\ln(\tau) + (M+\alpha(\beta)+\rho)\ln(M+\alpha(\beta)+\rho) - (M+\alpha(\beta)+z)\ln(M+\alpha(\beta)+z) \\
- (\rho-z)\ln(\rho-z) + i \rho \pi +\int_0^{\beta-\rho-\epsilon} du\, \ln\left(\frac{\alpha(u) - \alpha(\beta)}{\alpha(u) - \alpha(\beta) - \rho}\right)  + \int_{\beta-\rho+\epsilon}^{1+M}du\, \ln\left(\frac{\alpha(u) - \alpha(\beta)}{\alpha(u) - \alpha(\beta) - \rho}\right)
\end{split}
\]
}
with $\epsilon = \frac{1}{n}$ tending to $0$.  

The sole critical point of $S$ occurs when
\[
-\frac{\sigma}{\tau} \frac{M+\alpha(\beta)+\rho}{\rho-z} e^{\int_0^{\beta - \rho -\epsilon} \frac{1}{\alpha(u)-\alpha(\beta)-\rho} du+\int_{\beta - \rho +\epsilon}^{1+M} \frac{1}{\alpha(u)-\alpha(\beta)-\rho} du} = 1.
\]
This can be simplified to 
\[
z = (\alpha(\beta)+\rho) \frac{\sigma}{\tau} e^{-p.v.\int_0^{1}du \frac{1}{\alpha(\beta)+\rho-\alpha(u)}}  +\rho
\]
where we use that 
\[
e^{\int_{1}^{1+M} \frac{1}{\alpha(u)-\alpha(\beta)-\rho} du} = e^{\int_{1}^{1+M} \frac{1}{1-u-\alpha(\beta)-\rho} du} = \frac{\alpha(\beta)+\rho}{M+\alpha(\beta)+\rho}.
\]
\end{proof}

\begin{thm}
Assuming the tangent method holds, the portion of the arctic curve bounding such a frozen region is parameterized by
\begin{align}
\begin{split}
X(x) = \frac{x^2I'(x)}{I(x) + x I'(x)} \\
Y(x) = \frac{\tau}{I(x) + x I'(x)}
\end{split}
\end{align}
where $I(x) = e^{-p.v.\int_0^{1}du \frac{1}{x-\alpha(u)}}$ and $x\in [\alpha(\beta),\alpha(\beta)+\Delta]$. Recall $\lambda_k'=b\sim n\beta$ and $\lambda_k'-\lambda_{k+1}'\sim n\Delta$. 
\end{thm}
\begin{proof}
From the points $(\alpha(\beta)+z,-\sigma)$, $(\alpha(\beta)+\rho,0)$, and Lemma \ref{lem:fbV} we have a family of lines
\[
\frac{x}{\tau}I(x)Y +X-x=0
\] with $x=\alpha(\beta)+\rho$. Taking the derivative with respect to $x$, we get the system of equations
\begin{align*}
& \frac{x}{\tau}I(x)Y+X-x=0 \\
& \frac{1}{\tau} (I(x)+xI'(x))Y -1 = 0
\end{align*}
which can be solved to yield the desired parameterization. Note that since $\rho \in [0,\Delta]$, the range of $x=\alpha(\beta)+\rho$ is $[\alpha(\beta),\alpha(\beta)+\Delta]$.
\end{proof}

\subsection{Examples}
\subsubsection{$\lambda = (2n,\ldots,2n,n,\ldots,1)$}
As an example of a BLHT whose arctic curve contains a freezing boundary, consider the partition $\lambda=(2n,\ldots,2n,n,\ldots,1)$. The limiting profile is 
\[
\alpha(u) = \begin{cases} 4-2u & 1< u\le 2 \\ 4-u & 0\le u\le 1 \end{cases}.
\]
From this we have
\[
I(x) =\frac{x-4}{x-3} \sqrt{\frac{x-2}{x}}
\]
and
\[
I'(x) = \frac{2x^2-9x+12}{(x-3)^2 x^2} \sqrt{\frac{x}{x-2}}.
\]
Plugging this into equation (\ref{thm:main}), we get
\begin{align*}
X(x) = \frac{x (2x^2-9x+12)}{x^3-7x^2+17x-12} \\
Y(x) =\frac{\tau\; x (x-3)^2}{x^3-7x^2+17x-12} \sqrt{\frac{x-2}{x}}
\end{align*}
with $x\in (-\infty,0] \cup [2,\infty)$. The portion of the arctic curve corresponding to the freezing boundary is $x\in [2,3]$. See Figure \ref{fig:cuspRight}.

\begin{figure}[!htb] 
\includegraphics[width=\linewidth]{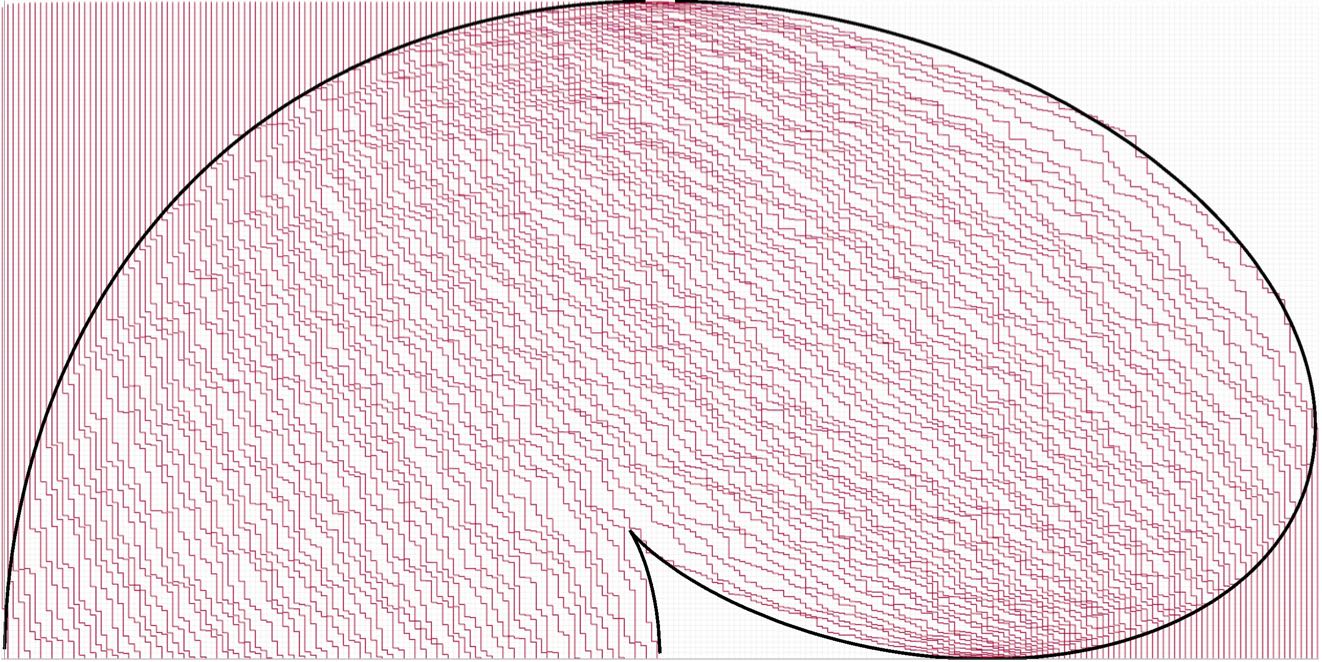}
\caption{A uniformly sampled configuration of paths corresponding to a BLHT of shape $\lambda=(2n,\ldots,2n,n,\ldots,1)$, with $n=60$ and $t=120$. In black is the computed arctic curve.}
\label{fig:cuspRight}
\end{figure}

\subsubsection{$\lambda = (2n,2n-1,\ldots,n+1,n,\ldots,n)$}
In the case $\lambda =  (2n,2n-1,\ldots,n+1,n,\ldots,n)$ the limiting profile is 
\[
\alpha(u) = \begin{cases} 3-u & 1< u\le 2 \\ 4-2u & 0\le u\le 1 \end{cases}.
\]
We have
\[
I(x) = \frac{1}{x-1} \sqrt{(x-4)(x-2)}
\] and
\[
I'(x) = \frac{2x-5}{(x-1)^2} \sqrt{\frac{1}{(x-2)(x-4)}}
\]
The gives the parameterization
\begin{align*}
X(x) =  \frac{x^2(2x-5)}{x^3-5x^2+9x-8}\\
Y(x) = \tau \frac{(x-1)^2}{x^3-5x^2+9x-8} \sqrt{(x-4)(x-2)}
\end{align*}
for $x\in(-\infty,2] \cup [4,\infty)$. The portion of the arctic curve corresponding to the freezing boundary is $x\in [1,2]$.  See Figure \ref{fig:cuspLeft}.

\begin{figure}[!htb] 
\includegraphics[width=\linewidth]{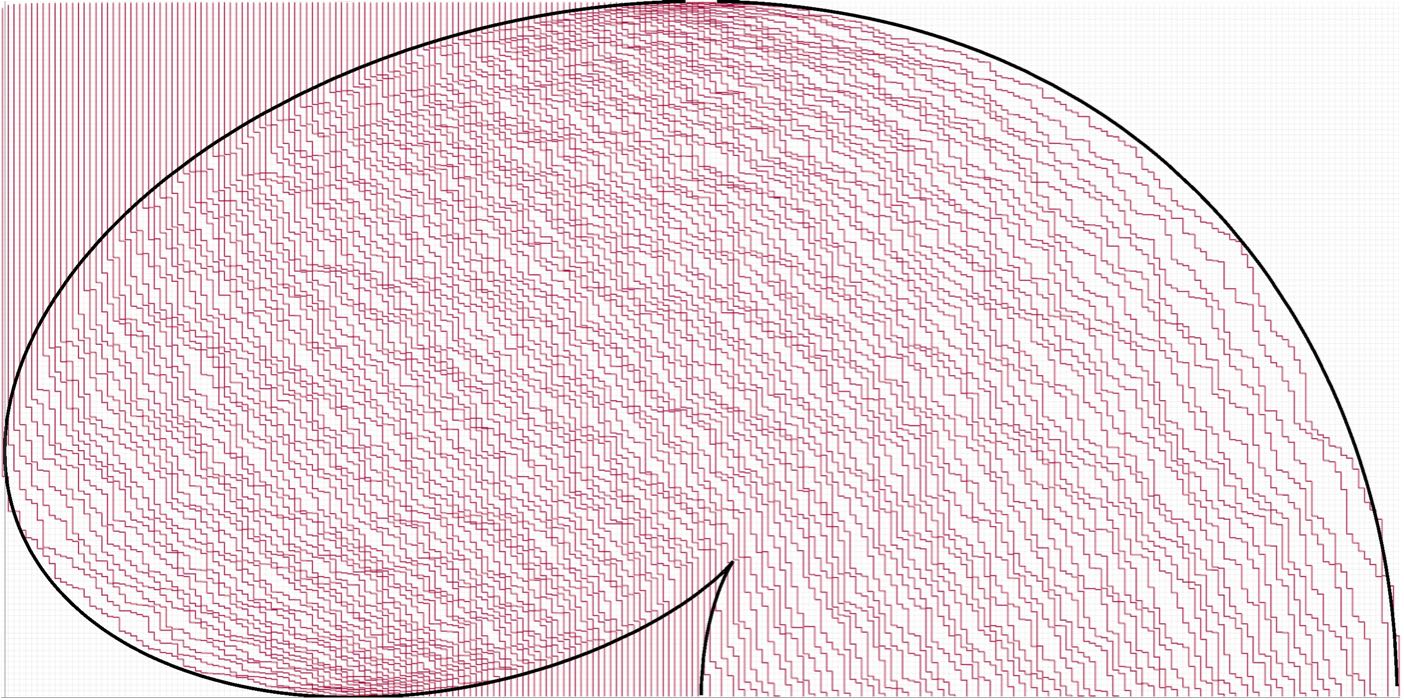}
\caption{A uniformly sampled configuration of paths corresponding to a BLHT of shape $\lambda=(2n,2n-1,\ldots,n+1,n,\ldots,n)$, with $n=60$ and $t=120$. In black is the computed arctic curve.}
 \label{fig:cuspLeft}
\end{figure}

More complex examples can be dealt in a similar way. An example with three internal jumps, for which we do not illustrate the calculations here, is shown in Figure \ref{fig:manycusps}.

\begin{figure}[!htb] 
\includegraphics[width=\linewidth]{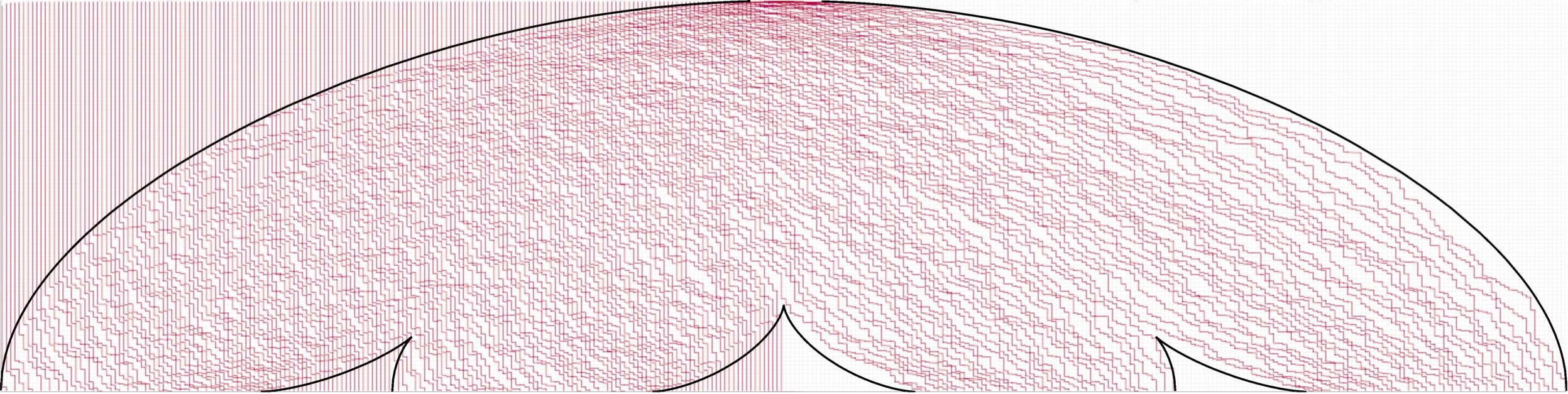}
\caption{A uniformly sampled configuration of paths corresponding to a BLHT of shape $\lambda=(6n,\ldots,5n+1,4n,\ldots,3n+1,2n,\ldots,2n,2n,\ldots,n+1,n,\ldots,n,n,\ldots,1)$, with $n=30$ and $t=60$. In black is the computed arctic curve.}
 \label{fig:manycusps}
\end{figure}

\section{Dimer models and arctic curves}
\label{dimer}

In this section we present a heuristic which gives another way to compute the arctic curves that 
we computed in Section \ref{tangent}.

We use the lecture hall lattice defined in Section \ref{combi}.
Given $\lambda$, recall that the lecture hall lattice ${\mathcal H}_t(\lambda)$ is the graph such that the vertices
are 
white vertices $(i,k+r/(i+1))$ and black vertices $(i,k+r/(i+1))$ for  $0\le i\le \lambda_1+n-1$, 
$0\le r\le i$ and $0\le k<t$. These have an edge in between them. 
Moreover
\begin{itemize}
\item we have a white vertex $(n-i,t)$ and an edge from this vertex to the black vertex
$(n-i,t-1/(n-i+1))$; and
\item we have a black vertex $(n-i+\lambda_i,-1/(n-i+\lambda_i+1))$ and an edge from the white vertex
$(n-i+\lambda_i,0)$ to  this vertex.
\end{itemize}
Finally let us list the other edges.
\begin{itemize}
\item For $0\le i\le \lambda_1+n-1$ and $0\le r\le i$ and $0\le k<t$, there is an edge from the black vertex  
$(i+1,k+r/(i+2))$  to the  white vertex  $(i,k+r/(i+1))$.
\item For $1\le i\ge \lambda_1+n-1$ and $0\le r\le i$ and $0\le k<t$, there is  an edge from the white vertex  $(i,k+(r+1)/(i+1))$ to the black vertex  $(i,k+r/(i+1))$.
\end{itemize}

\begin{figure}
  \centering
\begin{tikzpicture}
\DLL{3}3
\end{tikzpicture}
\caption{The lecture hall lattice  ${\mathcal H}_3$}
\label{t33again}
\end{figure}
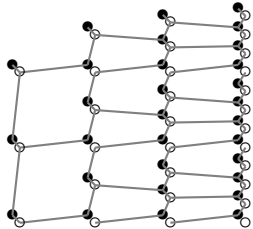

We can give the lecture hall lattice what is known as a \emph{Kasteleyn weighting}, which means an assignment of a sign $s(w, b)$ to each edge $(w,b)$ such that for a cycle $v_1, \dots, v_m, v_{m+1}=v_1$ around any face of the graph, the product of signs $\prod_{i=1}^m s(v_i, v_{i+1})$ is even if $m = 2$ mod $4$ and is odd if $m = 0$ mod $4$ \cite{Ka,K}. 
Thus we construct the Kasteleyn matrix $K$, with rows indexed by black vertices and columns indexed by white vertices, such that
$$K(b, w) =
\begin{cases}
      0 & \text{if $(w,b)$ is not an edge} \\
      -1 & \text{if $b$ and $w$ are in the same position} \\
      1 & \text{otherwise.}   \; \;
\end{cases} 
$$
This choice of signs is indeed a Kasteleyn weighting since the faces of ${\mathcal H}_t(\lambda)$ have either $6$ boundary edges with two negative signs, or $8$ boundary edges with $3$ negatives signs.

This then makes available to us all of the tools of the theory of the Kasteleyn operator; in particular, we have
\begin{theorem}\cite{Ka,K}
  	 	The number of dimer configurations $Z$ is equal  to $$Z = | \det K |.$$
\end{theorem}
If we know the inverse Kasteleyn  matrix, we can also compute the correlation functions.
  \begin{theorem}\cite{K}
 Given a set of edges $X = ((w_1, b_1), (w_2, b_2), \dots, (w_k, b_k))$, the probability that all of the edges in $X$ occur in a dimer configuration is
  	 	$$(\prod_{i = 1}^k K(b_i, w_i)) \det (K^{-1} (w_i, b_j) )_{1 \leq i, j \leq k} \;\;.$$
\end{theorem}

In general, if we have a finite bipartite graph with white vertices $W$ and black vertices $B$, then we can consider the Kasteleyn operator $K$ as a linear map from the vector space of functions on $W$ to the space of functions on $B$. For a function $f : W \rightarrow \mathbb{C}$, we may define $K f : B \rightarrow \mathbb{C}$ by

 \begin{equation}\label{eqn:findif}
    (K f)(b) = \sum_{w} K(b, w) f(w) = \sum_{ \text{edges } (w,b)} K(b, w) f(w)
  \end{equation}
which means that $K$ is a local operator.

In what follows we assume both $t,\lambda_1$ grow linearly in $n$ and we rescale the entire graph by $1/n$ to embed it into a finite rectangular region $R \subset \R^2$. Each vertex $(X, Y)$ in the graph $H_t(\lambda)$ has a pair of continuum coordinates $(x, y)  = (X/n, Y/n)$. We will identify vertices with their pair of rescaled coordinates $(x, y)$.

Consider a black vertex $(x, y)$ with corresponding lattice coordinates $(X, Y)$ which are not of the form $(i, k - 1/(i+1))$ for $k$ an integer. In this case if we have a smooth function $f$ on the plane, then

  \begin{equation}\label{eqn:kast} 
  (K f)(x, y)  = f(x-\epsilon, y ) + f(x, y+ \epsilon^2/(x+\epsilon)) - f(x, y)  \;\;.
   \end{equation}

See Figure \ref{fig:j}. (In fact Figure \ref{fig:j} displays a neighborhood of $(x, y)$ with a slightly different embedding of the graph from what we have defined above. The computations that follow are unaffected by this up to first order, so we omit details.)

We guess the form of the inverse of the Kasteleyn matrix using an ansatz developed by Keating, Reshetikhin and Sridhar \cite{KRS}. They developed such an ansatz for the dimer model on the
hexagonal lattice. Based on numerical evidence, we adapt their ansatz to our setting.

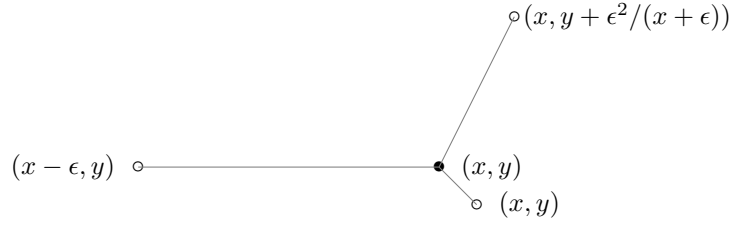
\begin{figure}
\centering
\begin{tikzpicture}  
\node at (0,0){$\bullet$};
\node at (.5,.-.5){$\circ$};
\node at (1,2){$\circ$};
\node at (-4,0){$\circ$};
\draw[gray](0,0)--(.5,-.5);
\draw[gray](0,0)--(1,2);
\draw[gray](0,0)--(-4,0);
\node at (0.7,0){$(x,y)$};
\node at (-4-1,0){$(x-\epsilon,y)$};
\node at (0.7+.5,-0.5){$(x,y)$};
\node at (2.5,2){$(x,y+\epsilon^2/(x+\epsilon))$};
\end{tikzpicture}
\caption{Neighborhood of the black vertex $(x,y)$}
\label{fig:j}
\end{figure}

    The ansatz is that for small $\epsilon$, as a function on pairs of vertices $K^{-1}$ will equal the restriction of the function on the continuous domain $R \times R$ defined by
    $$K^{-1}((x, y), (x', y')) = Re\left( \exp{\left( -(x - x') \frac{\log( \epsilon )}{\epsilon} + \frac{F((x, y), (x', y'))}{\epsilon} \right)} \phi((x, y), (x', y')) \right)$$
    with $F : R \times R \rightarrow \C$ piecewise smooth, and
    $$\phi((x, y), (x', y')) = \sum_{k=0}^\infty \phi^{(k)}((x, y), (x', y')) \epsilon^k$$
    for
     $$  \phi^{(k)} :R \times R \rightarrow \mathbb{C}$$
    piecewise smooth functions.
    
     In what follows we simply assume $K^{-1}$ has this form, and that the above power series in $\epsilon$ converges for $\epsilon$ in some neighborhood of $0$. We expand $K K^{-1}((x, y), (x', y'))$ for $(x, y) \neq (x', y')$ as a power series in $\epsilon$, and deduce a differential equation for $F$ from the fact that each coefficient of the series must exactly equal $0$.

     If $K^{-1}$ has the above form, then we claim that for all $(x', y')$, the function $(x, y) \mapsto F((x, y), (x', y'))$ satisfies the following differential equation for all $(x, y) \neq (x', y')$:
     \begin{equation} \label{eqn:Feq}
     e^{-F_x} + F_y/x = 0 
     \end{equation}
     where $F_x \defeq \frac{\partial F}{\partial x}((x, y), (x', y'))$ and similarly for $F_y$.

     	We compute: By \eqref{eqn:kast} we get
\begin{align*}
0 &= (K K^{-1}) ( (x, y), (x', y') )  \\
&=\sum_{(x'', y'') \sim (x, y)} K((x, y), (x'', y'')) K^{-1}((x'', y''), (x', y')) \\
&= K^{-1}((x - \epsilon, y), (x', y')) + K^{-1}((x, y + \epsilon^2 /(x+\epsilon)), (x', y')) - K^{-1}((x, y), (x', y')) \\
&=Re [  \exp{ \left( -(x - \epsilon - x') \log( \epsilon) / \epsilon + F(x - \epsilon, y , x', y')/\epsilon \right)} \phi((x-\epsilon,y);(x',y'))]\\ 
& \-\ + Re[\exp{ \left( -(x - x') \log(\epsilon)/ \epsilon + F(x, y + \epsilon^2/(x+\epsilon), x', y')/\epsilon \right)}(\phi(x,y+\epsilon^2/(x+\epsilon)),(x',y')) ]\\ 
& \-\ -  Re[\exp{\left( -(x - x') \log(\epsilon)/ \epsilon + F(x, y, x', y')/ \epsilon \right)} \phi((x,y),(x',y'))  ] \;\;.
\end{align*}

We assume that this finite difference equation is satisfied not just for the real part, but for complex part of the ansatz as well (recall we allow $F$ to have an imaginary part). We also assume $\phi^{(0)}((x, y), (x', y')) \neq 0$. So Taylor expanding around $((x, y), (x', y'))$ and simplifying (being careful to observe that any term with $\phi^{(k)}$ for $k \geq 1$ either cancels or contributes at higher order) we obtain 
\begin{align*}
0 
&= \exp{ \left( \log( \epsilon) - F_x + O(\epsilon) \right)} \\
 & \-\ + \exp{\left(  \epsilon F_y/x + O(\epsilon^2) \right)} 
 \\
 & \-\  - 1 \\
&= \epsilon \exp{ \left( - F_x \right)}\left( 1 + O(\epsilon)\right) \\
 & \-\ + \epsilon F_y/x  + O(\epsilon^2 ) \\
&= \epsilon \exp{ \left( - F_x \right)} + \epsilon F_y/x + O(\epsilon^2) \;\;.
\end{align*}
Then at order $\epsilon$ we get \eqref{eqn:Feq}, as desired.

\noindent{\bf Remark.} 
In the simplest case for which we could compute the inverse Kasteleyn's asymptotics analytically, the Ansatz is correct and $F$ indeed satisfies this differential equation. 
Consider the case where $t=n$ and $\lambda=(\kappa n,0^{n-1})$ for some constant $\kappa$.
The inverse Kasteleyn matrix is then
\begin{align*}
K_L^{-1}((X, Y), (X', Y')) &= \frac{P((n-1, n),(X, Y)) \times P((X', Y'),(n(\kappa+1)-1, 0))}{P((n-1, n),(n(\kappa+1)-1, 0))} \\ 
& \;\;\; -  P((X', Y'),(X, Y)) \; \; .
\end{align*}
where $P((a,b),(c,d))$ is the number of south-east paths from $(a,b)$ to $(c,d)$ on the lecture hall graph ${\mathcal G}_t$, which in particular is defined as $0$ if $c < a$ or $d > b$. Indeed, one can easily compute that $K K^{-1} ((X, Y),(X', Y')) = \delta_{((X, Y),(X', Y'))}$, using the following two facts:
First, we have for any pair of vertices $(X, Y), (X',Y')$
\begin{align*}
P((X', Y'),(X, Y)) - \sum_{ \substack{(\tilde{X},\tilde{Y}) \text{ adjacent } (X, Y) : \\
 \; \tilde{X} < X \text{ or } \tilde{Y} > Y } } P((X', Y'),(\tilde{X},\tilde{Y})) 
 &= \delta_{((X, Y),(X', Y'))}
 \end{align*}
and second, at the boundary vertex $(X_0, Y_0) = (n (\kappa + 1)-1, 0)$, we have
$$K_L^{-1}((X_0, Y_0), (X', Y')) = 0 \;\;.$$
As computed before in Section \ref{combi}, if $a,b,c,d\in \mathbb Z$ then
$$
P((a,b),(c,d))= (b-d)^{c-a}{c \choose c-a}.
$$
Using Stirling's approximation, we can compute the asymptotic behavior of $K_L^{-1}((X, Y), (X', Y'))$.
We leave the details of the computation to the interested reader. \\

Now we derive a Burgers equation from \eqref{eqn:Feq} via a variable change. Suppose $F(x, y)$ satisfies \eqref{eqn:Feq}, and let $z = e^{F_y}, w = e^{-F_x}$. Then
   \begin{equation}   \label{eqn:mixed}
z_x/z + w_y/w = 0 
\end{equation}
by equality of mixed partials of $F$ and by equation \eqref{eqn:Feq}, we get
\begin{equation} \label{eqn:P}
\exp(- w x) = z \;\;.
\end{equation}

   Setting $u = - w x$, \eqref{eqn:mixed} becomes
  
  \begin{equation}\label{eq:34}
u u_x + u_y = 0 \;\;.
\end{equation}

This is a remarkable and useful fact because it means that the limiting behavior of lecture hall tableaux can be described by the complex inviscid Burgers equation \cite{KO07}.

To solve the Burgers equation for our two running examples, we first state a lemma which says we can solve this via the method of complex characteristics, see \cite{K}.
We omit its proof, which is a short computation. 
   
\begin{lemma}
\label{lem:burgers}
 Given an analytic function $Q(z)$ and a function $u(x, y)$, then at points $(x, y)$ such that
$$Q(x - y u) + u - 1 = 0$$
$u$ also satisfies the complex Burgers equation
$$u u_x + u_y = 0 \;\;.$$

\end{lemma}

What this means is that if we can choose $Q$ such that the resulting function $u$ satisfies the correct boundary conditions (which depend on the limiting profile of $\lambda$), then we have solved our boundary value problem.

We have found solutions to this equation that agree with the results of the tangent method in various cases where the partition is very simple. In these cases the rescaled graph will be inside of the rectangle $R = [0, 2] \times [0, 1]$.
Using
$Q(z) := a/(z + d) + b -  c z$,
where $a, d, b, c \in \mathbb{R}$ are free parameters,
gives us some simple solutions to the Burgers equation.

To set our boundary conditions, we use numerical evidence and conjectures about 
 \begin{align*}
  \lim_{y \rightarrow 0} \arg(u(x, y)) \\
    \lim_{y \rightarrow 1} \arg(u(x, y)).
 \end{align*}
These should be related to the derivatives of the height function of the model. See Conjecture \ref{hei} in Section \ref{conclu}.

When $\lambda=(n,n-1,\ldots ,2,1)$ and $t=n$ and $n\rightarrow\infty$,
we use 
$$Q(z) = 1 - 1/(4 z) + z \;\;.$$
Solving for $u$ with Lemma \ref{lem:burgers}, after a simple coordinate change we obtain 
\begin{equation}
u(x, y) =      \frac{(x-1) y + \sqrt{(x-2) x + y^2}}{y^2-1}  \;\;.
\end{equation}
The arctic curve is given exactly by the boundary of the region $(x-2) x + y^2 \leq 0$.
So we get
$$(x-1)^2 + y^2 = 1$$
which is a semi circle of radius $1$ centered at $(1, 0)$, as the arctic circle.
See Figure \ref{staircase}.

When $\lambda=(n,n,\ldots ,n,n)$ and $t=n$ and $n\rightarrow\infty$, we can set
$$Q(z) = 1 - \frac{1}{z} + z \;\;.$$
Solving for $u$, we get
\begin{align*}
u(x, y) = \frac{-1 + x - 2 (-1 + x) y - \sqrt{(-1 + x)^2 - 4 y + 4 y^2}}{2 (y - y^2)}
\end{align*}
It is a mechanical check that $u$ indeed satisfies the Burgers equation. Also, the arctic curve is given by the equation
 $$(x-1)^2 - 4 y + 4 y^2 = 0$$
and this ellipse agrees with the arctic curve we observe in our simulations. See Figure \ref{square}.

This can be generalized to $\lambda=((p-1)n,(p-1)n,\ldots ,(p-1)n,(p-1)n)$ and t$=n$
and $n\rightarrow\infty$.
For $p \ge 2$, we get
$$u(x, y) =  \frac{1 + x + p (-1 + y) - 2 x y + 
    \sqrt{(1 + x)^2 + 2 p (1 + x) (-1 + y) + p^2 (-1 + y)^2 - 4 x y}}{2 (y - y^2)} \;\;.$$

We shall stress again that, although these results are conjectural,
they are supported by the fact that they match with our simulations presented in Section \ref{simu} and the computations done with the tangent method in Section \ref{tangent}.

\section{Conclusion, open problems and future work}
\label{conclu}

In this paper, we compute arctic curves for bounded lecture hall tableaux thanks to two methods:
the tangent method \cite{CS} using paths and the ansatz of \cite{KRS} using dimers. Both of these methods
are not fully rigorous. Nevertheless we conjecture that we find the true arctic curves as 
we find the same curves for our two running examples.

It would be interesting to study rigorously the dimer model on this lattice made of hexagons and
octagons. On this lattice (or the lecture hall graph) we can define a height function $h_n(x,y)$ on the faces of the graph.
Given a configuration of $n$ paths starting at $(n-i,t-1/(n-i+1)$
and ending at $(\lambda_i+n-i,0)$, the height of a face is the number of paths to the southwest of the face.
We give an example of the height function of Figure \ref{height} for $\lambda=(2,2)$.

\begin{figure}
  \centering
\begin{tikzpicture}
\BLHLL{3}3
\draw [red,ultra thick](0,2)--(1,2)--(1,3/2)--(2,4/3)--(2,0);
\draw [red,ultra thick](1,5/2)--(2,7/3)--(2,2)--(3,2)--(3,0);
\node at (0.5,0.5){0};
\node at (0.5,1.5){0};
\node at (1.5,1.25){0};
\node at (1.5,.25){0};
\node at (1.5,1.25){0};
\node at (1.5,2.25){1};
\node at (1.5,1.75){1};
\node at (1.5,0.75){0};
\node at (2.5,1/6){\tiny 1};
\node at (2.5,3/6){\tiny 1};
\node at (2.5,5/6){\tiny 1};
\node at (2.5,7/6){\tiny 1};
\node at (2.5,9/6){\tiny 1};
\node at (2.5,11/6){\tiny 1};
\node at (2.5,13/6){\tiny 2};
\node at (2.5,15/6){\tiny 2};
\end{tikzpicture}
\caption{Height function for paths of ${\mathcal G}_2$}
\label{height}
\end{figure}
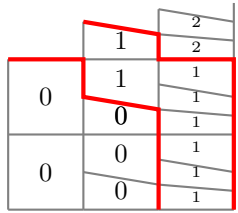

When $n\rightarrow \infty$ and $\lambda_i+n-i=n\alpha(i/n)$, 
let $h(x,y) = \lim h_n(x,y)/n$
be the limiting height function. 
Based on strong numerical evidence and on the structural similarities between this
dimer model and $\mathbb Z^2$
-periodic dimer models, we believe the following conjecture to be true.
\begin{conj}
The function $u(x,y)$ solution of the Burgers equation in Section \ref{dimer} satisfies
$$
Im(u) = \pi \frac{\partial h}{\partial y} 
$$
and
$$
arg(u) = \pi\left(\frac{\partial h}{\partial x}+1\right)
$$
for some branch of the arg.
\label{hei}
\end{conj}
When we impose that the hexagons and the octagons  have all the same area and shape, we get
a non-planar lattice made of fans of hexagons separated by line of octagons. This observation is due to N. Reshetikhin.
On Figure \ref{hexaocto}, we draw one fan of hexagons and the its line of octagons.

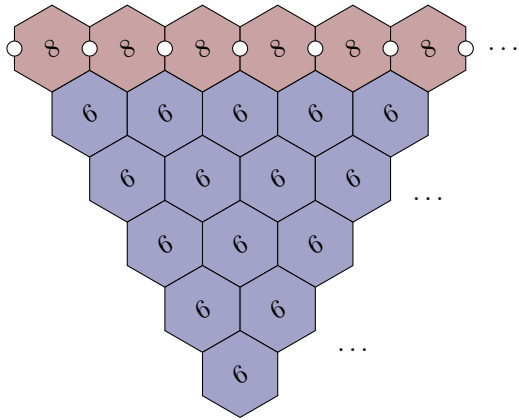
\begin{figure}
\begin{tikzpicture}
 \tikzset{hexa/.style= {shape=regular polygon,regular polygon sides=6,minimum size=1.15470053838cm, draw,inner sep=0,fill=lightgray!85!blue,rotate=30}}
 \tikzset{octagon/.style= {shape=regular polygon,regular polygon sides=6,minimum size=1.15470053838cm, draw,inner sep=0,fill=lightgray!85!red,rotate=30}}

\foreach \j in {1,...,5}{%
\pgfmathtruncatemacro\end{5- \j} 
  \foreach \i in {0,...,\end}{%
  \pgfmathsetmacro\x{\i + (\j * 0.5) + 4} 
    \pgfmathsetmacro\y{-\j * sin(60)} 
  \pgfmathtruncatemacro\k{\j}  
  \node[hexa] (h6\i;\k) at ({\x},{\y}) {6};}  }

\foreach \i in {0,...,5}{%
\pgfmathtruncatemacro\x{\i + 4} 
\pgfmathtruncatemacro\y{0}   
\node[octagon] (h5\i) at ({\x},{\y}) {8}; 
}

\node at (10,0){$\ldots$};
\node at (9,-2){$\ldots$};
\node at (8,-4){$\ldots$};

\foreach \i in {-1,...,5}{%
\pgfmathtruncatemacro\x{\i + 4} 
\pgfmathtruncatemacro\y{0}   
\draw [fill=white] (\x + 0.5,0.0) circle [radius=0.1];
}

\end{tikzpicture}

\caption{Another way of drawing a horizontal strip of the lecture hall lattice}
\label{hexaocto}
\end{figure}

We could also compute the asymptotic behavior of the lecture hall tableaux without the bounded condition. In this case we have 
an infinite number of tableaux. When $q<1$, we can compute the generating function of lecture hall tableaux
of shape $\lambda$ where each tableau gives the contribution $q^{|T|}$ where $|T|=\sum_{i,j} T_{i,j}$.
In \cite{CK}, Corteel and Kim computed the corresponding generating function which is~:
$$
\prod_{1\le i<j\le n} \frac{q^{\lambda_j+n-j}-q^{\lambda_i+n-i}}{q^{i-1}-q^{j-1}}
\prod_{i=1}^n \frac{(-q^{n-i+1})_{\lambda_i}}{(q^{2n-i+1})_{\lambda_i}}.
$$
with $(a)_k=\prod_{i=0}^{k-1}(1-aq^i)$.
The fact that this generating function  has a beautiful product formula makes us think that we could build the right algebraic tools to study the asymptotics of these unbounded tableaux. 
In \cite{DG1}, the authors consider the case of $q$-weighted lozenge tilings for which a single path will asymptotically travel along a geodesic (not necessarily a straight line). It would be interesting to consider the same for lecture hall tableaux.
It would also be really interesting to define a  ``lecture hall 
Schur process''. See for example \cite{OR03} for the classical case
and all the follow up papers.

Last but not least, the tangent method could be made rigorous in our case \cite{Aggarwal19}
at least for the computation of some parts of the curves (i.e. the ones that correspond to the trajectory of the first or the last path). This would require some detailed computations.


\begin{thebibliography}{99}
\bibitem{A}
A. Aggarwal.~Arctic Boundaries of the Ice Model on Three-Bundle Domains, To appear in Inventiones Mathematicae (2021). arXiv:1812.03847

\bibitem{Aggarwal19}
A.~Aggarwal, Private communication, March 2019.

\bibitem{BME1} 
M. Bousquet-M\'elou and K. Eriksson.
\newblock lecture hall Partitions.
\newblock {\em The Ramanujan Journal}, 1 no. 1 (1997):101-111.

\bibitem{BME2} 
M. Bousquet-M\'elou and K. Eriksson.
\newblock lecture hall Partitions II.
\newblock {\em The Ramanujan Journal}, 1 no. 2 (1997):165-186.


\bibitem{BME3} 
M. Bousquet-M\'elou and K. Eriksson.
\newblock  A refinement of the lecture hall Theorem, \newblock {\em J. Combin. Theory Ser. A.} 86 (1999) 63-84.

\bibitem{CEP96}
H. Cohn, N. Elkies, and J. Propp. Local statistics for random
domino tilings of the Aztec diamond. {\em Duke Math. J.}, 85(1):117-166, 1996, arXiv:math/0008243

\bibitem{CS}
F. Colomo and A. Sportiello, Arctic curves of the six-vertex model on generic domains: the tangent method, {\em J. Stat. Phys}, 164 6 (2016) 1488-1523.

\bibitem{CPS}
F. Colomo, A. G. Pronko and A. Sportiello 
Arctic curves of the free-fermion six-vertex model in an L-shaped domain 
{\em J. Stat. Phys.} 174 (2019).



\bibitem{CK2}
S.~Corteel and J.~S. Kim.
\newblock Enumeration of bounded lecture hall tableaux,
\newblock {\em Preprint} (2019), arXiv:1904.10602

\bibitem{CK}
S.~Corteel and J.~S. Kim.
\newblock Lecture hall tableaux
\newblock {\em Preprint} (2018), arXiv:1804.02489 .

\bibitem{DR}
B.~Debin, P.~Ruelle.
\newblock Tangent method for the arctic curve arising from freezing boundaries
\newblock {\em Preprint} (2018), arXiv:1810.04909


\bibitem{DG}
P.~Di Francesco, E.~Guitter.
\newblock Arctic curves for paths with arbitrary starting points: a tangent method approach
\newblock {\em Preprint} (2018), arXiv:1803.11463.

\bibitem{DG1}
P.~Di Francesco, E.~Guitter,
A tangent method derivation of the arctic curve for $q$-weighted paths with arbitrary starting points
\newblock {\em Preprint} (2018), arXiv:1810.07936.

\bibitem{DG2}
P.~Di Francesco, E.~Guitter,
The Arctic curve for Aztec rectangles with defects via the Tangent Method
\newblock {\em Preprint} (2019), arXiv:1902.06478
 
\bibitem{DL}
P. Di Francesco and M. F. Lapa, Arctic curves in path models from the tangent
method, {\em J. Phys. A: Math. Theor.} 51 (2018).

\bibitem{DFSG14}
P. Di Francesco and R. Soto-Garrido. Arctic curves of the octahedron
equation. J. Phys. A, 47(28):285204, 34, 2014, arXiv:1402.4493 

\bibitem{GV}
I. M. Gessel and X. G. Viennot, Binomial determinants, paths, and hook-length formulas, 
\newblock {\em Advances in Math.} 58 (1985), 300--321.

\bibitem{GesselViennot}
I. M. Gessel and X. Viennot, Determinants, paths, and plane partitions, preprint, 1989, available at
\texttt{http://www.cs.brandeis.edu/~ira.}

\bibitem{JPS98}
W. Jockusch, J. Propp, and P. Shor, Random domino tilings and the arctic circle theorem, arXiv:math/9801068 [math.CO] (1998).

\bibitem{Ka}
P. Kasteleyn, Graph theory and crystal physics, 1967 {\em Graph Theory and Theoretical Physics} pp. 43-110 Academic Press, London.

 
\bibitem{KRS}
D.~Keating, N.~Reshetikhin and A.~Sridhar, Conformal limit for Dimer models on the hexagonal lattice,
Zap. Nauchn. Sem. POMI  473 (2018), 174-193
\texttt{http://www.pdmi.ras.ru/znsl/2018/v473/abs174.html}

\bibitem{KS}
D.~Keating and A.~Sridhar, Random tilings with the GPU,
\newblock Journal of Mathematical Physics 59, 091420 (2018); https://doi.org/10.1063/1.5038732


\bibitem{K}
R.~Kenyon, 
\newblock Lectures on dimers, {\rm Statistical mechanics}, 
191--230, {IAS/Park City Math. Ser.}, 16, Amer. Math. Soc., Providence, RI, 2009.

\bibitem{KO06}
R.~Kenyon and A.~Okounkov, Planar dimers and Harnack curves, {\em Duke Math.} J 131 (2006), no. 3, 499-524.

\bibitem{KO07} R.~Kenyon and A.~Okounkov,  
Limit shapes and the complex Burgers equation, {\em Acta Math.} 199 (2007), no. 2, 263-302.


\bibitem{KP}
R. Kenyon and R. Pemantle,
Double-dimers, the Ising model and the hexahedron recurrence
{\em J. of Comb. Theory, Series A}, Vol 137 (2016) Pages 27-63.



\bibitem{OR03}
A. Okounkov and N. Reshetikhin. Correlation function of Schur process with
application to local geometry of a random 3-dimensional Young diagram. {\em J. Amer. Math. Soc.}, 16(3):581-603, 2003, arXiv:math/0107056

\bibitem{PS}
T. K. Petersen and D. Speyer, An arctic circle theorem for Groves, 
{\em J. of Comb. Theory, Series A} 111 (2005) 137-164.

\bibitem{PW}
J. Propp, and D. Wilson, 
\newblock Coupling from the past: a user's guide, Microsurveys in discrete probability (Princeton, NJ, 1997), 
{\em DIMACS Ser. Discrete Math. Theoret. Comput. Sci.}, 41, Providence, R.I.: American Mathematical Society, 
pp. 181-192.

\bibitem{LHPSavage}
C.~D. Savage.
\newblock The mathematics of lecture hall partitions.
\newblock {\em J. Combin. Theory Ser. A}, 144:443-475, 2016.

\bibitem{stanley}
 RP. Stanley. Enumerative Combinatorics, Volume 2. Combinatorics. Cambridge studies in advanced mathematics. (1999).

\end{thebibliography}
\end{document}